\documentclass[11pt]{article}

\usepackage[margin=1in]{geometry}

\usepackage{amssymb}
\usepackage{bm}
\usepackage{graphicx}
\usepackage[centertags]{amsmath}
\usepackage{amsfonts}
\usepackage{amsthm}
\usepackage{graphicx}
\usepackage{slashed}
\usepackage{float}

\usepackage[usenames,dvipsnames,svgnames,table]{xcolor}
\usepackage[colorlinks=true]{hyperref}
\hypersetup{linkcolor=BrickRed, urlcolor=green, citecolor=blue, linktoc=page}

\numberwithin{equation}{section}
\graphicspath{/Users/yannis/Documents/}

\newtheorem{thm}{Theorem}

\newtheorem{lem}[thm]{Lemma}

\newtheorem{prop}[thm]{Proposition}

\theoremstyle{definition}
\newtheorem{rem}{Remark}

\newcommand{\ee}{\varepsilon}
\newcommand{\ph}{\varphi}

\newcommand{\si}{\Sigma}
\newcommand{\sis}{\widetilde{\Sigma}}
\newcommand{\rrr}{\mathcal{R}}

\newcommand{\mi}{\mathring{\mu}}

\newcommand{\nnn}{\mathcal{N}}
\newcommand{\so}{\bar{S}}
\newcommand{\aaa}{\alpha}

\newcommand{\cala}{\mathcal{A}}
\newcommand{\calc}{\mathcal{C}}

\newcommand{\meg}{\geqslant}
\newcommand{\mik}{\leqslant}

\begin{document}
\title{Asymptotic blow-up for a class of semilinear wave equations on extremal Reissner--Nordstr\"{o}m spacetimes}

\author{Y. Angelopoulos, S. Aretakis, and D. Gajic}

\date{February 15, 2018}

\maketitle

\begin{abstract}

We prove small data global existence for a class of semilinear wave equations satisfying the null condition on extremal Reissner--Nordstr\"om black hole backgrounds with nonlinear terms that degenerate at the event horizon. We impose no symmetry assumptions.  The study of such equations is motivated by their covariance properties under the Couch--Torrence conformal isometry. We show decay, non-decay and asymptotic blow-up results analogous to those in the linear case.

\end{abstract}
\tableofcontents

\section{Introduction}
We consider equations of the following form:
\begin{equation}\label{nwg}
\left\{\begin{aligned}
       \Box_g \psi = \sqrt{D} \cdot A(\psi ) g^{\alpha \beta} \cdot  \partial_{\alpha}  \psi  \cdot \partial_{\beta} \psi ,\\
       \psi |_{\widetilde{\Sigma}_0 } = \ee f , \quad n_{\widetilde{\Sigma}_0} \psi |_{\widetilde{\Sigma}_0}  = \ee g, \
       \end{aligned} \right.
\end{equation}
where $g = -D dt^2 + D^{-1} dr^2 + r^2 \gamma_{\mathbb{S}^2}$ with $D = \left( \frac{r-M}{r} \right)^2$ is the metric of the extremal Reissner--Nordstr\"{o}m spacetime and $\Box_g$ is the d'Alembertian operator with respect to the metric $g$. Furthermore, $\widetilde{\si}_{0}$ is a spacelike hypersurface in the domain of outer communications of the spacetimes that crosses the event horizon. We assume that $A$ is a function that depends on both $\psi$ and the spacetime metric and it is bounded along with its derivatives, i.e.
$$ |A^{(k)} | \leq a_k \mbox{  for all $k \in \mathbb{N}$.} $$
These equations are typical nonlinear wave equations satisfying the null condition, but they have an additional $\sqrt{D}$ weight that degenerates at the event horizon of the extremal Reissner--Nordstr\"{o}m black hole spacetimes. They arise naturally when studying the transformation of nonlinear wave equations satisfying null condition (introduced by Klainerman in \cite{sergiunull}) in the far-away region under the Couch--Torrence conformal isometry, which was introduced in \cite{couch}. See Section \ref{motivation} for more details.

We study the problem of global well-posedness for small and smooth data for equations of the form \eqref{nwg}. The present paper sets the stage for addressing the more challenging problem of proving  small data global well-posedness for classical nonlinear wave equations satisfying the null condition \emph{without} the additional $\sqrt{D}$ weight. The latter problem is studied in our upcoming \cite{rnnonlin3}. Our work is motivated by the stability problem for extremal black holes. 

\section{Linear and nonlinear waves on extremal Reissner--Nordstr\"{o}m spacetimes}
\label{sec:linnonlin}
The Reissner--Nordstr\"{o}m family of metrics is the unique 2-parameter family of spherically symmetric and asymptotically flat solutions to the Einstein--Maxwell equations. The two parameters are $(e,M)$ where $-M\leq e \leq M$ denotes the electromagnetic charge and $M>0$ denotes the mass of the spacetime. The metric is given by: $$g = -Ddt^2 + D^{-1} dr^2 + r^2 \gamma_{\mathbb{S}^2},$$ with $t\in \mathbb{R}$, $r\in (M,\infty)$ and $\gamma_{\mathbb{S}^2}$ the standard metric on the round unit 2-sphere, where $$D = 1-\frac{2M}{r} + \frac{e^2}{r^2}.$$ In ingoing Eddington--Finkelstein coordinates (that we will use throughout this article) the metric can be expressed as: $$g = -Ddv^2 + 2dvdr + r^2 \gamma_{\mathbb{S}^2}$$ for $v = t+r^{*}$ where $\frac{dr^{*}}{dr} = \frac{1}{D}$ and we can take $r \in [M,\infty)$. We will denote from now on $T = \partial_v$ and $Y = \partial_r$ in the ingoing Eddington--Finkelstein coordinate system.

Extremal Reissner--Nordstr\"{o}m solutions make up a 1-parameter subfamily of Reissner--Nordstr\"{o}m  satisfying $|e| = M$. In the extremal case we can compute that on the event horizon $\mathcal{H}^{+} = \{ r=M \}$ where $T$ is Killing, null and tangential (hence there exists some smooth function $\kappa$ called \textit{surface gravity} such that $\left. \nabla_T T \right|_{\mathcal{H}^{+}} = \left. \kappa T \right|_{\mathcal{H}^{+}}$) we have that $\left. \nabla_T T \right|_{\mathcal{H}^{+}} = \left. \frac{D'}{2} T \right|_{\mathcal{H}^{+}} = 0$, as $D = \left( \frac{r-M}{r} \right)^2$. So in the $|e| = M$ case, the event horizon is a \textit{degenerate horizon} (as its surface gravity vanishes), which is a geometric defining property of extremal black hole spacetimes. 

We note here that on extremal Reissner--Nordstr\"{o}m spacetimes we will work with the spacelike-null foliation $\Sigma_{\tau}$ that covers the the domain of outer communications $\{ r \meg M \}$, which is defined as $\Sigma_{\tau} = S_{\tau} \cup N_{\tau}$ for $S_{\tau} = \{ t^{*} = \tau \} \cap \{ r \mik R_0 \}$ for some large enough $R_0 > 2M$ (note that the photon sphere is situated at $r = 2M$) and for $t^{*} = v-r$, and $N_{\tau} = \{ u = u_{\tau} , v \meg v_{\tau} \} \cap \{ r \meg R_0 \}$ for $u = t-r^{*}$ and $v = t+r^{*}$ (where $(u,v)$ are Eddington--Finkelstein \textit{double null} coordinates, that we will use occasionally), and with the spacelike foliation of the domain of outer communications $\widetilde{\Sigma}_{\tau} = \{ t^{*} = \tau \}$ in the case of the local theory.

In the works \cite{A1} and \cite{A2} of the second author, several instability results were proved for the linear wave equation on extremal Reissner--Nordstr\"{o}m spacetimes. Specifically, it was shown that there are conservation laws on the event horizon  which imply non-decay results for first-order derivatives and blow-up (asymptotically along $\mathcal{H}^{+}$) for  higher derivatives. A general theory of conservation laws on null hypersurfaces was presented in \cite{aretakis4} and \cite{aretakisglue}.

Additionally, \cite{aretakis2013} illustrates how these instabilities can be used to prove \textit{finite time blow-up} on $\mathcal{H}^{+}$ for a certain class of semilinear wave equations (it should be noted that in the subextremal case, the nonlinear wave equations of the type studied in \cite{aretakis2013} do \emph{not} exhibit any type of blow-up for small data). 

In \cite{rnnonlin1}, the first author showed that if the nonlinearity satisfies the so called \textit{null condition} and the initial data are spherically symmetric, then  there is no finite time blow-up for small data, yet on $\mathcal{H}^{+}$ there exists a quantity that is \emph{almost} conserved (so in particular there is no decay for the derivative $Y\psi$, which is transversal to the horizon). As a consequence, it was shown that all the \emph{asymptotic} blow-up phenomena from the linear case do occur. A major difficulty in all nonlinear applications arises from the fact that close to the horizon one is forced to work with energies with degenerate weights, since as it was shown in \cite{trapping}, the integrated energy without any degenerate weights on the horizon generically blows up.

The precise influence of conservation laws along the event horizon on the late-time asymptotics of fixed frequency solutions to the linear wave equation in extremal Reissner--Nordstr\"om was investigated numerically in \cite{lmrtprice}; see also the related heuristics in \cite{ori2013,sela}.

Decay estimates and conservations laws have also been obtained in extremal Kerr for axisymmetric solutions to the linear wave equation \cite{aretakis3,aretakis4,aretakis2012}. See also \cite{hj2012,murata2012} for generalisations of the above conservation laws along the event horizon to higher-spin equations and higher-dimensional spacetimes, respectively.

Precise late-time asymptotics and decay estimates for the linear wave equation in the domain of outer communication of extremal Reissner--Nordstr\"om and Kerr play an important role in the study of the regularity and boundedness properties for the linear wave equation in the corresponding black hole interior regions, as was shown by the third author in \cite{gajic,gajic2}. Regularity properties in extremal black hole interiors were first investigated numerically in \cite{harvey2013} in the context of the nonlinear spherically symmetric Einstein--Maxwell-scalar field system of equations.

\section{Motivation for our model}
\label{motivation}

By considering equations of the form \eqref{nwg}, we are incorporating a degeneracy at the event horizon into our nonlinearity. As we will see below, the leading-order terms in the nonlinearity remain \emph{covariant} under a conformal isometry that is special to extremal Reissner--Nordstr\"om: the \textit{Couch--Torrence conformal isometry}. This discrete conformal isometry, denoted by $\Phi$, is given in $(t,r,\theta , \varphi)$ coordinates by the transformation
$$ (t,r,\theta , \phi) \xrightarrow{\Phi} \left( t , r' = M+\frac{M^2}{r-M}, \theta, \phi \right) , $$
and in ingoing Eddington--Finkelstein coordinates
$$ (v,r, \theta , \phi) \xrightarrow{\Phi} \left( u =v , r' = M +\frac{M^2}{r-M}, \theta, \phi \right) , $$ 
so in particular, $\mathcal{H}^{+}$ is mapped to $\mathcal{I}^{+}$ and vice versa (where $\mathcal{I}^{+} = \{ (u, \infty ) \}$ denotes future null infinity). This conformal isometry was introduced in \cite{couch} and further explored in the context of the linear wave equation in \cite{lmrtprice, bizon2012,ori2013,sela}.

In \cite{bizon2016} the authors used this transformation in order to study the Yang--Mills equation within spherical symmetry on an extremal Reissner--Nordstr\"om background. Note that the Yang--Mills equation is conformally invariant. Here we pose a more general problem: consider a nonlinear wave equation that satisfies the classical null condition at the infinity of an asymptotically flat spacetime, what is the equation that we get after applying the transformation, close to the horizon of an extremal Reissner--Nordstr\"{o}m spacetime?

A simple computation shows that the equation 
$$ \Box_g \psi = g^{\alpha \beta} \cdot  \partial_{\alpha}  \psi  \cdot \partial_{\beta} \psi , $$
close to $\mathcal{I}^{+}$ (i.e.\ for $r>R$, with $R>0$ sufficiently large), roughly corresponds to the equation
$$ \Box_g \widetilde{\psi} = \sqrt{D} \cdot g^{\alpha \beta} \cdot  \partial_{\alpha}  \widetilde{\psi}  \cdot \partial_{\beta} \widetilde{\psi} , $$
close to $\mathcal{H}^{+}$, and vice versa. More specifically after setting 
$$ \Omega = \frac{M}{r'-M} = \frac{r-M}{M} , $$
for $r'$ the radial variable near infinity, and $r$ the radial variable close to $\mathcal{H}^{+}$, due to the fact the extremal Reissner--Nordstr\"{o}m spacetime has vanishing scalar curvature, we have that for a wave $\psi_{O}$ near null infinity, we obtain another wave near the horizon 
$\psi_{I} \doteq \Omega^{-1} \psi_O$ 
through $\Phi$ by the equation  
\begin{equation}\label{eq:conf_omega}
\Box_{g_I} \psi_I = \Omega^{-3} \Box_{g_O} \psi_O , 
\end{equation}
where $g_I$, $g_O$ is the metric close to the horizon and infinity respectively, which are related by
$$ g_I = \Omega^2 g_O . $$
Note that in \textit{ingoing} Eddington--Finkelstein coordinates, $g_I$ has the standard form $g_I = -D dv^2 + 2dvdr + r^2 \gamma_{\mathbb{S}^2}$ and $g_O$ has the form $g_O = -D du^2 - 2dudr' + (r')^2 \gamma_{\mathbb{S}^2}$ where $(u,r,\theta,\phi)$ are called \textit{outgoing} Eddington--Finkelstein coordinates. If $\psi_O$ solves a nonlinear wave equation where the nonlinearity satisfies the standard null condition, we then have from \eqref{eq:conf_omega} that $\psi_I$ satisfies the following equation
\begin{equation}\label{eq:conf_hor}
\Box_{g_I} \psi_I = \Omega \cdot g^{\mu \nu}_I \cdot \partial_{\mu} \psi_I \cdot \partial_{\nu} \psi_I + \frac{2}{M}\cdot  T\psi_I \cdot \psi_I + \frac{2D}{M} \cdot Y\psi_I \cdot \psi_I - \frac{\sqrt{D}}{M} \psi_I^2 .  
\end{equation}
Since we have that $\Omega = \frac{r-M}{M}$, and $\sqrt{D} = \frac{r-M}{r}$, we can easily see that up to some extra terms (that present no additional difficulties for small data global well-posedness questions), equation \eqref{eq:conf_hor} is of the same form as \eqref{nwg}, demonstrating therefore the covariance of \eqref{nwg}.
 
This explains on a heuristic level why the analysis of nonlinear wave equations satisfying the null condition \textit{without} any weights on the horizon cannot be performed using solely classical methods as the following difficulties have to be dealt with:
\begin{itemize}
\item The null condition has to be exploited not only at infinity but also at the horizon, where we also need a combination of energy estimates and of $L^1$ estimates provided by the method of characteristics (as it was done in \cite{rnnonlin1}).
\item Without the assumption of spherical symmetry (as in the upcoming \cite{rnnonlin3}) we need extra decay at the linear level (see the upcoming \cite{improvedrn}). Moreover we need to separate the wave into its spherically symmetric part and its non-spherically symmetric part, with a different analysis for each case.
\item Some of the above additional difficulties of the analysis cannot be removed if we increase the degree of the nonlinearity (see \cite{rnnonlinthesis} where nonlinear wave equations were studied with nonlinearities of quartic degree satisfying the null condition and without any weights on the horizon).
\end{itemize}

\begin{rem}
It is also interesting to perform a computation similar to the derivation of \eqref{eq:conf_hor} but in the opposite direction. Let us consider a wave $\psi_I$ that satisfies a nonlinear wave equation with the classical null condition close to the horizon. Since $\Phi = \Phi^{-1}$, we now set 
$$ \Omega = \frac{M}{r-M} = \frac{r' - M}{M} , $$
and we relate again the metric $g_I$ close to the horizon to the metric $g_O$ close to infinity by
$$ g_O = \Omega^2 g_I .$$
The new wave $\psi_O = \Omega^{-1} \psi_I$ now can be seen to satisfy the equation
\begin{equation}\label{eq:conf_inf}
\begin{split}
\Box_{g_O} \psi_O = \Omega^{-3} & \Box_{g_I} \psi_I \Rightarrow  \\ &\Box_{g_O} \psi_O = \Omega \cdot g^{\mu \nu}_O \cdot \partial_{\mu} \psi_O \cdot \partial_{\nu} \psi_O - \frac{2}{M} \partial_u \psi_O \cdot \psi_O - \frac{2D}{M} \partial_{r'} \psi_O \cdot \psi_O - \frac{r' - M}{(r')^2} \psi_O^2 .
\end{split}
\end{equation}
When we performed the opposite transformation and arrived at equation \eqref{eq:conf_hor}, we were interested in identifying the extra factor of $r-M$ in front of the nonlinearity. In equation \eqref{eq:conf_inf} we are interested in identifying the decay with respect to powers of $r'$ of the nonlinearity. With the classical null condition we expect the following behaviour in $r'$ (using the estimates provided by the linear theory)
$$ g^{\mu \nu}_O \cdot \partial_{\mu} \psi_O  \cdot \partial_{\nu} \psi_O \sim \frac{1}{(r')^3} . $$
This can be seen more clearly after writing the null form in terms of $\varphi_O \doteq r' \psi_O$ instead of $\psi_O$ where we have that
\begin{equation*}
\begin{split}
g^{\mu \nu}_O \partial_{\mu} \psi_O  \cdot \partial_{\nu} \psi_O = \frac{1}{(r')^2} D \cdot (\partial_{r'} \varphi_O )^2 - & \frac{2}{(r')^2}  \partial_u \varphi_O \cdot \partial_{r'} \varphi_O +  \frac{1}{(r')^2} | \slashed{\nabla} \varphi_O |^2  \\ & + \frac{2D}{(r')^2} \partial_{r'} \varphi_O \cdot \varphi_O + \frac{2}{(r')^3} \partial_u \varphi_O \cdot \varphi_O - \frac{D}{(r')^4} \varphi_O^4 . 
\end{split}
\end{equation*}
By the following expected behaviour in $r'$ towards infinity: $\varphi_O \sim 1$, $\partial_u \varphi_O \sim 1$ and $\partial_{r'} \varphi_O \sim \frac{1}{(r')^2}$, we can see that the term that imposes the $\frac{1}{(r')^3}$ decay for the classical null condition is the term
$$  \frac{2}{(r')^3} \partial_u \varphi_O \cdot \varphi_O . $$
The rest are of the order $\frac{1}{(r')^4}$ and better. 

Using the asymptotics of the classical null condition described above, since $\Omega = \frac{r'-M}{M} \approx r'$  close to infinity, the nonlinearity of \eqref{eq:conf_inf} \emph{seems} at first glance to be of order $\frac{1}{(r')^2}$. 

However, after rewriting the nonlinearity in terms of $\varphi_O$ we notice that
\begin{equation*}
\begin{split}
 \Omega \cdot g^{\mu \nu}_O \cdot & \partial_{\mu} \psi_O  \cdot \partial_{\nu} \psi_O - \frac{2}{M} \partial_u \psi_O \cdot \psi_O -  \frac{2D}{M} \partial_{r'} \psi_O \cdot \psi_O - \frac{r' - M}{(r')^2} \psi_O^2 \\ = & \frac{\Omega}{(r')^2} D \cdot (\partial_{r'} \varphi_O )^2 -  \frac{2\Omega}{(r')^2}  \partial_u \varphi_O \cdot \partial_{r'} \varphi_O +  \frac{\Omega}{(r')^2} | \slashed{\nabla} \varphi_O |^2  \\ & + \frac{2 D}{(r')^2} \left( \Omega + \frac{1}{M} \right) \partial_{r'} \varphi_O \cdot \varphi_O + \left( \frac{2\Omega}{(r')^3} - \frac{2}{M (r')^2} \right) \partial_u \varphi_O \cdot \varphi_O - \left( \frac{\Omega \cdot D}{(r')^4} + \frac{2D}{M (r' )^3} + \frac{r'-M}{(r')^4} \right) \varphi_O^4 .
 \end{split}
 \end{equation*}
Again the term that seems to impose worst decay of order $\frac{1}{(r')^2}$ is the one involving the term $\partial_u \phi_O \cdot \phi_O$, but now we notice that
$$ \left( \frac{2\Omega}{(r')^3} - \frac{2}{M (r')^2} \right) \partial_u \varphi_O \cdot \varphi_O = - \frac{2}{(r')^3} \partial_u \varphi_O \cdot \varphi_O , $$
so we can conclude that the nonlinearity of equation \eqref{eq:conf_inf} should in fact be of order $\frac{1}{(r')^3}$, which is consistent with the observation above in the case where the nonlinearity satisfies the standard null condition. Still though, there is the difference that in the nonlinearity of \eqref{eq:conf_inf} many terms that in the classical null form are of order $\frac{1}{(r')^4}$ are now of order $\frac{1}{(r')^3}$ and this poses several difficulties in dealing with questions of small data global well-posedness with equations of the form \eqref{eq:conf_inf} on a spherically symmetric and asymptotically flat spacetime. Such equations, along with more general nonlinear wave equations where the nonlinearity satisfies the null condition but carries also growing weights in the radial variable will be investigated in future work.
\end{rem}

\section{The main theorem}
The main results of this article are summarized in the following theorem.
\begin{thm}[\textbf{Main theorem}]
There exists an $\epsilon > 0$ such that if $0 \mik \epsilon' \mik \epsilon$, then for all compactly supported data $(\epsilon' f , \epsilon' g)$ on a spacelike hypersurface $\widetilde{\Sigma}_0$ that crosses the event horizon and that ends at spacelike infinity that is of size $\epsilon'$ in the sense that
$$ \| ( \epsilon' f , \epsilon' g \|_{H^s (\widetilde{\Sigma}_0 ) \times H^{s-1} ( \widetilde{\Sigma}_0 ) } \mik \sqrt{E_0} \epsilon' , $$
for some $s > 20$, $s \in \mathbb{N}$, and for $E_0$ defined by
\begin{equation*}
\begin{split}
 E_0 = \sum_{k \mik 5, l \mik 4} &\left[ \int_{S_0} \left(  (\Omega^k T^{l+1} \psi )^2 +  (\Omega^k T^l Y\psi )^2 + | \Omega^k T^l \slashed{\nabla} \psi |^2 \right)  d\mi_{g_S} \right]  \\ & + \sum_{k \mik 5, l \mik 2} \left[ \int_{S_0} \left( (\Omega^k T^{l+1} Y \psi )^2 +  (\Omega^k T^l Y^2\psi )^2 + | \Omega^k T^l \slashed{\nabla} Y\psi |^2 \right) d\mi_{g_S} \right]  \\ & + \sum_{k \mik 5, l \mik 4} \left[ \int_{N_0} \left( (\Omega^k T^{l+1} Y \psi )^2 +  (\Omega^k T^l Y^2\psi )^2 + | \Omega^k T^l \slashed{\nabla} \psi |^2 \right) d\mi_{g_N} \right] ,
 \end{split}
\end{equation*} 
there exists a unique global solution $\psi$ of \eqref{nwg} that is in $H^s (\widetilde{\si}_{\tau} ) \times H^{s-1} (\widetilde{\si}_{\tau} )$ for any $\tau \meg 0$ in the domain of outer communications up to and including the event horizon of an extremal Reissner--Nordstr\"{o}m spacetime. 

The solution has the following asymptotic behaviour.

1) (\textbf{Decay estimates}). For any $\tau > 0$ we have that
$$ \| \psi \|_{L^{\infty} (\Sigma_{\tau} )} \lesssim \frac{\sqrt{E_0} \epsilon'}{(1+\tau )^{1/2}} ,  \quad \quad  \| \slashed{\nabla} \psi \|_{L^{\infty} (\Sigma_{\tau} )} \lesssim \frac{\sqrt{E_0} \epsilon'}{(1+\tau )^{1/2}}, $$ $$ \| T\psi \|_{L^{\infty} (\Sigma_{\tau} )} \lesssim \frac{\sqrt{E_0} \epsilon'}{(1+\tau )^{1/2}}, \quad \quad \| \sqrt{D} Y\psi \|_{L^{\infty} (\Sigma_{\tau} )} \lesssim \frac{\sqrt{E_0} \epsilon'}{(1+\tau )^{1/4}} . $$

2) (\textbf{Asymptotic instabilities on $\mathcal{H}^{+}$}). On the event horizon $\mathcal{H}^{+}$ the quantity
$$ \int_{\mathbb{S}^2} \left( Y \psi (v,r=M) + \frac{1}{M} \psi (v, r=M) \right) \, d\omega $$
is conserved, while if initially we have that
$$ \int_{\mathbb{S}^2} \psi (0 , r=M) \, d\omega > 0 \mbox{  and  }  \int_{\mathbb{S}^2} Y\psi (0 , r=M) \, d\omega > 0 , $$
then we observe the following asymptotic blow-up phenomenon
$$ \int_{\mathbb{S}^2} Y^k \psi (v , r=M) \, d\omega \xrightarrow{v \rightarrow \infty} \infty \mbox{  for all $k \meg 2$ along $\mathcal{H}^{+}$.} $$
\end{thm}
Several comments are in order.

1. As in \cite{rnnonlin1}, we only use the hierarchy of energy estimates provided by \cite{A1} and \cite{A2} (which in our situation gives us only the weak pointwise decay of order $\frac{1}{\tau^{1/2}}$ close to the horizon which presents an additional difficulty in our analysis) combined with the $r^p$-weighted method at infinity as in \cite{shiwu}, with the minor refinement that we are able to close all estimates for $p=2$ without any $\alpha$ loss, after closing an extra bootstrap assumption.

2. We commute 5 times with angular derivatives (something that is standard), and four times in total with respect to $T$ and $Y$ derivatives (where the number of commutations of $Y$ is always restricted to 1). The four commutations with the $T$ derivatives take place in order to take care of the trapping effect at the photon sphere. As it is standard in such situations (see \cite{luknullcondition}), the top-order energy is allowed to grow, although in our case this is even more complicated as there are different decay rates within the terms of the form $T^k \psi$ due to the use of the $T-P-N$ hierarchy of degenerate energy estimates (see the next section for details on this hierarchy). The consistency of the all the energy decay statements depends crucially on Step 1. and our ability to close certain bootstrap estimates for $p=2$.

3. The most problematic commutation is with $Y$. The upshot in our situation is that the continuation criterion involves the global boundedness only of $\sqrt{D} \cdot Y\psi$ and not of $Y\psi$ itself (for which we have no estimates). We are able to deal with $\sqrt{D} \cdot Y\psi$ only through the use of some new degenerate weighted energy estimates, and without relying on the method of characteristics (as it was done \cite{rnnonlin1} and \cite{rnnonlinthesis}, and as it will be done in the upcoming \cite{rnnonlin3}). 

\begin{rem}
As our continuation criterion involves only $\sqrt{D} \cdot Y\psi$, we believe that our method carries over to the study of axisymmetric solutions of nonlinear wave equation of the form \eqref{nwg} on extremal Kerr backgrounds. Note that on an extremal Kerr spacetime the boundedness of $Y\psi$ itself is not even known for linear waves.
\end{rem}

\section{Energy estimates}
In this section as well as in the one that follows we prove certain $L^2$ estimates for solutions of the following inhomogeneous wave equation
\begin{equation}\label{nwf}
\Box_g \psi = F .
\end{equation}
For future reference we define the following regions for any given $\tau_1$, $\tau_2$ with $\tau_1 < \tau_2$
\begin{equation}\label{mcala}
\mathcal{A}_{\tau_1}^{\tau_2} = \rrr (\tau_1 , \tau_2 ) \cap \{ M \mik r \mik r_0 < 2M \} , 
\end{equation}
for some fixed $r_0$, and
\begin{equation}\label{mcalc}
\calc_{\tau_1}^{\tau_2} = \rrr (\tau_1 , \tau_2 ) \cap \{ M < 2M - \delta \mik r \mik 2M + \delta \} ,
\end{equation}
for some $\delta > 0$ which is sufficiently small. The volume forms that we will use are the ones that are naturally induced from $g$, we just note that between two $N_{\tau_1}$ and $N_{\tau_2}$ hypersurfaces we have that the volume form is $D r^2 d\omega dv du$ for $v$ between $v_{R_0}$ and infinity, and $u$ between $u_{\tau_1}$ and $u_{\tau_2}$.

We present here the $T-P-N$ hierarchy for $\psi$. Recall that $T$ is a Killing vector field, while the $P$ and $N$ vector fields were first introduced in \cite{A2} and \cite{A1}, and they have the form
$$ N \approx T - Y \mbox{  close to $\mathcal{H}^{+}$ and  } N \sim T \mbox{  away from $\mathcal{H}^{+}$,} $$
$$ P \sim T - \sqrt{D} \cdot Y \mbox{  close to $\mathcal{H}^{+}$ and  } P \approx T \mbox{  away from $\mathcal{H}^{+}$.} $$
For the energy momentum tensor of the wave equation
$$ T_{\mu \nu } [\psi ] = \partial_{\mu} \psi \cdot \partial_{\nu} \psi - \frac{1}{2} g_{\mu \nu} \cdot \partial^c \psi \cdot \partial_c \psi $$
we define the energy current for a vector field $V$ by 
$$ J_{\mu}^V [\psi ] \doteq T_{\mu \nu} [\psi ] \cdot V^{\nu} , $$
and we can compute its divergence as follows
$$ Div ( J^V [\psi ] ) \doteq K^V [\psi ] + \mathcal{E}^V [\psi ] \mbox{  where  } K^V [\psi ] = T_{\mu \nu} [\psi ] \cdot ( \nabla^{\mu} V )^{\nu} \mbox{  and  } \mathcal{E}^V [\psi ] = \Box_g \psi \cdot V\psi .$$
We now recall here that for $T$ we have $J^T_{\mu} [\psi ] n^{\mu}_{\si} \sim (T\psi )^2 + D (Y\psi )^2 + |\slashed{\nabla} \psi |^2$ and $K^T [\psi ] = 0$ (as $T$ is Killing), while for $P$ and $N$ we have the following estimates close to the horizon
\begin{equation}\label{defpn}
J^P_{\mu} [\psi ] n^{\mu}_{\si} \sim (T\psi )^2 + \sqrt{D} (Y\psi )^2 + |\slashed{\nabla} \psi |^2 , \quad J^N_{\mu} [\psi ] n^{\mu}_{\si} \sim (T\psi )^2 + (Y\psi )^2 + |\slashed{\nabla} \psi |^2 , 
\end{equation}
$$ K^P [\psi ] \sim J^T_{\mu} [\psi ] n^{\mu}_{\si} \sim (T\psi )^2 + D (Y\psi )^2 + |\slashed{\nabla} \psi |^2 , \quad K^N [\psi ] \sim J^P_{\mu} [\psi ] n^{\mu}_{\si} \sim (T\psi )^2 + \sqrt{D} (Y\psi )^2 + |\slashed{\nabla} \psi |^2 . $$

Let us also recall the standard Hardy inequality (see \cite{A1} for a proof) in extremal Reissner--Nordstr\"{o}m spacetimes given by
\begin{equation}\label{hardy}
\int_{\Sigma_{\tau}} \frac{1}{r^2} f^2 d\mi_{g_{\Sigma}} \lesssim \int_{\Sigma_{\tau}} J^T_{\mu} [f] n^{\mu} d \mi_{g_{\Sigma}} ,
\end{equation} 
which holds for any function $f$ that decays fast enough towards infinity.
We state here energy estimates that come from these 3 vector fields combined with integrated energy estimates for a bounded region in $r$ that passes over the photon sphere.

First we state the Morawetz estimate that implies a boundedness estimate for the degenerate energy.
\begin{prop}[\textbf{Morawetz and degenerate energy estimate}]\label{ent}
Let $\psi$ be a solution of \eqref{nwf}. Then for all $\tau_1$, $\tau_2$ with $\tau_1 < \tau_2$ and any $\eta > 0$ we have that
\begin{equation}\label{en1}
\int_{\si_{\tau_2}} J^T_{\mu} [\Omega^k T^l \psi ] n^{\mu} d\mi_{g_{\si}} + \int_{\rrr (\tau_1 , \tau_2 )} \chi_{(\calc_{\tau_1}^{\tau_2} )^c} \left( \dfrac{(\Omega^k T^l T\psi )^2 }{r^{1+\eta}} + \dfrac{D^{5/2} \cdot (\Omega^k T^l Y\psi )^2}{r^{1+\eta}} + \dfrac{\sqrt{D}|\Omega^k T^l \slashed{\nabla} \psi |^2}{r} \right)  d\mi_{g_{\rrr}} \lesssim_{R_0}
\end{equation}
$$ \lesssim_{R_0} \int_{\si_{\tau_1}} J^T_{\mu} [\Omega^k T^l \psi ] n^{\mu} d\mi_{g_{\si}} +  \int_{\tau_1}^{\tau_2} \int_{S_{\tau'}} |\Omega^k T^l F|^2 d\mi_{g_{\so}} + \int_{\tau_1}^{\tau_2} \int_{N_{\tau'}} r^{1+\eta} |\Omega^k T^l F|^2 d\mi_{g_{\nnn}} + $$ $$ + \int_{\mathcal{C}_{\tau_1}^{\tau_2}} |\Omega^k T^{l+1} F|^2 d\mi_{g_{\mathcal{C}}} + \sup_{\tau' \in [\tau_1, \tau_2]} \int_{\si_{\tau'} \cap \calc_{\tau_1}^{\tau_2}} |\Omega^k T^l F|^2 d\mi_{g_{\si_{\calc}}}, $$
for any $k \in \mathbb{N}$, $l \in \mathbb{N}$, where $\calc$ was defined in \eqref{mcalc}, and where $\chi_{(\calc_{\tau_1}^{\tau_2} )^{c}}$ is a smooth function that is equal to 1 on the complement of $\calc$ and 0 around the photon sphere.
\end{prop}
\begin{proof}
We will show the proof for the case of $k=0$ and $l=0$ (nothing changes in all the other cases due to the fact that both $T$ and $\Omega$ commute with $\Box_g$). 

First from Theorem 1 of \cite{A1} we have that
\begin{equation}\label{amor1}
 \int_{\rrr (\tau_1 , \tau_2 ) \setminus \calc_{\tau_1}^{\tau_2}} \sqrt{D} \left( \dfrac{(T\psi )^2}{r^{1+\eta}} + \dfrac{D^2 (Y\psi )^2}{r^{1+\eta}} + \dfrac{|\slashed{\nabla} \psi |^2}{r} \right) d\mi_{g_{\rrr}} \lesssim \int_{\si_{\tau_1}} J^T_{\mu} [\psi ] n^{\mu} d\mi_{g_{\si}} + \sup_{\tau' \in [\tau_1, \tau_2]} \int_{\si_{\tau'} \cap \calc_{\tau_1}^{\tau_2}} | F|^2 d\mi_{g_{\si_{\calc}}} + 
\end{equation} 
  $$ + \int_{\calc_{\tau_1}^{\tau_2}} |TF|^2 d\mi_{g_{\calc}} + \int_{\rrr (\tau_1 , \tau_2 ) \setminus \calc_{\tau_1}^{\tau_2}} |F \cdot T\psi | d\mi_{g_{\rrr}} + \int_{\rrr (\tau_1 , \tau_2 ) } |F  |^2 d\mi_{g_{\rrr}}  , $$
for any $\eta > 0$, where we separated the inhomogeneous term $\int_{\rrr (\tau_1 , \tau_2 ) } F \cdot T\psi  \,  d\mi_{g_{\rrr}}$ into two terms $\int_{\rrr (\tau_1 , \tau_2 ) \setminus \calc_{\tau_1}^{\tau_2} } F \cdot T\psi   \, d\mi_{g_{\rrr}}$ and $\int_{\rrr (\tau_1 , \tau_2 ) \cap \calc_{\tau_1}^{\tau_2} } F \cdot T\psi \, d\mi_{g_{\rrr}}$ and for the second term we integrated by parts with respect to $T$ (see also the related computation in \cite{luknullcondition}).

Moreover from Proposition 9.2.5 of \cite{A1} we have that
\begin{equation}\label{amor2}
 \int_{\rrr (\tau_1 , \tau_2 ) \setminus \calc_{\tau_1}^{\tau_2}}  \left( \dfrac{(\partial_{r^{*}} \psi )^2}{r^{1+\eta}} + \dfrac{\sqrt{D}|\slashed{\nabla} \psi |^2}{r} \right) d\mi_{g_{\rrr}} \lesssim \int_{\si_{\tau_1}} J^T_{\mu} [\psi ] n^{\mu} d\mi_{g_{\si}} + \sup_{\tau' \in [\tau_1, \tau_2]} \int_{\si_{\tau'} \cap \calc_{\tau_1}^{\tau_2}} | F|^2 d\mi_{g_{\si_{\calc}}} + 
\end{equation} 
  $$ + \int_{\calc_{\tau_1}^{\tau_2}} |TF|^2 d\mi_{g_{\calc}} + \int_{\rrr (\tau_1 , \tau_2 ) \setminus \calc_{\tau_1}^{\tau_2}} |F \cdot T\psi | d\mi_{g_{\rrr}} + \int_{\rrr (\tau_1 , \tau_2 ) } |F  |^2 d\mi_{g_{\rrr}}  , $$
for any $\eta > 0$, where we recall that $\partial_{r^{*}} = T + DY$, and where we used \eqref{amor1}. 

We can use Stokes' Theorem to the energy current $J^X [\psi]$ for $X = f(r^{*}) \partial_{r^{*}}$ where $f$ is a smooth function that is 1 close to $\mathcal{H}^{+}$ (and away from the photon sphere of course -- consider the spacetime area $\cala$ defined by \eqref{mcala}) and 0 elsewhere. From the formulas given in Section 9.1.1 of \cite{A1} we have that
\begin{equation}\label{amor3}
 \int_{\cala_{\tau_1}^{\tau_2}} \left( (T\psi )^2 - (\partial_{r^{*}} \psi )^2 - \sqrt{D}|\slashed{\nabla} \psi |^2 \right) d\mi_{g_{\rrr}} \lesssim \int_{\si_{\tau_1}} J^T_{\mu} [\psi ] n^{\mu} d\mi_{g_{\si}} + \sup_{\tau' \in [\tau_1, \tau_2]} \int_{\si_{\tau'} \cap \calc_{\tau_1}^{\tau_2}} | F|^2 d\mi_{g_{\si_{\calc}}} + 
\end{equation} 
  $$ + \int_{\calc_{\tau_1}^{\tau_2}} |TF|^2 d\mi_{g_{\calc}} + \int_{\rrr (\tau_1 , \tau_2 ) \setminus \calc_{\tau_1}^{\tau_2}} |F \cdot T\psi | d\mi_{g_{\rrr}} + \int_{\rrr (\tau_1 , \tau_2 ) } |F  |^2 d\mi_{g_{\rrr}}  , $$
for any $\eta > 0$. 

Combining \eqref{amor2} and \eqref{amor3} we get that:
\begin{equation}\label{amor4}
 \int_{\cala_{\tau_1}^{\tau_2}} (T\psi )^2 d\mi_{g_{\rrr}} \lesssim \int_{\si_{\tau_1}} J^T_{\mu} [\psi ] n^{\mu} d\mi_{g_{\si}} + \sup_{\tau' \in [\tau_1, \tau_2]} \int_{\si_{\tau'} \cap \calc_{\tau_1}^{\tau_2}} | F|^2 d\mi_{g_{\si_{\calc}}} + 
\end{equation} 
  $$ + \int_{\calc_{\tau_1}^{\tau_2}} |TF|^2 d\mi_{g_{\calc}} + \int_{\rrr (\tau_1 , \tau_2 ) \setminus \calc_{\tau_1}^{\tau_2}} |F \cdot T\psi | d\mi_{g_{\rrr}} , $$
for any $\eta > 0$.

Finally combining \eqref{amor4} with \eqref{amor1} we have that
 \begin{equation}\label{amor5}
 \int_{\rrr (\tau_1 , \tau_2 ) \setminus \calc_{\tau_1}^{\tau_2}} \left( \dfrac{(T\psi )^2}{r^{1+\eta}} + \dfrac{D^{5/2} (Y\psi )^2}{r^{1+\eta}} + \dfrac{\sqrt{D} |\slashed{\nabla} \psi |^2}{r} \right) d\mi_{g_{\rrr}} \lesssim \int_{\si_{\tau_1}} J^T_{\mu} [\psi ] n^{\mu} d\mi_{g_{\si}} + \sup_{\tau' \in [\tau_1, \tau_2]} \int_{\si_{\tau'} \cap \calc_{\tau_1}^{\tau_2}} | F|^2 d\mi_{g_{\si_{\calc}}} + 
\end{equation} 
  $$ + \int_{\calc_{\tau_1}^{\tau_2}} |TF|^2 d\mi_{g_{\calc}} + \int_{\rrr (\tau_1 , \tau_2 ) \setminus \calc_{\tau_1}^{\tau_2}} |F \cdot T\psi | d\mi_{g_{\rrr}} + \int_{\rrr (\tau_1 , \tau_2 ) } |F  |^2 d\mi_{g_{\rrr}}  , $$
for any $\eta > 0$. 

The desired result now follows since we can treat the term $\int_{\rrr (\tau_1 , \tau_2 ) \setminus \calc_{\tau_1}^{\tau_2}} |F \cdot T\psi | d\mi_{g_{\rrr}}$ through the Cauchy-Schwarz inequality.

\end{proof}

Now using the previous estimates and the definitions of the $P$ and $N$ vector fields we state some basic estimates.
\begin{prop}[\textbf{Basic energy estimate for $\Omega^k T^l \psi$, $k \in \mathbb{N}$, $l \in \mathbb{N}$}]\label{3deg}
Let $\psi$ be a solution of \eqref{nwf}. Then for any $\tau_1$, $\tau_2$ with $\tau_1 < \tau_2$ and any $\eta > 0$ we have that:
\begin{equation}\label{3et1}
\int_{\si_{\tau_2}} J^T_{\mu} [\Omega^k T^l \psi ] n^{\mu} d\mi_{g_{\si}}  \lesssim_{R_0} \int_{\si_{\tau_1}} J^T_{\mu} [\Omega^k T^l \psi ] n^{\mu} d\mi_{g_{\si}} + \left( \int_{\tau_1}^{\tau_2} \left( \int_{\si_{\tau'}} |\Omega^k T^l F|^2 d\mi_{g_{\si}} \right)^{1/2} d\tau' \right)^2 ,
\end{equation}
\begin{equation}\label{3ep1}
\int_{\si_{\tau_2}} J^P_{\mu} [\Omega^k T^l \psi ] n^{\mu} d\mi_{g_{\si}}  \lesssim_{R_0} \int_{\si_{\tau_1}} J^P_{\mu} [\Omega^k T^l \psi ] n^{\mu} d\mi_{g_{\si}} +\left( \int_{\tau_1}^{\tau_2} \left( \int_{\si_{\tau'}} |\Omega^k T^l F|^2 d\mi_{g_{\si}} \right)^{1/2} d\tau' \right)^2 ,
\end{equation}
and
\begin{equation}\label{3en1}
\int_{\si_{\tau_2}} J^N_{\mu} [\Omega^k T^l \psi ] n^{\mu} d\mi_{g_{\si}}  \lesssim_{R_0} \int_{\si_{\tau_1}} J^N_{\mu} [\Omega^k T^l \psi ] n^{\mu} d\mi_{g_{\si}} + \left( \int_{\tau_1}^{\tau_2} \left( \int_{\si_{\tau'}} |\Omega^k T^l F|^2 d\mi_{g_{\si}} \right)^{1/2} d\tau' \right)^2 ,
\end{equation}
for any $k \in \mathbb{N}$, $l \in \mathbb{N}$.
\end{prop}
The proof of estimates \eqref{3et1}, \eqref{3ep1} and\eqref{3en1} are standard. A useful variation of \eqref{3et1} is the following estimate
\begin{equation}\label{3et1h}
\begin{split}
\int_{\si_{\tau_2}} J^T_{\mu} [\Omega^k T^l \psi ] n^{\mu} d\mi_{g_{\si}}  & \lesssim_{R_0} \int_{\si_{\tau_1}} J^T_{\mu} [\Omega^k T^l \psi ] n^{\mu} d\mi_{g_{\si}} + \\ & + \int_{\cala_{\tau_1}^{\tau_2} } |\Omega^k T^l F|^2 d\mi_{g_{\cala}} + \left( \int_{\tau_1}^{\tau_2} \left( \int_{\si_{\tau'} \cap ( \cala_{\tau_1}^{\tau_2} )^c } |\Omega^k T^l F|^2 d\mi_{g_{\si}} \right)^{1/2} d\tau' \right)^2 .
\end{split}
\end{equation}

Now we consider an integrated energy estimate using the same degenerate energy as in Proposition \ref{ent} which now includes a neighbourhood of the photon sphere as well. This introduces a loss of a $T$-derivative on the linear level (a loss which is mandatory as shown in \cite{janpaper}). Its proof is similar to the one of Proposition \ref{ent}.
\begin{prop}[\textbf{Morawetz estimates without degeneracy on the photon sphere}]\label{ent1}
Let $\psi$ be a solution of \eqref{nwf}. Then for all $\tau_1$, $\tau_2$ with $\tau_1 < \tau_2$ and any $\eta > 0$ we have that
\begin{equation}\label{en11}
\int_{\rrr (\tau_1 , \tau_2 )} \left( \dfrac{(\Omega^k  T^{l+1}\psi )^2 }{r^{1+\eta}} + \dfrac{D^{5/2} \cdot (\Omega^k  T^l Y\psi )^2}{r^{1+\eta}} + \dfrac{\sqrt{D}|\Omega^k T^l \slashed{\nabla} \psi |^2}{r} \right)  d\mi_{g_{\rrr}} \lesssim_{R_0}  
\end{equation}
$$  \lesssim_{R_0} \sum_{m = l}^{l+1} \left( \int_{\si_{\tau_1}} J^T_{\mu} [\Omega^k T^m \psi ] n^{\mu} d\mi_{g_{\si}} + \int_{\tau_1}^{\tau_2} \int_{S_{\tau'}} |\Omega^k T^m F|^2 d\mi_{g_{\so}} + \int_{\tau_1}^{\tau_2} \int_{N_{\tau'}} r^{1+\eta} |\Omega^k T^m F|^2 d\mi_{g_{\nnn}} \right) + $$ $$ + \int_{\mathcal{C}_{\tau_1}^{\tau_2}} |\Omega^k T^{l+2} F|^2 d\mi_{g_{\mathcal{C}}} + \sum_{m=l}^{l+1} \sup_{\tau' \in [\tau_1, \tau_2]} \int_{\si_{\tau'} \cap \calc_{\tau_1}^{\tau_2}} |\Omega^k T^{m}   F|^2 d\mi_{g_{\si_{\calc}}}  , $$
and
\begin{equation}\label{en12}
\int_{\rrr (\tau_1 , \tau_2 )} \left( \dfrac{(\Omega^k  T^{l+1}\psi )^2 }{r^{1+\eta}} + \dfrac{D^{5/2} \cdot (\Omega^k  T^l Y\psi )^2}{r^{1+\eta}} + \dfrac{\sqrt{D}|\Omega^k T^l \slashed{\nabla}  \psi |^2}{r} \right)  d\mi_{g_{\rrr}} \lesssim_{R_0}  
\end{equation}
$$  \lesssim_{R_0} \sum_{m = l}^{l+1} \int_{\si_{\tau_1}} J^T_{\mu} [\Omega^k T^m \psi ] n^{\mu} d\mi_{g_{\si}} + \int_{\tau_1}^{\tau_2} \int_{S_{\tau'}} |\Omega^k T^l F|^2 d\mi_{g_{\so}} + \int_{\tau_1}^{\tau_2} \int_{N_{\tau'}} r^{1+\eta} |\Omega^k T^l F|^2 d\mi_{g_{\nnn}}  +  $$ $$  + \left( \int_{\tau_1}^{\tau_2} \left( \int_{\si_{\tau'}} |\Omega^k T^{l+1}  F|^2 d\mi_{g_{\si}} \right)^{1/2} d\tau' \right)^2  , $$

for any $k \in \mathbb{N}$, and where $\calc$ was defined in \eqref{mcalc}.
\end{prop}

We state now the energy estimate for the $P$-flux, that gives us also an integrated estimate for the $T$-flux for a region close to the horizon.
\begin{prop}[\textbf{$P$-energy estimate}]\label{enp}
Let $\psi$ be a solution of \eqref{nwf}. Then for all $\tau_1$, $\tau_2$ with $\tau_1 < \tau_2$ and any $\eta > 0$ we have that
\begin{equation}\label{en2}
\int_{\si_{\tau_2}} J^P_{\mu} [\Omega^k T^l \psi ] n^{\mu} d\mi_{g_{\si}} + \int_{\tau_1}^{\tau_2} \int_{\si_{\tau'} \cap \cala_{\tau_1}^{\tau_2}} J^T_{\mu} [\Omega^k T^l \psi] n^{\mu}_{\si} d\mi_{g_{\cala}} \lesssim_{R_0} \int_{\si_{\tau_1}} J^P_{\mu} [\Omega^k T^l \psi ] n^{\mu} d\mi_{g_{\si}} +  
\end{equation}
$$  + \int_{\tau_1}^{\tau_2} \int_{S_{\tau'}} |\Omega^k T^l F|^2 d\mi_{g_{\so}} + \int_{\tau_1}^{\tau_2} \int_{N_{\tau'}} r^{1+\eta} |\Omega^k T^l F|^2 d\mi_{g_{\nnn}} + $$ $$ + \int_{\calc_{\tau_1}^{\tau_2}} |\Omega^k T^{l+1} F|^2 d\mi_{g_{\mathcal{C}}} + \sup_{\tau' \in [\tau_1, \tau_2]} \int_{\si_{\tau'} \cap \calc_{\tau_1}^{\tau_2}} |\Omega^k T^l F|^2 d\mi_{g_{\si_{\calc}}} , $$
for any $k \in \mathbb{N}$, $l\in \mathbb{N}$, and where $\cala$, $\calc$ were defined in \eqref{mcala} and \eqref{mcalc} respectively.
\end{prop}
This proof follows along the lines of the related Proposition in \cite{rnnonlin1}. The loss of the derivative around the photon sphere comes from the use of Proposition \ref{ent} in order to treat the contribution of the bulk term $K^P [\psi]$ away from the horizon.

The next Proposition gives us an integrated estimate (for a region close to the horizon) for the $P$-flux and an energy estimate for the $\dot{H}^1$ norm without any degeneracy on $\mathcal{H}^{+}$.
\begin{prop}[\textbf{Non-degenerate energy estimate}]\label{enn}
Let $\psi$ be a solution of \eqref{nwf}. Then for all $\tau_1$, $\tau_2$ with $\tau_1 < \tau_2$ and any $\eta > 0$ we have that
\begin{equation}\label{en3}
\int_{\si_{\tau_2}} J^N_{\mu} [\Omega^k T^l \psi ] n^{\mu} d\mi_{g_{\si}} + \int_{\tau_1}^{\tau_2} \int_{\si_{\tau'} \cap \cala_{\tau_1}^{\tau_2}} J^P_{\mu} [\Omega^k T^l \psi] n^{\mu}_{\si} d\mi_{g_{\cala}}  \lesssim_{R_0}  
\end{equation}
$$ \lesssim_{R_0} \int_{\si_{\tau_1}} J^N_{\mu} [\Omega^k T^l \psi ] n^{\mu} d\mi_{g_{\si}} + \left( \int_{\tau_1}^{\tau_2} \left( \int_{S_{\tau'}} |\Omega^k T^l F|^2 d\mi_{g_S} \right)^{1/2} d\tau' \right)^2 + \int_{\tau_1}^{\tau_2} \int_{S_{\tau'}} |\Omega^k T^l F|^2 d\mi_{g_{\so}} + $$ $$ + \int_{\tau_1}^{\tau_2} \int_{N_{\tau'}} r^{1+\eta} |\Omega^k T^l F|^2 d\mi_{g_{\nnn}} + \int_{\calc_{\tau_1}^{\tau_2}} |\Omega^k T^{l+1} F|^2 d\mi_{g_{\mathcal{C}}} + \sup_{\tau' \in [\tau_1, \tau_2]} \int_{\si_{\tau'} \cap \calc_{\tau_1}^{\tau_2}} |\Omega^k T^l F|^2 d\mi_{g_{\si_{\calc}}}, $$
for any $k \in \mathbb{N}$, $l\in \mathbb{N}$, and where $\cala$, $\calc$ were defined in \eqref{mcala} and \eqref{mcalc} respectively.
\end{prop}
This proof as well follows along the lines that were sketched in the related Proposition in the spherically symmetric case in \cite{rnnonlin1}. As before, the loss of the derivative around the photon sphere comes from the use of Proposition \ref{ent} in order to treat the contribution of the bulk term $K^N [\psi]$ away from the horizon.

We state now a hierarchy of integrated energy estimates close to infinity. They are the so-called $r^p$-weighted energy inequalities that were first introduced by Dafermos and Rodnianski in \cite{newmethod} (see \cite{A2} for a proof in the linear case), and in order to derive them we also use in our case Proposition \ref{ent} and estimate \eqref{3et1}. 

\begin{equation}\label{rwins}
\int_{N_{\tau_2}} r^p \dfrac{( \partial_v \Omega^k T^l \ph )^2}{r^2} d\mi_{g_N} + \int_{\tau_1}^{\tau_2} \int_{N_{\tau}} p r^{p-1} \dfrac{(\partial_v \Omega^k T^l \ph)^2}{r^2}  d\mi_{g_{\nnn}} +
\end{equation}
$$
+ \int_{\tau_1}^{\tau_2} \int_{N_{\tau}} \frac{r^{p-1}}{4} \left( (2-p)D - r D' \right) |\slashed{\nabla} \Omega^k T^l \psi |^2 d\mi_{g_{\nnn}} \lesssim_{p , R_0} \int_{N_{\tau_1}} r^p \dfrac{( \partial_v \Omega^k T^l \ph )^2}{r^2} d\mi_{g_N} + \int_{\Sigma_{\tau_1}} J^T_{\mu} [\Omega^k T^l \psi] n^{\mu} d\mi_{g_{\si}}  + $$ $$ + \int_{\tau_1}^{\tau_2} \int_{S_{\tau'}} |\Omega^k T^l F|^2 d\mi_{g_{\so}} +  \int_{\tau_1}^{\tau_2} \int_{N_{\tau'}} r^{1+\eta} |\Omega^k T^l F|^2 d\mi_{g_{\nnn}} + \int_{\tau_1}^{\tau_2} \int_{N_{\tau'}} r^{p+1} |\Omega^k T^l F|^2 d \mi_{g_{\nnn}} +
$$ $$ + \int_{\mathcal{C}_{\tau_1}^{\tau_2}} |\Omega^k T^{l+1}F|^2 d\mi_{g_{\mathcal{C}}} + \sup_{\tau' \in [\tau_1, \tau_2]} \int_{\si_{\tau'} \cap \calc_{\tau_1}^{\tau_2}} |\Omega^k T^l F|^2 d\mi_{g_{\si_{\calc}}}, $$
for any $k \in \mathbb{N}$, $l \in \mathbb{N}$, for any $\tau_1$, $\tau_2$ with $\tau_1 < \tau_2$ and any $\eta  > 0$. 

With the above estimate \eqref{rwins} and the results of Propositions \ref{ent}, \ref{3deg}, \ref{enp} and \ref{enn} we are able to prove several integrated estimates whose support is global in space. First we state an integrated estimate for the $T$-flux of $\psi$ and its $T$ and angular derivatives.
\begin{prop}[\textbf{Integrated estimate for the degenerate energy of $\Omega^k T^l\psi$}]\label{intt}
Let $\psi$ be a solution of \eqref{nwf}. Then for all $\tau_1$, $\tau_2$ with $\tau_1 < \tau_2$, $\ph = r\psi$, and any $\eta > 0$ we have that
\begin{equation}\label{ien1}
\int_{\tau_1}^{\tau_2} \int_{\si_{\tau'}} J^T_{\mu} [\Omega^k T^l \psi ] n^{\mu} d\mi_{g_{\rrr}} \lesssim_{R_0} \int_{\si_{\tau_1}} J^P_{\mu} [\Omega^k T^l \psi ] n^{\mu} d\mi_{g_{\si}} + \int_{\si_{\tau_1}} J^T_{\mu} [\Omega^k T^{l+1} \psi ] n^{\mu} d\mi_{g_{\si}} +  
\end{equation}
$$ + \int_{N_{\tau_1}} \dfrac{(\partial_v \Omega^k T^l \ph )^2}{r} d\mi_{g_{\si}} + \sum_{m = l}^{l+1} \int_{\tau_1}^{\tau_2} \int_{S_{\tau'}} |\Omega^k T^m F|^2 d\mi_{g_{\so}} + \int_{\tau_1}^{\tau_2} \int_{N_{\tau'}} r^{2} |\Omega^k T^m F|^2 d\mi_{g_{\nnn}} +  $$ $$   +   \int_{\mathcal{C}_{\tau_1}^{\tau_2}} |\Omega^k  T^{l+2} F|^2 d\mi_{g_{\mathcal{C}}}  + \sum_{m=l}^{l+1}\sup_{\tau' \in [\tau_1, \tau_2]} \int_{\si_{\tau'} \cap \calc_{\tau_1}^{\tau_2}} |\Omega^k  T^m F|^2 d\mi_{g_{\si_{\calc}}}, $$
for any $k \in \mathbb{N}$, $l \in \mathbb{N}$, and where $\calc$ was defined in \eqref{mcalc}.
\end{prop}
The proof follows by using estimate \eqref{rwins}, the Morawetz estimate of Proposition \ref{ent} and the integrated estimate (close to the horizon) provided by Proposition \ref{enp}.

Next we state an integrated estimate for the $P$-flux of $\psi$ and its angular derivatives.

\begin{prop}[\textbf{Integrated estimate for the $P$-energy of $\Omega^k T^l \psi$}]\label{intp}
Let $\psi$ be a solution of \eqref{nwf}. Then for all $\tau_1$, $\tau_2$ with $\tau_1 < \tau_2$, $\ph = r\psi$, and any $\eta > 0$ we have that
\begin{equation}\label{ien2}
\int_{\tau_1}^{\tau_2} \int_{\si_{\tau'}} J^P_{\mu} [\Omega^k  T^l \psi ] n^{\mu} d\mi_{g_{\rrr}} \lesssim_{R_0} \int_{\si_{\tau_1}} J^{N}_{\mu} [\Omega^k  T^l \psi ] n^{\mu} d\mi_{g_{\si}} + \int_{N_{\tau_1}} \dfrac{(\partial_v \Omega^k T^l \ph )^2}{r} d\mi_{g_{\si}} + \int_{\si_{\tau_1}} J^T_{\mu} [\Omega^k T^{l+1} \psi ] n^{\mu} d\mi_{g_{\si}} +
\end{equation}
$$ + \left( \int_{\tau_1}^{\tau_2} \left( \int_{S_{\tau'}} |\Omega^k  T^l F|^2 d\mi_{g_S} \right)^{1/2} d\tau' \right)^2 + \int_{\tau_1}^{\tau_2} \int_{S_{\tau'}} |\Omega^k T^l F|^2 d\mi_{g_{\so}} + \int_{\tau_1}^{\tau_2} \int_{N_{\tau'}} r^{2} |\Omega^k T^l F|^2 d\mi_{g_{\nnn}} + $$ $$ + \int_{\mathcal{C}_{\tau_1}^{\tau_2}} |\Omega^k T^{l+2} F|^2 d\mi_{g_{\mathcal{C}}} + \sum_{m=l}^{l+1}\sup_{\tau' \in [\tau_1, \tau_2]} \int_{\si_{\tau'} \cap \calc_{\tau_1}^{\tau_2}} |\Omega^k T^m F|^2 d\mi_{g_{\si_{\calc}}}, $$
for any $k \in \mathbb{N}$, $l \in \mathbb{N}$, and where $\calc$ was defined in \eqref{mcalc}.
\end{prop}

Similarly this proof follows by using estimate \eqref{rwins}, the Morawetz estimate of Proposition \ref{ent} and the integrated estimate provided by Proposition \ref{enn}.

Below we now state alternative versions of the previous two Propositions. The difference here is that we use estimate \eqref{en12} in order to deal with the trapping effect of the photon sphere, hence losing one less derivative at the expense of dealing with a worse inhomogeneous term.

\begin{prop}[\textbf{Integrated estimate for the degenerate energy of $\Omega^k T^l \psi$}]\label{intt1}
Let $\psi$ be a solution of \eqref{nwf}. Then for all $\tau_1$, $\tau_2$ with $\tau_1 < \tau_2$, $\ph = r\psi$, and any $\eta > 0$ we have that
\begin{equation}\label{ien11}
\int_{\tau_1}^{\tau_2} \int_{\si_{\tau'}} J^T_{\mu} [\Omega^k T^l \psi ] n^{\mu} d\mi_{g_{\rrr}} \lesssim_{R_0} \int_{\si_{\tau_1}} J^P_{\mu} [\Omega^k  T^l \psi ] n^{\mu} d\mi_{g_{\si}} + \int_{\si_{\tau_1}} J^T_{\mu} [\Omega^k T^{l+1} \psi ] n^{\mu} d\mi_{g_{\si}} +  
\end{equation}
$$ + \int_{N_{\tau_1}} \dfrac{(\partial_v \Omega^k T^l \ph )^2}{r} d\mi_{g_{\si}} +  \int_{\tau_1}^{\tau_2} \int_{S_{\tau'}} |\Omega^k T^l F|^2 d\mi_{g_{\so}}  + \int_{\tau_1}^{\tau_2} \int_{N_{\tau'}} r^2 |\Omega^k T^l F|^2 d\mi_{g_{\nnn}} + $$ $$ + \int_{\mathcal{C}_{\tau_1}^{\tau_2}} |\Omega^k  T^{l+1} F|^2 d\mi_{g_{\mathcal{C}}}  + \sup_{\tau' \in [\tau_1, \tau_2]} \int_{\si_{\tau'} \cap \calc_{\tau_1}^{\tau_2}} |\Omega^k  T^l F|^2 d\mi_{g_{\si_{\calc}}} + \left( \int_{\tau_1}^{\tau_2} \left( \int_{\si_{\tau'}} |\Omega^k T^{l+1} F |^2 d\mi_{g_{\si}} \right)^{1/2} d\tau' \right)^2 , $$
for any $k \in \mathbb{N}$, $l \in \mathbb{N}$, and where $\calc$ was defined in \eqref{mcalc}.
\end{prop}

\begin{prop}[\textbf{Integrated estimate for the $P$--energy of $\Omega^k T^l \psi$}]\label{intp1}
Let $\psi$ be a solution of \eqref{nwf}. Then for all $\tau_1$, $\tau_2$ with $\tau_1 < \tau_2$, $\ph = r\psi$, and any $\eta > 0$ we have that
\begin{equation}\label{ien21}
\int_{\tau_1}^{\tau_2} \int_{\si_{\tau'}} J^P_{\mu} [\Omega^k T^l \psi ] n^{\mu} d\mi_{g_{\rrr}} \lesssim_{R_0} \int_{\si_{\tau_1}} J^{N}_{\mu} [\Omega^k T^l \psi ] n^{\mu} d\mi_{g_{\si}} + \int_{N_{\tau_1}} \dfrac{(\partial_v \Omega^k T^l \ph )^2}{r} d\mi_{g_{\si}} + 
\end{equation}
$$ + \int_{\si_{\tau_1}} J^T_{\mu} [\Omega^k T^{l+1} \psi ] n^{\mu} d\mi_{g_{\si}} + \left( \int_{\tau_1}^{\tau_2} \left( \int_{S_{\tau'}} |\Omega^k  T^l F|^2 d\mi_{g_S} \right)^{1/2} d\tau' \right)^2 + \int_{\tau_1}^{\tau_2} \int_{S_{\tau'}} |\Omega^k  T^l F|^2 d\mi_{g_{\so}}  + $$ $$ + \int_{\tau_1}^{\tau_2} \int_{N_{\tau'}} r^2 |\Omega^k T^l F|^2 d\mi_{g_{\nnn}}  + \int_{\mathcal{C}_{\tau_1}^{\tau_2}} |\Omega^k T^{l+1} F|^2 d\mi_{g_{\mathcal{C}}} + $$ $$ + \sup_{\tau' \in [\tau_1, \tau_2]} \int_{\si_{\tau'} \cap \calc_{\tau_1}^{\tau_2}} |\Omega^k  T^l F|^2 d\mi_{g_{\si_{\calc}}} + \left( \int_{\tau_1}^{\tau_2} \left( \int_{\si_{\tau'}} |\Omega^k T^{l+1} F |^2 d\mi_{g_{\si}} \right)^{1/2} d\tau' \right)^2 , $$
for any $k \in \mathbb{N}$, $l \in \mathbb{N}$, and where $\calc$ was defined in \eqref{mcalc}.
\end{prop}

Finally we state a basic inequality after commuting the equation $\Box_g \psi = F$ not just with angular derivatives and $T$ derivatives, but also once with $Y$. This is the main estimate that will allow us later on to bound $\sqrt{D} \cdot Y\psi$ for the nonlinear problem.

\begin{prop}[\textbf{Energy estimate for $T^l Y \psi$ in spherical symmetry}]\label{d32ty}
Let $\psi$ be a \textbf{spherically symmetric} linear wave
\begin{equation}\label{eq:waveF}
 \Box_g \psi = F . 
 \end{equation}
Let $l \in \mathbb{N}$,  then we have for all $\tau_1 < \tau_2$, $\phi = r\psi$ and any $\eta > 0$ that
\begin{equation}\label{d32tye1}
\begin{split}
\int_{\Sigma_{\tau_2}} \left[ (T^{l+1} Y\psi )^2 + D ( T^l Y^2 \psi )^2 \right] & +  \int_{\mathcal{A}_{\tau_1}^{\tau_2}} \left[ (TY\psi)^2 + D^{3/2} (Y^2 \psi )^2 \right] \\  \lesssim & \int_{\Sigma_{\tau_1}} \left[ (T^{l+1} Y\psi )^2 + D ( T^l Y^2 \psi )^2 \right] \\  & + \int_{\Sigma_{\tau_2}} J^N [ T^l \psi ] + \int_{\Sigma_{\tau_1}} J^T [T^{l+1} \psi ] \\ & + \int_{\mathcal{A}_{\tau_1}^{\tau_2}}  | T^l YF|^2 + \int_{\mathcal{C}_{\tau_1}^{\tau_2}} | T^{l+1} YF|^2 \\ & + \sum_{m=l}^{l+1} \Bigg( \int_{\tau_1}^{\tau_2} \int_{S_{\tau'}} |T^m F |^2 + \int_{\mathcal{C}_{\tau_1}^{\tau_2}} T^{m+1} F |^2 \\ & + \sup_{\tau' \in [\tau_1 , \tau_2]} \int_{\Sigma_{\tau'} \cap \mathcal{C}} | T^m F |^2 + \int_{\tau_1}^{\tau_2} \int_{\mathcal{N}_{\tau'}} r^{1+\eta} | T^m F|^2 \Bigg) \\ & + \left( \int_{\tau_1}^{\tau_2} \left( \int_{S_{\tau'}} | T^l F |^2 d\mi_{g_S} \right)^{1/2} d\tau' \right)^2 .  
\end{split}
\end{equation}
\end{prop}
\begin{proof}
We will prove the above estimate for the case of $l=0$ since $T$ commutes with \eqref{eq:waveF}. We have that
\begin{equation}\label{eq:ycomm}
\Box_g ( Y\psi ) = D' (Y^2 \psi) + \frac{2}{r^2} (T\psi) - R' (Y\psi) + \frac{2}{r} ( \slashed{\Delta} \psi) + YF ,
\end{equation}
and by choosing the vector field
\begin{equation}\label{vf:dp}
L = f^v (r) T + f^r (r) Y ,
\end{equation}
where
$$ f^v >  1 , \quad \partial_r f^v = \frac{1}{\sigma}, \quad f^r = - D , \quad \partial_r f^r = -D' \mbox{  in $\mathcal{A}_{\tau_1}^{\tau_2}$,} $$
where $ \sigma > 0 $ is chosen to be small enough. On the other hand we have that
$$ f^v = f^r = 0 \mbox{  in $\left( \cup_{\tau' \in [\tau_1 , \tau_2]} \Sigma_{\tau'} \right) \cap\{ r \geq r_1 > r_0 \}$, } $$
and $f^v$, $f^r$ are smoothly extended in the region $\left( \cup_{\tau' \in [\tau_1 , \tau_2]} \Sigma_{\tau'} \right) \cap \{ r_0 \leq r \leq r_1 \}$.

We have that
\begin{equation*}
\begin{split}
K^{L} [Y \psi ] + \mathcal{E}^{L} [ Y\psi ] & =  H_1 (TY\psi )^2 + H_2 (Y^2 \psi )^2 + H_3 | \slashed{\nabla} Y \psi |^2 \\ + & H_4 (TY\psi ) \cdot (T\psi) + H_5 (TY \psi ) \cdot (Y\psi ) + H_6 (Y^2 \psi ) \cdot (T\psi ) \\ + & H_7 (TY \psi ) \cdot (\slashed{\Delta} \psi ) + H_8 (Y^2 \psi ) \cdot ( \slashed{\Delta} \psi ) + H_9 (Y^2 \psi ) \cdot (Y\psi ) + L [ Y \psi ] \cdot YF .
\end{split}
\end{equation*}
We then have that
\begin{equation*}
\begin{split}
H_1 = \partial_r f^v & = \frac{1}{\sigma} , H_2 = \frac{D \cdot \partial_r f^r}{2} - \frac{D \cdot f^r}{r} - \frac{3D' \cdot f^r}{2} =  D \cdot D' + \frac{D^{2}}{r} \\ & \Rightarrow H_2 \simeq \frac{2M}{r^2} D^{3/2} ,  H_3 = -\frac{1}{2} \partial_r f^r = \frac{1}{2}  D' = \frac{M}{r^2} D^{1/2} , \\ & H_4 = \frac{2f^v}{r^2}  , \quad H_5 = f^v \cdot R' , \quad  H_6 = \frac{2f^r}{r^2} = - \frac{2D}{r^2} ,  \\ & H_7 = \frac{2f^v}{r}  , \quad H_8 = \frac{2f^r}{r} = - \frac{2D}{r}, \\ & H_9 = D \cdot \partial_r f^v - D' \cdot f^v_p - \frac{2f^r}{r} \simeq \frac{2D}{r} , \\ & H_{10} = - f^r \cdot R' = D \cdot R'  . 
\end{split}
\end{equation*}
Note that due to the fact that $\psi$ is spherically symmetric  away from the horizon everything can be bounded by $ \int_{\mathcal{A}_{\tau_1}^{\tau_2}} J^T [ T\psi ] $. For the inhomogeneity we have that
\begin{equation*}
\begin{split}
\int_{\mathcal{A}_{\tau_1}^{\tau_2}} YF \cdot L [ Y \psi ] & \leq \epsilon  \int_{\mathcal{A}_{\tau_1}^{\tau_2}} \left[ (TY \psi )^2 + D^{3/2} (Y^2 \psi )^2 \right] \\ & + \frac{1}{\epsilon} \int_{\mathcal{A}_{\tau_1}^{\tau_2}} ( 1 + D^{1/2} ) |YF|^2 .
\end{split}
\end{equation*}

For the $H_4$ term we have by Cauchy-Schwarz that
\begin{equation*}
\int_{\mathcal{A}_{\tau_1}^{\tau_2}} H_4 (TY \psi ) \cdot (T\psi )  \lesssim  \epsilon \int_{\mathcal{A}_{\tau_1}^{\tau_2}} (TY\psi )^2 + \frac{1}{\epsilon} \int_{\mathcal{A}_{\tau_1}^{\tau_2}} (T\psi )^2 ,
\end{equation*}
and for $\epsilon$ small enough the first term can be absorbed by the $H_1$ term, while the second term can be bounded using the Morawetz estimate \eqref{en1}.

For the $H_5$ term we integrate by parts in $T$ and we have that
\begin{equation*}
\begin{split}
\int_{\mathcal{A}_{\tau_1}^{\tau_2}} H_5 (TY \psi ) \cdot (Y\psi) & = \frac{1}{2} \int_{\Sigma_{\tau_2}} H_5 (Y\psi )^2 - \frac{1}{2} \int_{\Sigma_{\tau_1}} H_5 (Y\psi )^2 \\ & \lesssim \int_{\Sigma_{\tau_2}} J^N [  \psi ] + \int_{\Sigma_{\tau_1}}  J^N [\psi ] .
\end{split}
\end{equation*}

For the $H_6$ term we apply again Cauchy-Schwarz and we have that
\begin{equation*}
\begin{split}
\int_{\mathcal{A}_{\tau_1}^{\tau_2}} H_6 (Y^2 \psi ) \cdot (T \psi ) & \lesssim \epsilon \int_{\mathcal{A}_{\tau_1}^{\tau_2}} D^{3/2} (Y^2 \psi )^2 + \frac{1}{\epsilon} \int_{\mathcal{A}_{\tau_1}^{\tau_2}} D^{1/2} (T\psi )^2 \\ & \lesssim \epsilon \int_{\mathcal{A}_{\tau_1}^{\tau_2}} D^{3/2} (Y^2 \psi )^2 + \frac{1}{\epsilon} \int_{\mathcal{A}_{\tau_1}^{\tau_2}} (T\psi )^2 ,
\end{split}
\end{equation*} 
where for $\epsilon$ small enough the first term can be absorbed by the $H_1$ term, while we can use the Morawetz estimate \eqref{en1} for the second term.

Since we work with a spherically symmetric $\psi$ the terms $H_7$ and $H_8$ vanish (as does the term $H_3$). For the term $H_9$ we have that
\begin{equation*}
\int_{\mathcal{A}_{\tau_1}^{\tau_2}} H_9 (Y^2 \psi ) \cdot (TY\psi ) \lesssim \epsilon \int_{\mathcal{A}_{\tau_1}^{\tau_2}} D^{3/2} (Y^2 \psi )^2 + \frac{1}{\epsilon} \int_{\mathcal{A}_{\tau_1}^{\tau_2}} D^{1/2} (TY \psi )^2 
\end{equation*}
and for $\epsilon$ small enough the first term can be absorbed by the $H_2$ term, while the second term can be absorbed by the $H_1$ term due the factor $D^{1/2}$ which vanishes on the horizon.

Finally for the last term $H_{10}$ we have that
\begin{equation*}
\int_{\mathcal{A}_{\tau_1}^{\tau_2}} H_{10} (Y^2 \psi ) \cdot (Y\psi) \mik \epsilon \int_{\mathcal{A}_{\tau_1}^{\tau_2}} D^{3/2} (Y^2 \psi )^2 \, d\mi_{g_{\mathcal{A}}} + \frac{1}{\epsilon} \int_{\mathcal{A}_{\tau_1}^{\tau_2}} \sqrt{D} (Y\psi )^2 \, d\mi_{g_{\mathcal{A}}} ,
\end{equation*}
where for $\epsilon$ small enough the first term can be absorbed by the $H_1$ term, while the second term can be bounded by the $J^N$ flux and its accompanying terms.
\end{proof}
\begin{rem}
Note that in the estimate \eqref{d32tye1} we did not lose any $T$ derivatives as there is no trapping on the photon if $\psi$ is assumed to be spherically symmetric.
\end{rem}

For a general solution of 
\begin{equation*}
\Box_g \psi = F,
\end{equation*}
by separating $\psi$ into its spherically symmetric part $\psi_0$ and into its non-spherically symmetric part $\psi_{\geq 1}$ and then combining estimate \eqref{d32tye1} with Proposition 4.4.1 of \cite{trapping} and estimate \eqref{en1} we have the following estimate for any $l \in \mathbb{N}$:
\begin{equation}\label{d32tye2}
\begin{split}
\int_{\Sigma_{\tau_2}} \left[ (T^{l+1} Y\psi )^2 + | \slashed{\nabla} T^l Y \psi |^2 \right] \, d\mi_{g_{\Sigma}} & +  \int_{\Sigma_{\tau_2}} \left[ D (T^{l} Y^2 \psi_0 )^2 + \sqrt{D} (T^l Y^2 \psi_{\geq 1} )^2 \right] \, d\mi_{g_{\Sigma}} \\ & +\int_{\mathcal{A}_{\tau_1}^{\tau_2}} \left[  D^{3/2} (Y^2 \psi_0 )^2 + D (Y^2 \psi_{\geq 1} \right] \, d\mi_{g_{\mathcal{A}}} \\  \lesssim & \int_{\Sigma_{\tau_1}} \left[ (T^{l+1} Y\psi )^2 + |\slashed{\nabla} T^l Y \psi |^2 \right] \, d\mi_{g_{\Sigma}} \\ & + \int_{\Sigma_{\tau_1}} \left[ D (T^{l} Y^2 \psi_0 )^2 +\sqrt{D} (T^l Y^2 \psi_{\geq 1} )^2 \right] \, d\mi_{g_{\Sigma}} \\  & + \int_{\Sigma_{\tau_1}} J^N [ T^l \psi ] \, d\mi_{g_{\Sigma}} + \int_{\Sigma_{\tau_1}} J^T [T^{l+1} \psi ] \, d\mi_{g_{\Sigma}} \\ & + \int_{\mathcal{A}_{\tau_1}^{\tau_2}}  | T^l YF|^2 \, d\mi_{g_{\mathcal{A}}} + \int_{\mathcal{C}_{\tau_1}^{\tau_2}} | T^{l+1} YF|^2 \, d\mi_{g_{\mathcal{C}}}  \\ & + \sum_{m=l}^{l+1} \Bigg( \int_{\tau_1}^{\tau_2} \int_{S_{\tau'}} |T^m F |^2 \, d\mi_{g_{\mathcal{R}}}  + \int_{\mathcal{C}_{\tau_1}^{\tau_2}} T^{m+1} F |^2 \, d\mi_{g_{\mathcal{C}}} \\ & + \sup_{\tau' \in [\tau_1 , \tau_2]} \int_{\Sigma_{\tau'} \cap \mathcal{C}} | T^m F |^2 \, d\mi_{g_{\mathcal{C}}} + \int_{\tau_1}^{\tau_2} \int_{\mathcal{N}_{\tau'}} r^{1+\eta} | T^m F|^2 \, d\mi_{g_{\mathcal{N}}} \Bigg) \\ & + \left( \int_{\tau_1}^{\tau_2} \left( \int_{S_{\tau'}} | T^l F |^2 d\mi_{g_S} \right)^{1/2} d\tau' \right)^2 .  
\end{split}
\end{equation}
Note that the above estimate also holds after commuting with any number of angular derivatives $\Omega^k$, $k \in \mathbb{N}$, without the $\psi_0$ terms (this is just the inhomogeneous version of Proposition 4.4.1 of \cite{trapping}).

\section{Local theory}\label{tlwpt}
\begin{thm}\label{lwp}
Equation \eqref{nwg}  is locally well-posed in $X (\sis_0 ) \doteq H^3 \times H^2 (\sis_0 )$ in the sense that if we start with data $( \psi_0 , \psi_1 ) \in H^3 \times H^2 (\widetilde{\Sigma}_0 )$ then there exists some $\mathcal{T} > 0$ such that there exists a unique solution of \eqref{nwg} in $\mathcal{R}(0,\mathcal{T} ) = \cup_{\tau \in [0,\mathcal{T}]} \widetilde{\Sigma}_{\tau}$ for which we have that for each $\tau \in [0, \mathcal{T}]$ it holds that $\left( \psi (\tau) , n_{\sis_{\tau}} \psi (\tau) \right) \in X (\widetilde{\Sigma}_{\tau} ) = H^3 \times H^2 (\sis_{\tau} )$.
\end{thm}
\begin{rem}
Recall that if we start with data that is spherically symmetric then our solution of \eqref{nwg} will be spherically symmetric as well. The same holds for solutions of equation \eqref{nwg}.
\end{rem}
The proof of Theorem \ref{lwp} is standard (it follows by using a Gr\"{o}nwall type estimate) and will be omitted.

Next we will state here a condition that allows us to extend a solution beyond the time $\mathcal{T}$ given by the local theory. We have the following continuation criterion.
\begin{prop}[\textbf{Breakdown Criterion}]\label{cont}
Let $\psi$ be a solution of \eqref{nwg} with smooth compactly supported data $\psi [0] = (f , g)$ with finite $X (\sis_0 )$ norm. Denote by $\mathcal{T} = \mathcal{T} (f,g)$ the maximal time of existence for $\psi$ given by Theorem \ref{lwp}. Then we have that either $\mathcal{T} = \infty$ (in which case we say that $\psi$ is globally well-posed) or we have that
$$ T\psi \not\in L^{\infty} (\mathcal{\widetilde{R}} (0,\mathcal{T} )) \mbox{ or/and } \sqrt{D} \cdot Y\psi \not\in L^{\infty} (\mathcal{\widetilde{R}} (0,\mathcal{T} )) \mbox{ or/and } \slashed{\nabla} \psi \not\in L^{\infty} (\mathcal{\widetilde{R}} (0,\mathcal{T} )).  $$ 
\end{prop}

This is a standard by-product of Theorem \ref{lwp} and its proof will be omitted as well. 
 
Proposition \ref{cont} tells us essentially that we have to verify that both 
$$ \| T\psi \|_{L^{\infty} (\mathcal{\widetilde{R}} (0,\mathcal{T} ))}, \quad \| \sqrt{D} \cdot Y\psi \|_{L^{\infty} (\mathcal{\widetilde{R}} (0,\mathcal{T} ))} \mbox{ and } \| \slashed{\nabla} \psi \|_{L^{\infty} (\mathcal{\widetilde{R}} (0,\mathcal{T} ))}$$
are finite at any given $\mathcal{T}$ in order to conclude that our solution $\psi$ (as in the assumption of Proposition \ref{cont}) is globally defined.

\section{Bootstrap assumptions}\label{boot}

From this section and for the rest of the paper in order to simplify our presentation we will study the equation:
\begin{equation}\label{nw}
\left\{\begin{aligned}
       \Box \psi = \sqrt{D} \cdot g^{\alpha \beta} \cdot  \partial_{\alpha}  \psi  \cdot \partial_{\beta} \psi ,\\
       \psi |_{\widetilde{\Sigma_0} } = \ee f , \quad n_{\widetilde{\Sigma}_0} \psi |_{\widetilde{\Sigma_0} } = \ee g, \
       \end{aligned} \right.
\end{equation}

We will assume the estimates stated below in all the sections that will follow. We will later prove them using a bootstrap argument. We  let $\ee > 0$, $\aaa > 0$,  and assume that $\tau \meg 0 $, $\tau_1 \meg 0$, $\tau_2 \meg 0$ are any numbers with $\tau_1 < \tau_2$, and that $C$ is a constant. 

\begin{equation}\label{A1}
\sum_{k\mik 5}  \int_{S_{\tau}} \left( |\Omega^k F|^2 + |\Omega^k TF|^2 \right) d\mi_{g_S}  \mik C E_0 \ee^2 (1+\tau)^{-3} \tag{\textbf{A1}},
 \end{equation} 
\begin{equation}\label{A2}
\sum_{k\mik 5} \int_{\tau_1}^{\tau_2} \int_{N_{\tau'}} r^{3-\alpha} \left( |\Omega^k F |^2 + |\Omega^k TF|^2 \right) d\mi_{g_{\nnn}} \mik C E_0 \ee^2 (1+\tau_1 )^{-2} \tag{\textbf{A2}},
\end{equation}
\begin{equation}\label{A3}
\sum_{k\mik 5}   \left( \int_{\tau_1}^{\tau_2} \left( \int_{N_{\tau'}} r^2 \left( |\Omega^k F |^2 + |\Omega^k TF|^2 \right) d\mi_{g_N} \right)^{1/2} d\tau' \right)^2 \mik C E_0 \ee^2 \tag{\textbf{A3}},
 \end{equation}
 \begin{equation}\label{A4}
\sum_{k\mik 5}  \int_{\calc_{\tau_1}^{\tau_2}} \left( |\Omega^k TF |^2 + |\Omega^k T^2 F |^2 \right) d\mi_{g_{\calc}}   \mik C E_0 \ee^2 (1+ \tau_1)^{-2} \tag{\textbf{A4}} , 
\end{equation}
\begin{equation}\label{B1}
 \sum_{k\mik 5} \int_{S_{\tau}} |\Omega^k T^2 F|^2 d\mi_{g_{S}} \mik C E_0 \ee^2 (1+\tau )^{-11/4+\alpha}  \tag{\textbf{B1}} ,
\end{equation}
\begin{equation}\label{B2}
\sum_{k\mik 5} \int_{\tau_1}^{\tau_2} \int_{S_{\tau'}}  |\Omega^k T^2 F |^2  d\mi_{g_{\rrr}} \mik C E_0 \ee^2 (1+\tau )^{-2+\alpha} \tag{\textbf{B2}} ,
\end{equation}
\begin{equation}\label{B3}
\sum_{k\mik 5} \int_{\tau_1}^{\tau_2} \int_{N_{\tau'}} r^{3-\alpha}  |\Omega^k T^2 F |^2  d\mi_{g_{\nnn}} \mik C E_0 \ee^2 (1+\tau )^{-2+\alpha} \tag{\textbf{B3}} ,
\end{equation}
\begin{equation}\label{B4}
\sum_{k\mik 5}  \int_{\calc_{\tau_1}^{\tau_2}} |\Omega^k T^3 F |^2 d\mi_{g_{\calc}}   \mik C E_0 \ee^2 (1+ \tau_1)^{-2+\aaa} \tag{\textbf{B4}} , 
\end{equation}
\begin{equation}\label{C1}
\sum_{k\mik 5} \int_{S_{\tau}}  |\Omega^k T^3 F |^2  d\mi_{g_S} \mik C E_0 \ee^2 (1+\tau )^{-3/2+\alpha} \tag{\textbf{C1}} ,
\end{equation}
\begin{equation}\label{C2}
\sum_{k\mik 5} \int_{\tau_1}^{\tau_2} \int_{S_{\tau'}}  |\Omega^k T^3 F |^2  d\mi_{g_{\rrr}} \mik C E_0 \ee^2 (1+\tau_1 )^{-3/2+\alpha} \tag{\textbf{C2}} ,
\end{equation}
\begin{equation}\label{C3}
\sum_{k\mik 5} \int_{\tau_1}^{\tau_2} \int_{N_{\tau'}} r^{3-\alpha}  |\Omega^k T^3 F |^2  d\mi_{g_{\nnn}} \mik C E_0 \ee^2 (1+\tau )^{-3/2+\alpha} \tag{\textbf{C3}} ,
\end{equation}
\begin{equation}\label{C4}
\sum_{k\mik 5}  \int_{\calc_{\tau_1}^{\tau_2}}|\Omega^k T^4 F |^2  d\mi_{g_{\calc}}   \mik C E_0 \ee^2 (1+ \tau_1)^{-1+\aaa} \tag{\textbf{C4}} , 
\end{equation}
\begin{equation}\label{D1}
\sum_{k\mik 5}  \int_{\tau_1}^{\tau_2} \int_{\si_{\tau'} \cap \cala_{\tau_1}^{\tau_2}}  |\Omega^k T^4 F|^2 d\mi_{g_{\rrr}}  \mik C E_0 \ee^2 (1+\tau_2 )^{\alpha} \tag{\textbf{D1}} ,
\end{equation}
\begin{equation}\label{D2}
\sum_{k\mik 5}  \left( \int_{\tau_1}^{\tau_2} \left( \int_{\si_{\tau'} \cap \{ r \meg 2M -\delta \}}  |\Omega^k T^4 F |^2 d\mi_{g_{\rrr}} \right)^{1/2} d\tau' \right)^2  \mik C E_0 \ee^2 (1+\tau_2 )^{\alpha} \tag{\textbf{D2}} ,
\end{equation}
\begin{equation}\label{E1}
\sum_{k\mik 5 , l\mik 2}  \int_{\tau_1}^{\tau_2} \int_{\si_{\tau'} \cap \cala_{\tau_1}^{\tau_2}}  |\Omega^k T^l Y F|^2 d\mi_{g_{\rrr}}  \mik C E_0 \ee^2  \tag{\textbf{E1}} ,
\end{equation}
\begin{equation}\label{E2}
\sum_{k\mik 5 , l\mik 2}  \int_{\calc_{\tau_1}^{\tau_2}} |\Omega^k T^{l+1} Y F |^2  d\mi_{g_{\calc}}   \mik C E_0 \ee^2  \tag{\textbf{E2}} ,
\end{equation}
for $\delta$ given in \eqref{mcalc}.

\section{A priori energy and pointwise estimates}

\begin{lem}[\textbf{Estimates for degenerate energies}]\label{denp}
Let $\psi$ be a solution of \eqref{nw}, and assume that the bootstrap assumptions \eqref{A1}, \eqref{A2}, \eqref{A3}, \eqref{A4}, \eqref{B1}, \eqref{B2}, \eqref{B3}, \eqref{B4}, \eqref{C1}, \eqref{C2}, \eqref{C3}, \eqref{C4},  \eqref{D1}, \eqref{D2}, \eqref{E1}, \eqref{E2} hold for $\delta > 0$ given in \eqref{mcalc}, for any $\tau_1$, $\tau_2$ with $\tau_1 < \tau_2$, for some $\ee > 0$, for some $\aaa > 0$, and some constant $C$. Then we have that for any $\tau$ the following estimates are true
\begin{equation}\label{denp1}
\sum_{k\mik 5}\int_{\si_{\tau}} J^T_{\mu} [ \Omega^k \psi ] n^{\mu} d\mi_{g_{\si}} \lesssim \dfrac{E_0 \ee^2}{(1+\tau )^{2}} ,
\end{equation}
\begin{equation}\label{denp2}
\sum_{k\mik 5}\int_{\si_{\tau}} J^T_{\mu} [\Omega^k T\psi ] n^{\mu} d\mi_{g_{\si}} \lesssim \dfrac{E_0 \ee^2}{(1+\tau )^{2}} ,
\end{equation}
\begin{equation}\label{denp3}
\sum_{k \mik 5} \int_{\si_{\tau}} J^T_{\mu} [\Omega^k T^2 \psi ] n^{\mu} d\mi_{g_{\si}} \lesssim \dfrac{E_0 \ee^2}{(1+\tau )^{2-\alpha}} ,
\end{equation}
\begin{equation}\label{denp4}
\sum_{k \mik 5} \int_{\si_{\tau}} J^T_{\mu} [\Omega^k T^3 \psi ] n^{\mu} d\mi_{g_{\si}} \lesssim \dfrac{E_0 \ee^2}{( 1+\tau )^{1+\delta}} ,
\end{equation}
\begin{equation}\label{denp5}
\sum_{k\mik 5}\int_{\si_{\tau}} J^T_{\mu} [\Omega^k T^4 \psi ] n^{\mu} d\mi_{g_{\si}} \lesssim E_0 \ee^2 ( 1+\tau )^{\alpha} .
\end{equation}

\end{lem} 
for some $\delta > 0$. The proof of the above Lemma follows the lines of the analogous statement in \cite{rnnonlin1}, so it will not be repeated. From this proof we just record here the following three Lemmas that follow from it. The first one deals with the behaviour of the $P$-energies.

\begin{lem}[\textbf{Estimates for $P$-energies}]\label{penp}
Let $\psi$ be a solution of \eqref{nw}, and assume that the bootstrap assumptions \eqref{A1}, \eqref{A2}, \eqref{A3}, \eqref{A4}, \eqref{B1}, \eqref{B2}, \eqref{B3}, \eqref{B4}, \eqref{C1}, \eqref{C2}, \eqref{C3}, \eqref{C4}, \eqref{D1}, \eqref{D2}, \eqref{E1}, \eqref{E2} hold for $\delta > 0$ given in \eqref{mcalc}, for any $\tau_1$, $\tau_2$ with $\tau_1 < \tau_2$, for some $\ee > 0$, for some $\aaa > 0$, and some constant $C$. Then we have that for any $\tau$ the following estimates are true
\begin{equation}\label{penp1}
\sum_{k\mik 5}\int_{\si_{\tau}} J^P_{\mu} [ \Omega^k \psi ] n^{\mu} d\mi_{g_{\si}} \lesssim \dfrac{E_0 \ee^2}{1+\tau  } ,
\end{equation}
\begin{equation}\label{penp2}
\sum_{k\mik 5}\int_{\si_{\tau}} J^P_{\mu} [\Omega^k T\psi ] n^{\mu} d\mi_{g_{\si}} \lesssim \dfrac{E_0 \ee^2}{(1+\tau )^2 } ,
\end{equation}
\begin{equation}\label{penp3}
\sum_{k \mik 5} \int_{\si_{\tau}} J^P_{\mu} [\Omega^k T^2 \psi ] n^{\mu} d\mi_{g_{\si}} \lesssim \dfrac{E_0 \ee^2}{(1+\tau )^{2-\alpha}} ,
\end{equation}
\begin{equation}\label{penp4}
\sum_{k \mik 5} \int_{\si_{\tau}} J^P_{\mu} [\Omega^k T^3 \psi ] n^{\mu} d\mi_{g_{\si}} \lesssim \dfrac{E_0 \ee^2}{(1+\tau )^{1+\delta}} .
\end{equation}

\end{lem} 
for some $\delta > 0$. Note that from the last three estimates we have the same decay for the $J^P$ fluxes as with the $J^T$ fluxes due to the fact that we can use estimate \eqref{d32tye} to bound the integrated $J^N$ flux.

The second Lemma gives us estimates on the various non-degenerate energies.
\begin{lem}[\textbf{Estimates for non-degenerate energies}]\label{nenp}
Let $\psi$ be a solution of \eqref{nw}, and assume that the bootstrap assumptions \eqref{A1}, \eqref{A2}, \eqref{A3}, \eqref{A4}, \eqref{B1}, \eqref{B2}, \eqref{B3}, \eqref{B4}, \eqref{C1}, \eqref{C2}, \eqref{C3}, \eqref{C4}  \eqref{D1}, \eqref{D2}, \eqref{E1}, \eqref{E2} hold for $\delta > 0$ given in \eqref{mcalc}, for any $\tau_1$, $\tau_2$ with $\tau_1 < \tau_2$, for some $\ee > 0$, for some $\aaa > 0$, and some constant $C$. Then we have that for any $\tau$ the following estimates are true
\begin{equation}\label{nenp1}
\sum_{k\mik 5}\int_{\si_{\tau}} J^N_{\mu} [\Omega^k \psi ] n^{\mu} d\mi_{g_{\si}} \lesssim E_0 \ee^2 ,
\end{equation}
\begin{equation}\label{nenp2}
\sum_{k\mik 5}\int_{\si_{\tau}} J^N_{\mu} [\Omega^k T\psi ] n^{\mu} d\mi_{g_{\si}} \lesssim E_0 \ee^2,
\end{equation}
\begin{equation}\label{nenp3}
\sum_{k \mik 5} \int_{\si_{\tau}} J^N_{\mu} [\Omega^k T^2 \psi ] n^{\mu} d\mi_{g_{\si}} \lesssim E_0 \ee^2  ,
\end{equation}
\begin{equation}\label{nenp4}
\sum_{k \mik 5} \int_{\si_{\tau}} J^N_{\mu} [\Omega^k T^3 \psi ] n^{\mu} d\mi_{g_{\si}} \lesssim E_0 \ee^2  .
\end{equation}

\end{lem}

The third one is about the decay of the $r^p$-weighted energies on the null hypersurfaces $N_{\tau}$.
\begin{lem}[\textbf{Estimates for $r^p$-weighted energies}]\label{rp}
Let $\psi$ be a solution of \eqref{nw}, and assume that the bootstrap assumptions \eqref{A1}, \eqref{A2}, \eqref{A3}, \eqref{A4}, \eqref{B1}, \eqref{B2}, \eqref{B3}, \eqref{B4}, \eqref{C1}, \eqref{C2}, \eqref{C3}, \eqref{C4}  \eqref{D1}, \eqref{D2}, \eqref{E1}, \eqref{E2} hold for $\delta > 0$ given in \eqref{mcalc}, for any $\tau_1$, $\tau_2$ with $\tau_1 < \tau_2$, for some $\ee > 0$, for some $\aaa > 0$, and some constant $C$. Then we have that for any $\tau$ the following estimates are true
\begin{equation}\label{rp2}
\sum_{k\mik 5}\int_{\si_{\tau}} \left( \frac{(\partial_v \Omega^k \ph )^2}{r^2} + \frac{(\partial_v \Omega^k T\ph )^2}{r^2} \right)  d\mi_{N_{\si}} \lesssim \frac{E_0 \ee^2}{(1+\tau )^2} ,
\end{equation}
\begin{equation}\label{rp1}
\sum_{k\mik 5}\int_{\si_{\tau}}\left( \frac{(\partial_v \Omega^k \ph )^2}{r} + \frac{(\partial_v \Omega^k T\ph )^2}{r} \right)  d\mi_{N_{\si}} \lesssim \frac{E_0 \ee^2}{1+\tau } ,
\end{equation}
\begin{equation}\label{rp2tt}
\sum_{k\mik 5}\int_{\si_{\tau}} \frac{(\partial_v \Omega^k T^2\ph )^2}{r^2}  d\mi_{N_{\si}} \lesssim \frac{E_0 \ee^2}{(1+\tau )^{2-\aaa}} ,
\end{equation}
\begin{equation}\label{rp1tt}
\sum_{k\mik 5}\int_{\si_{\tau}} \frac{(\partial_v \Omega^k T^2\ph )^2}{r}  d\mi_{N_{\si}} \lesssim \frac{E_0 \ee^2}{(1+\tau )^{1-\aaa}} ,
\end{equation}
\begin{equation}\label{rp2ttt}
\sum_{k\mik 5}\int_{\si_{\tau}} \frac{(\partial_v \Omega^k T^3\ph )^2}{r^2}  d\mi_{N_{\si}} \lesssim \frac{E_0 \ee^2}{(1+\tau )^{1-\aaa}} ,
\end{equation}
\begin{equation}\label{rp1ttt}
\sum_{k\mik 5}\left( \int_{\si_{\tau}} \frac{(\partial_v \Omega^k T^3\ph )^2}{r}  d\mi_{N_{\si}} + \int_{\si_{\tau}} \frac{(\partial_v \Omega^k T^3\ph )^2}{r^{\alpha}}  d\mi_{N_{\si}} \right) \lesssim E_0 \ee^2 .
\end{equation}

\end{lem}

Finally we prove he boundedness of the degenerate energies for $Y\psi$.
\begin{lem}[\textbf{Estimates for degenerate energies For $T^l Y\psi$, $l\mik 2$}]\label{ydenp}
Let $\psi$ be a solution of \eqref{nw}, and assume that the bootstrap assumptions \eqref{A1}, \eqref{A2}, \eqref{A3}, \eqref{A4}, \eqref{B1}, \eqref{B2}, \eqref{B3}, \eqref{B4}, \eqref{C1}, \eqref{C2}, \eqref{C3}, \eqref{C4}, \eqref{D1}, \eqref{D2}, \eqref{E1}, \eqref{E2} hold for $\delta > 0$ given in \eqref{mcalc}, for any $\tau_1$, $\tau_2$ with $\tau_1 < \tau_2$, for some $\ee > 0$, for some $\aaa > 0$, and some constant $C$. Then we have that for any $\tau$ the following estimate is true
\begin{equation}\label{d32tye}
\begin{split}
\sum_{k\mik 5 , l \mik 2}\int_{\si_{\tau}} \Big[ & (\Omega^k T^{l+1} Y\psi )^2 +  D(\Omega^k T^l Y^2 \psi )^2 + | \slashed{\nabla} \Omega^k T^l Y \psi |^2 \Big] d\mi_{g_{\si}} \\ & + \sum_{k\mik 5 , l \mik 2}\int_{\cala_{\tau_1}^{\tau_2}} \left[ (\Omega^k T^{l+1} Y\psi )^2 +  D^{3/2} (\Omega^k T^l Y^2 \psi )^2 + | \slashed{\nabla} \Omega^k T^l Y \psi |^2 \right] d\mi_{g_{\cala}}\lesssim E_0 \ee^2 ,
\end{split}
\end{equation}
for any $k\mik 5$ and any $l \mik 2$.
\end{lem} 
\begin{proof}
This is an immediate consequence of the estimates provided by estimate \ref{d32tye2} combined with the bootstrap assumptions \eqref{E1} and \eqref{E2}.
\end{proof}

\begin{thm}[\textbf{Pointwise decay estimates away from the horizon}]\label{dpsigaway}
Let $\psi$ be a solution of \eqref{nw}, and assume that the bootstrap assumptions \eqref{A1}, \eqref{A2}, \eqref{A3}, \eqref{A4}, \eqref{B1}, \eqref{B2}, \eqref{B3}, \eqref{B4}, \eqref{C1}, \eqref{C2}, \eqref{C3}, \eqref{C4}, \eqref{D1}, \eqref{D2}, \eqref{E1}, \eqref{E2} hold for $\delta > 0$ given in \eqref{mcalc}, for any $\tau_1$, $\tau_2$ with $\tau_1 < \tau_2$, for some $\ee > 0$, for some $\aaa > 0$, and some constant $C$. Then we have that for any $\tau$ and any $r$ such that $ r \meg r_0 > M$ for any fixed $r_0 > M$ the following estimates are true with a constant that depends on $r_0$ which goes to infinity as $r_0 \rightarrow M$
\begin{equation}\label{dg1away} 
| r^{1/2} \Omega^k \psi | \lesssim_{r_0} \dfrac{ \sqrt{E_0} \ee}{1+ \tau } \mbox{ for $k \in \{0,1,2,3\}$} ,
\end{equation}
\begin{equation}\label{dg2away} 
\int_{\mathbb{S}^2} r ( \Omega^l \psi )^2 (\tau , r , \omega ) d\omega \lesssim_{r_0} \dfrac{E_0 \ee^2}{ (1+\tau)^2 } \mbox{ for $l \in \{4,5\}$},
\end{equation}
\begin{equation}\label{dg3away} 
| r \Omega^k \psi | \lesssim_{r_0} \dfrac{ \sqrt{E_0} \ee}{(1+ \tau )^{1/2}} \mbox{ for $k \in \{0,1,2,3\}$} ,
\end{equation}
\begin{equation}\label{dg4away} 
\int_{\mathbb{S}^2} r^2 ( \Omega^l \psi )^2 (\tau , r , \omega ) d\omega \lesssim_{r_0} \dfrac{E_0 \ee^2}{1+\tau } \mbox{ for $l \in \{4,5\}$},
\end{equation}
\begin{equation}\label{dgt1away} 
| r^{1/2} \Omega^k T\psi | + | r^{1/2} \Omega^k Y\psi | \lesssim_{r_0} \dfrac{ \sqrt{E_0} \ee}{1+ \tau } \mbox{ for $k \in \{0,1,2,3\}$} ,
\end{equation}
\begin{equation}\label{dgt2away} 
\int_{\mathbb{S}^2} r ( \Omega^l T\psi )^2 (\tau , r , \omega ) d\omega + \int_{\mathbb{S}^2} r ( \Omega^l Y\psi )^2 (\tau , r , \omega ) d\omega \lesssim_{r_0} \dfrac{E_0 \ee^2}{ (1+\tau)^2 } \mbox{ for $l \in \{4,5\}$},
\end{equation}
\begin{equation}\label{dgt3away} 
| r \Omega^k T\psi | + | r \Omega^k Y\psi |  \lesssim_{r_0} \dfrac{ \sqrt{E_0} \ee}{(1+ \tau )^{1/2}} \mbox{ for $k \in \{0,1,2,3\}$} ,
\end{equation}
\begin{equation}\label{dgt4away} 
\int_{\mathbb{S}^2} r^2 ( \Omega^l T\psi )^2 (\tau , r , \omega ) d\omega + \int_{\mathbb{S}^2} r^2 ( \Omega^l Y\psi )^2 (\tau , r , \omega ) d\omega \lesssim_{r_0} \dfrac{E_0 \ee^2}{1+\tau } \mbox{ for $l \in \{4,5\}$},
\end{equation}
\begin{equation}\label{dgtt1away} 
| r^{1/2} \Omega^k T^2\psi | + | r^{1/2} \Omega^k TY\psi | \lesssim_{r_0} \dfrac{ \sqrt{E_0} \ee}{(1+ \tau )^{1-\aaa/2} } \mbox{ for $k \in \{0,1,2,3\}$} ,
\end{equation}
\begin{equation}\label{dgtt2away} 
\int_{\mathbb{S}^2} r ( \Omega^l T^2 \psi )^2 (\tau , r , \omega ) d\omega + \int_{\mathbb{S}^2} r ( \Omega^l TY\psi )^2 (\tau , r , \omega ) d\omega \lesssim_{r_0} \dfrac{E_0 \ee^2}{ (1+\tau)^{2-\aaa} } \mbox{ for $l \in \{4,5\}$},
\end{equation}
\begin{equation}\label{dgtt3away} 
| r^{1-\aaa} \Omega^k T^2\psi | + | r^{1-\aaa} \Omega^k TY\psi |  \lesssim_{r_0} \dfrac{ \sqrt{E_0} \ee}{(1+ \tau )^{1/2 -\aaa/4}} \mbox{ for $k \in \{0,1,2,3\}$} ,
\end{equation}
\begin{equation}\label{dgtt4away} 
\int_{\mathbb{S}^2} r^{2-2\aaa} ( \Omega^l T^2\psi )^2 (\tau , r , \omega ) d\omega + \int_{\mathbb{S}^2} r^{2-2\aaa} ( \Omega^l TY\psi )^2 (\tau , r , \omega ) d\omega \lesssim_{r_0} \dfrac{E_0 \ee^2}{( 1+\tau )^{1-\aaa/2} } \mbox{ for $l \in \{4,5\}$},
\end{equation}
\begin{equation}\label{dgttt1away} 
| r^{1/2} \Omega^k T^3\psi | + | r^{1/2} \Omega^k T^2 Y\psi | \lesssim_{r_0} \dfrac{ \sqrt{E_0} \ee}{(1+ \tau )^{1/2+\beta/2} } \mbox{ for $k \in \{0,1,2,3\}$} ,
\end{equation}
\begin{equation}\label{dgttt2away} 
\int_{\mathbb{S}^2} r ( \Omega^l T^3 \psi )^2 (\tau , r , \omega ) d\omega + \int_{\mathbb{S}^2} r ( \Omega^l T^2 Y\psi )^2 (\tau , r , \omega ) d\omega \lesssim_{r_0} \dfrac{E_0 \ee^2}{ (1+\tau)^{1+\beta} } \mbox{ for $l \in \{4,5\}$},
\end{equation}
\begin{equation}\label{dgttt3away} 
| r^{1-\aaa} \Omega^k T^3 \psi | + | r^{1-\aaa} \Omega^k T^2 Y\psi |  \lesssim_{r_0} \dfrac{ \sqrt{E_0} \ee}{(1+ \tau )^{1/4 +\beta/4}} \mbox{ for $k \in \{0,1,2,3\}$} ,
\end{equation}
\begin{equation}\label{dgttt4away} 
\int_{\mathbb{S}^2} r^{2-2\aaa} ( \Omega^l T^3 \psi )^2 (\tau , r , \omega ) d\omega + \int_{\mathbb{S}^2} r^{2-2\aaa} ( \Omega^l T^2 Y\psi )^2 (\tau , r , \omega ) d\omega \lesssim_{r_0} \dfrac{E_0 \ee^2}{( 1+\tau )^{1/2+\beta/2} } \mbox{ for $l \in \{4,5\}$}.
\end{equation}

\end{thm}
\begin{proof}
First we state some basic inequalities that give us the proofs of the above estimates. For $r^{1/2} (\Omega^l T^m \psi)$ we have that
\begin{equation}\label{basic1}
\int_{\mathbb{S}^2} (\Omega^l T^m \psi )^2 (\tau , r , \omega) d\omega \lesssim_{r_0} \frac{1}{r} \int_{\Sigma_{\tau}} J^T_{\mu} [\Omega^l T^m \psi ] n^{\mu} d\mi_{g_{\Sigma}} ,
\end{equation}
for $l \in \{0,1,2,3,4,5\}$ and $m \in \{0,1,2,3 \}$.

For $r (\Omega^l T^m \psi)$ we have that
$$ \int_{\mathbb{S}^2} r^2 (\Omega^l T^m \psi )^2 (\tau , r , \omega ) d\omega = \int_{\mathbb{S}^2} r_0^2 (\Omega^l T^m \psi )^2 (\tau , r_0 , \omega ) d\omega + 2 \int_{\mathbb{S}^2} \int_{r_0}^r \frac{\psi}{\rho} \partial_{\rho} (\rho \psi) \rho^2 d\rho d\omega \Rightarrow $$
\begin{equation}\label{basic2}
\int_{\mathbb{S}^2} r^2 (\Omega^l T^m \psi )^2 (\tau , r , \omega ) d\omega \lesssim_{r_0} \int_{\Sigma_{\tau}} J^T_{\mu} [\Omega^l T^m \psi ] n^{\mu} d\mi_{g_{\Sigma}} + 
\end{equation}
$$ + \left( \int_{\Sigma_{\tau}} J^T_{\mu} [\Omega^l T^m \psi ] n^{\mu} d\mi_{g_{\Sigma}} \right)^{1/2} \cdot \left( \int_{\Sigma_{\tau}} J^T_{\mu} [\Omega^l T^m \psi ] n^{\mu} d\mi_{g_{\Sigma}} + \int_{N_{\tau}} ( \partial_v \Omega^l T^m \ph )^2 d\mi_{g_N} \right)^{1/2} , $$
for $l \in \{0,1,2,3,4,5\}$ and $m \in \{0,1,2,3 \}$.

For $r^{1-\alpha} (\Omega^l T^m \psi)$ we have that
\begin{equation}\label{basic3}
\int_{\mathbb{S}^2} r^{2-2\alpha} (\Omega^l T^m \psi )^2 (\tau , r , \omega ) d\omega \lesssim_{r_0} \int_{\mathbb{S}^2} r_0^{2-2\alpha} (\Omega^l T^m \psi )^2 (\tau , r_0 , \omega ) d\omega + 
\end{equation}
$$ + \left( \int_{\Sigma_{\tau}} J^T_{\mu} [\Omega^l T^m \psi ] n^{\mu} d\mi_{g_{\Sigma}} \right)^{1/2} \cdot \left( \int_{N_{\tau}} \frac{( \partial_v \Omega^l T^m \ph )^2}{r^{4\alpha}} d\mi_{g_N} \right)^{1/2} , $$
for $l \in \{0,1,2,3,4,5\}$ and $m \in \{0,1,2,3 \}$.
Estimates \eqref{dg1away} and \eqref{dg2away} follow from estimate \eqref{basic1} and estimate \eqref{denp1} of Lemma \ref{denp}, noting that for \eqref{dg1away} we additionally use Sobolev's inequality.

Estimates \eqref{dg3away} and \eqref{dg4away} follow from estimate \eqref{basic2}, estimate \eqref{denp1} of Lemma \ref{denp} and the boundedness of $\int_{N_{\tau}} (\partial_v \Omega^l \ph )^2 d\mi_{g_N}$, noting that for \eqref{dg3away} we additionally use Sobolev's inequality.

Estimates \eqref{dgt1away} and \eqref{dgt2away} follow from estimate \eqref{basic1} and estimate \eqref{denp2} of Lemma \ref{denp}, noting that for \eqref{dgt1away} we additionally use Sobolev's inequality, and that for the term involving $Y\psi$ we use the elliptic estimates of the Appendix \ref{appendix}.

Estimates \eqref{dgt3away} and \eqref{dgt4away} follow from estimate \eqref{basic2}, estimate \eqref{denp2} of Lemma \ref{denp} and the boundedness of $\int_{N_{\tau}} (\partial_v \Omega^l T\psi )^2 d\mi_{g_N}$, noting that for \eqref{dgt3away} we additionally use Sobolev's inequality, and that for the term involving $Y\psi$ we use the elliptic estimates of the Appendix \ref{appendix}.

Estimates \eqref{dgtt1away} and \eqref{dgtt2away} follow from estimate \eqref{basic1} and estimate \eqref{denp3} of Lemma \ref{denp}, noting that for \eqref{dgtt1away} we additionally use Sobolev's inequality, and that for the term involving $Y\psi$ we use the elliptic estimates of the Appendix \ref{appendix}.

Estimates \eqref{dgtt3away} and \eqref{dgtt4away} follow from estimate \eqref{basic3}, estimate \eqref{denp3} of Lemma \ref{denp} and the boundedness of $\int_{N_{\tau}} \frac{(\partial_v \Omega^l T^2 \psi )^2}{r^{4\alpha}} d\mi_{g_N}$, noting that for \eqref{dgtt3away} we additionally use Sobolev's inequality, and that for the term involving $Y\psi$ we use the elliptic estimates of the Appendix \ref{appendix}.

Estimates \eqref{dgttt1away} and \eqref{dgttt2away} follow from estimate \eqref{basic1} and estimate \eqref{denp4} of Lemma \ref{denp}, noting that for \eqref{dgttt1away} we additionally use Sobolev's inequality, and that for the term involving $Y\psi$ we use the elliptic estimates of the Appendix \ref{appendix}.

Estimates \eqref{dgttt3away} and \eqref{dgttt4away} follow from estimate \eqref{basic3}, estimate \eqref{denp4} of Lemma \ref{denp} and the boundedness of $\int_{N_{\tau}} \frac{(\partial_v \Omega^l T^3 \psi )^2}{r^{4\alpha}} d\mi_{g_N}$, noting that for \eqref{dgtt3away} we additionally use Sobolev's inequality, and that for the term involving $Y\psi$ we use the elliptic estimates of the Appendix \ref{appendix}.
\end{proof}

We conclude this section by proving pointwise estimates that are valid in the entirety of the domain of outer communications.
\begin{thm}[\textbf{Global pointwise decay estimates}]\label{dpsig}
Let $\psi$ be a solution of \eqref{nw}, and assume that the bootstrap assumptions \eqref{A1}, \eqref{A2}, \eqref{A3}, \eqref{A4}, \eqref{B1}, \eqref{B2}, \eqref{B3}, \eqref{B4}, \eqref{C1}, \eqref{C2}, \eqref{C3}, \eqref{C4}, \eqref{D1}, \eqref{D2}, \eqref{E1}, \eqref{E2} hold for $\delta > 0$ given in \eqref{mcalc}, for any $\tau_1$, $\tau_2$ with $\tau_1 < \tau_2$, for some $\ee > 0$, for some $\aaa > 0$, and some constant $C$. Then we have that for any $\tau$ and any $r \meg M$ the following estimates are true for some $\beta > 0$
\begin{equation}\label{dg} 
\|\Omega^k \psi \|_{L^{\infty} (\Sigma_{\tau})} \lesssim\dfrac{ \sqrt{E_0} \ee}{(1+ \tau )^{1/2}} \mbox{ for $k \in \{0,1,2,3\}$} ,
\end{equation}
\begin{equation}\label{dg1} 
\int_{\mathbb{S}^2} ( \Omega^l \psi )^2 (\tau , r , \omega ) d\omega \lesssim \dfrac{E_0 \ee^2}{1+\tau } \mbox{ for $l \in \{4,5\}$},
\end{equation}
\begin{equation}\label{dgt} 
\|\Omega^k T\psi \|_{L^{\infty} (\Sigma_{\tau})} \lesssim\dfrac{ \sqrt{E_0} \ee}{(1+ \tau )^{1/2+\beta/2}} \mbox{ for $k \in \{0,1,2,3\}$} ,
\end{equation}
\begin{equation}\label{dgt1} 
\int_{\mathbb{S}^2} ( \Omega^l T\psi )^2 (\tau , r , \omega ) d\omega \lesssim \dfrac{E_0 \ee^2}{(1+\tau )^{1+\beta} } \mbox{ for $l \in \{4,5\}$},
\end{equation}
\begin{equation}\label{dgtt} 
\|\Omega^k T^2 \psi \|_{L^{\infty} (\Sigma_{\tau})} \lesssim\dfrac{ \sqrt{E_0} \ee}{(1+ \tau )^{1/2+\beta/2}} \mbox{ for $k \in \{0,1,2,3\}$} ,
\end{equation}
\begin{equation}\label{dgtt1} 
\int_{\mathbb{S}^2} ( \Omega^l T^2 \psi )^2 (\tau , r , \omega ) d\omega \lesssim \dfrac{E_0 \ee^2}{(1+\tau )^{1+\beta} } \mbox{ for $l \in \{4,5\}$},
\end{equation}
\begin{equation}\label{d2gtt} 
\sqrt{D} |\Omega^k T^2 \psi | (\tau , r , \omega ) \lesssim\dfrac{ \sqrt{E_0} \ee}{(1+ \tau )^{1 - \aaa/2}} \mbox{ for $k \in \{0,1,2,3\}$} ,
\end{equation}
\begin{equation}\label{d2gtt1} 
\int_{\mathbb{S}^2} D ( \Omega^l T^2 \psi )^2 (\tau , r , \omega ) d\omega \lesssim \dfrac{E_0 \ee^2}{(1+\tau )^{2-\aaa} } \mbox{ for $l \in \{4,5\}$},
\end{equation}
\begin{equation}\label{dgttt} 
\sqrt{D} |\Omega^k T^3 \psi | (\tau , r , \omega ) \lesssim\dfrac{ \sqrt{E_0} \ee}{(1+ \tau )^{1/2 + \beta/2}} \mbox{ for $k \in \{0,1,2,3\}$} ,
\end{equation}
\begin{equation}\label{dgttt1} 
\int_{\mathbb{S}^2} D ( \Omega^l T^3 \psi )^2 (\tau , r , \omega ) d\omega \lesssim \dfrac{E_0 \ee^2}{(1+\tau )^{1 + \beta}} \mbox{ for $l \in \{4,5\}$},
\end{equation}
\begin{equation}\label{dgttt2} 
 |\Omega^k T^3 \psi | (\tau , r , \omega ) \lesssim\dfrac{ \sqrt{E_0} \ee}{(1+ \tau )^{1/2 + \beta/2}} \mbox{ for $k \in \{0,1,2,3\}$} ,
\end{equation}
\begin{equation}\label{dgttt3} 
\int_{\mathbb{S}^2} ( \Omega^l T^3 \psi )^2 (\tau , r , \omega ) d\omega \lesssim \dfrac{E_0 \ee^2}{(1+\tau )^{1 + \beta}} \mbox{ for $l \in \{4,5\}$},
\end{equation}
\begin{equation}\label{dgtttt} 
 \sqrt{D} |\Omega^k T^4 \psi | (\tau , r , \omega ) \lesssim \sqrt{E_0} \ee (1+ \tau )^{\aaa/2} \mbox{ for $k \in \{0,1,2,3\}$} ,
\end{equation}
\begin{equation}\label{dgtttt1} 
\int_{\mathbb{S}^2}  D ( \Omega^l T^4 \psi )^2 (\tau , r , \omega ) d\omega \lesssim E_0 \ee^2 (1+\tau )^{\aaa} \mbox{ for $l \in \{4,5\}$},
\end{equation}
\begin{equation}\label{dyg} 
\sqrt{D} \cdot |\Omega^k Y \psi | (\tau , r ,\omega) \lesssim\dfrac{ \sqrt{E_0} \ee}{(1+ \tau )^{1/4}} \mbox{ for $k \in \{0,1,2,3\}$} ,
\end{equation}
\begin{equation}\label{dyg1} 
\int_{\mathbb{S}^2} D ( \Omega^l Y\psi )^2 (\tau , r , \omega ) d\omega \lesssim \dfrac{E_0 \ee^2}{(1+\tau )^{1/2} } \mbox{ for $l \in \{4,5\}$},
\end{equation}
\begin{equation}\label{dtyg} 
\sqrt{D} \cdot |\Omega^k TY \psi | (\tau , r , \omega) \lesssim\dfrac{ \sqrt{E_0} \ee}{(1+ \tau )^{1/4}} \mbox{ for $k \in \{0,1,2,3\}$} ,
\end{equation}
\begin{equation}\label{dtyg1} 
\int_{\mathbb{S}^2} D ( \Omega^l TY\psi )^2 (\tau , r , \omega ) d\omega \lesssim \dfrac{E_0 \ee^2}{( 1+\tau )^{1/2} } \mbox{ for $l \in \{4,5\}$},
\end{equation}
\begin{equation}\label{dttyg} 
\sqrt{D} \cdot |\Omega^k T^2 Y \psi | (\tau , r , \omega) \lesssim\dfrac{ \sqrt{E_0} \ee}{(1+ \tau )^{1/4 - \aaa/4}} \mbox{ for $k \in \{0,1,2,3\}$} ,
\end{equation}
\begin{equation}\label{dttyg1} 
\int_{\mathbb{S}^2} D ( \Omega^l T^2 Y\psi )^2 (\tau , r , \omega ) d\omega \lesssim \dfrac{E_0 \ee^2}{( 1+\tau )^{1/2-\aaa/2} } \mbox{ for $l \in \{4,5\}$},
\end{equation}
\begin{equation}\label{d2yg} 
D \cdot |\Omega^k Y \psi | (\tau , r ,\omega) \lesssim\dfrac{ \sqrt{E_0} \ee}{(1+ \tau )^{1/2}} \mbox{ for $k \in \{0,1,2,3\}$} ,
\end{equation}
\begin{equation}\label{d2yg1} 
\int_{\mathbb{S}^2} D^2 ( \Omega^l Y\psi )^2 (\tau , r , \omega ) d\omega \lesssim \dfrac{E_0 \ee^2}{1+\tau } \mbox{ for $l \in \{4,5\}$},
\end{equation}
\begin{equation}\label{d2tyg} 
D \cdot |\Omega^k TY \psi | (\tau , r , \omega) \lesssim\dfrac{ \sqrt{E_0} \ee}{(1+ \tau )^{1/2}} \mbox{ for $k \in \{0,1,2,3\}$} ,
\end{equation}
\begin{equation}\label{d2tyg1} 
\int_{\mathbb{S}^2} D^2 ( \Omega^l TY\psi )^2 (\tau , r , \omega ) d\omega \lesssim \dfrac{E_0 \ee^2}{1+\tau } \mbox{ for $l \in \{4,5\}$},
\end{equation}
\begin{equation}\label{d2ttyg} 
D \cdot |\Omega^k T^2 Y \psi | (\tau , r , \omega) \lesssim\dfrac{ \sqrt{E_0} \ee}{(1+ \tau )^{1/2 - \aaa/4}} \mbox{ for $k \in \{0,1,2,3\}$} ,
\end{equation}
\begin{equation}\label{d2ttyg1} 
\int_{\mathbb{S}^2} D^2 ( \Omega^l T^2 Y\psi )^2 (\tau , r , \omega ) d\omega \lesssim \dfrac{E_0 \ee^2}{( 1+\tau )^{1-\aaa/2} } \mbox{ for $l \in \{4,5\}$}.
\end{equation}
\end{thm}
\begin{proof}
First we state again some auxiliary estimates. For $\Omega^l T^m \psi$ we have that
\begin{equation}\label{basic11}
\int_{\mathbb{S}^2} (\Omega^l T^m \psi )^2 (\tau , r, \omega ) d\omega \lesssim \int_{\si_{\tau}} J^T_{\mu} [\Omega^l T^m \psi ] n^{\mu} d\mi_{g_{\si}} + 
\end{equation}
$$ + \left( \int_{\si_{\tau}} J^T_{\mu} [\Omega^l T^m \psi ] n^{\mu} d\mi_{g_{\si}} \right)^{1/2} \cdot \left( \int_{\si_{\tau}} J^N_{\mu} [\Omega^l T^m \psi ] n^{\mu} d\mi_{g_{\si}} \right)^{1/2} , $$
for any $r\meg M$ and for $l \in \{0,1,2,3,4,5\}$ and $m \in \{0,1,2,3,4 \}$. We will use the above estimate \eqref{basic11} for $m=0$. We have for $\Omega^l T^m \psi$ that
\begin{equation}\label{basic12}
\int_{\mathbb{S}^2} (\Omega^l T^m \psi )^2 (\tau , r, \omega ) d\omega \lesssim \int_{\si_{\tau}} J^T_{\mu} [\Omega^l T^m \psi ] n^{\mu} d\mi_{g_{\si}} + 
\end{equation}
$$ + \left( \int_{\si_{\tau}} J^T_{\mu} [\Omega^l T^{m-1} \psi ] n^{\mu} d\mi_{g_{\si}} \right)^{1/2} \cdot \left( \int_{\si_{\tau}} J^N_{\mu} [\Omega^l T^m \psi ] n^{\mu} d\mi_{g_{\si}} \right)^{1/2} , $$
for any $r\meg M$ and for $l \in \{0,1,2,3,4,5\}$ and $m \in \{1,2,3,4 \}$, and we also have that
\begin{equation}\label{basic121}
\int_{\mathbb{S}^2} (\Omega^l T^m \psi )^2 (\tau , r, \omega ) d\omega \lesssim \int_{\si_{\tau}} J^T_{\mu} [\Omega^l T^m \psi ] n^{\mu} d\mi_{g_{\si}} + 
\end{equation}
$$ + \left( \int_{\si_{\tau}}D^{1/2 + \bar{\delta}} ( T^m Y\psi )^2 \, d\mi_{g_{\si}} \right)^{1/2} \cdot \left( \int_{\si_{\tau}} D^{1/2 - \bar{\delta}} ( T^m Y\psi )^2 \,  d\mi_{g_{\si}} \right)^{1/2} , $$
for some $\bar{\delta} > 0$, again for any $r\meg M$ and for $l \in \{0,1,2,3,4,5\}$ and $m \in \{1,2,3,4 \}$, which can be used to prove the decay estimates for $T^k \psi$, $k=1,2,3$, after a standard interpolation argument between the $J^T$ and $J^P$ fluxes for the term with $D^{1/2 + \bar{\delta}}$ and between the $J^P$ and $J^N$ fluxes for the term with $D^{1/2 - \bar{\delta}}$. On the other hand for $\sqrt{D} | \Omega^l T^m \psi |$ we have that
\begin{equation}\label{basic13}
\int_{\mathbb{S}^2} D (\Omega^l T^m \psi )^2 (\tau , r, \omega ) d\omega \lesssim \int_{\si_{\tau}} J^T_{\mu} [\Omega^l T^m \psi ] n^{\mu} d\mi_{g_{\si}} + \int_{\si_{\tau}} J^P_{\mu} [\Omega^l T^{m-1} \psi ] n^{\mu} d\mi_{g_{\si}} +
\end{equation}
$$ + \left( \int_{\si_{\tau}} J^T_{\mu} [\Omega^l T^{m-1} \psi ] n^{\mu} d\mi_{g_{\si}} \right)^{1/2} \cdot \left( \int_{\si_{\tau}} J^T_{\mu} [\Omega^l T^m \psi ] n^{\mu} d\mi_{g_{\si}} \right)^{1/2} , $$
for any $r\meg M$ and for $l \in \{0,1,2,3,4,5\}$ and $m \in \{1,2,3,4 \}$.

Finally for $\sqrt{D} |\Omega^l T^k Y\psi|$ we have that
\begin{equation}\label{basic14}
\int_{\mathbb{S}^2} D (\Omega^l T^k Y\psi )^2 (\tau , r, \omega ) d\omega \lesssim \int_{\si_{\tau}} J^T_{\mu} [\Omega^l T^{k+1} \psi ] n^{\mu} d\mi_{g_{\si}} + \int_{\si_{\tau}} J^P_{\mu} [\Omega^l T^{k} \psi ] n^{\mu} d\mi_{g_{\si}} +
\end{equation}
$$ + \left( \int_{\si_{\tau}} J^P_{\mu} [\Omega^l T^{k} \psi ] n^{\mu} d\mi_{g_{\si}} \right)^{1/2} \cdot \left( \int_{\si_{\tau}} D^{3/2} (\Omega^l T^k Y^2 \psi )^2 d\mi_{g_{\si}} \right)^{1/2} , $$
for any $r\meg M$ and for $l \in \{0,1,2,3,4,5\}$ and $k \in \{0,1,2\}$.

For $D | \Omega^l T^k Y\psi |$ we have that
\begin{equation}\label{basic15}
\int_{\mathbb{S}^2} D^2 (\Omega^l T^k Y\psi )^2 (\tau , r, \omega ) d\omega \lesssim \sum_{m = k , k+1} \int_{\si_{\tau}} J^T_{\mu} [\Omega^l T^{m} \psi ] n^{\mu} d\mi_{g_{\si}} + 
\end{equation}
$$ + \left( \int_{\si_{\tau}} J^T_{\mu} [\Omega^l T^{k} \psi ] n^{\mu} d\mi_{g_{\si}} \right)^{1/2} \cdot \left( \int_{\si_{\tau}} D^{3/2} (\Omega^l T^k Y^2 \psi )^2 d\mi_{g_{\si}} \right)^{1/2} , $$
for any $r\meg M$ and for $l \in \{0,1,2,3,4,5\}$ and $k \in \{0,1,2\}$.

\end{proof}

\section{Closing the bootstrap assumptions}
\begin{thm}[\textbf{Bootstrap results}]\label{bootthm}
Let $\psi$ be a solution of \eqref{nw}, and assume that the bootstrap assumptions \eqref{A1}, \eqref{A2}, \eqref{A3}, \eqref{A4}, \eqref{B1}, \eqref{B2}, \eqref{B3}, \eqref{B4}, \eqref{C1}, \eqref{C2}, \eqref{C3}, \eqref{C4}, \eqref{D1}, \eqref{D2}, \eqref{E1}, \eqref{E2} hold for $\delta > 0$ given by \eqref{mcalc}, for some $\alpha > 0$, for any $\tau$, $\tau_1$, $\tau_2$ with $\tau_1 < \tau_2$, for some $\ee > 0$, and some constant $C$. Then we have that 
\begin{equation}\label{A1'}
\sum_{k\mik 5}   \int_{S_{\tau}} \left( |\Omega^k F|^2 + |\Omega^k TF|^2 \right) d\mi_{g_S}  \lesssim E_0^2 \ee^4 (1+\tau )^{-3} \tag{\textbf{A1'}},
 \end{equation} 
\begin{equation}\label{A2'}
\sum_{k\mik 5} \int_{\tau_1}^{\tau_2} \int_{N_{\tau'}} r^{3-\alpha} \left( |\Omega^k F |^2 + |\Omega^k TF|^2 \right) d\mi_{g_{\nnn}} \lesssim E_0^2 \ee^4 (1+\tau_1 )^{-2} \tag{\textbf{A2'}},
\end{equation}
\begin{equation}\label{A3'}
\sum_{k\mik 5}   \left( \int_{\tau_1}^{\tau_2} \left( \int_{N_{\tau'}} r^2 \left( |\Omega^k F |^2 + |\Omega^k TF|^2 \right) d\mi_{g_N} \right)^{1/2} d\tau' \right)^2 \lesssim E_0^2 \ee^4 \tag{\textbf{A3'}},
 \end{equation}
 \begin{equation}\label{A4'}
\sum_{k\mik 5}  \int_{\calc_{\tau_1}^{\tau_2}} \left( |\Omega^k TF |^2 + |\Omega^k T^2 F |^2 \right) d\mi_{g_{\calc}}   \lesssim E_0^2 \ee^4 (1+ \tau_1)^{-2} \tag{\textbf{A4'}} , 
\end{equation}
\begin{equation}\label{B1'}
 \sum_{k\mik 5} \int_{S_{\tau}} |\Omega^k T^2 F|^2 d\mi_{g_{S}} \lesssim E_0^2 \ee^4 (1+\tau )^{-11/4+\alpha}  \tag{\textbf{B1'}} ,
\end{equation}
\begin{equation}\label{B2'}
\sum_{k\mik 5} \int_{\tau_1}^{\tau_2} \int_{S_{\tau'}}  |\Omega^k T^2 F |^2  d\mi_{g_{\rrr}} \lesssim E_0^2 \ee^4 (1+\tau )^{-2+\alpha} \tag{\textbf{B2'}} ,
\end{equation}
\begin{equation}\label{B3'}
\sum_{k\mik 5} \int_{\tau_1}^{\tau_2} \int_{N_{\tau'}} r^{3-\alpha}  |\Omega^k T^2 F |^2  d\mi_{g_{\nnn}} \lesssim E_0^2 \ee^4 (1+\tau )^{-2+\alpha} \tag{\textbf{B3'}} ,
\end{equation}
\begin{equation}\label{B4'}
\sum_{k\mik 5}  \int_{\calc_{\tau_1}^{\tau_2}} |\Omega^k T^3 F |^2 d\mi_{g_{\calc}}   \lesssim E_0^2 \ee^4 (1+ \tau_1)^{-2+\aaa} \tag{\textbf{B4'}} , 
\end{equation}
\begin{equation}\label{C1'}
\sum_{k\mik 5} \int_{S_{\tau}}  |\Omega^k T^3 F |^2  d\mi_{g_S} \lesssim E_0^2 \ee^4 (1+\tau )^{-3/2+\alpha} \tag{\textbf{C1'}} ,
\end{equation}
\begin{equation}\label{C2'}
\sum_{k\mik 5} \int_{\tau_1}^{\tau_2} \int_{S_{\tau'}}  |\Omega^k T^3 F |^2  d\mi_{g_{\rrr}} \lesssim E_0^2 \ee^4 (1+\tau_1 )^{-1+\alpha} \tag{\textbf{C2'}} ,
\end{equation}
\begin{equation}\label{C3'}
\sum_{k\mik 5} \int_{\tau_1}^{\tau_2} \int_{N_{\tau'}} r^{3-\alpha}  |\Omega^k T^3 F |^2  d\mi_{g_{\nnn}} \lesssim E_0^2 \ee^4 (1+\tau )^{-1+\alpha} \tag{\textbf{C3'}} ,
\end{equation}
\begin{equation}\label{C4'}
\sum_{k\mik 5}  \int_{\calc_{\tau_1}^{\tau_2}}|\Omega^k T^4 F |^2  d\mi_{g_{\calc}}   \lesssim E_0^2 \ee^4 (1+ \tau_1)^{-1+\aaa} \tag{\textbf{C4'}} , 
\end{equation}
\begin{equation}\label{D1'}
\sum_{k\mik 5}  \int_{\tau_1}^{\tau_2} \int_{\si_{\tau'} \cap \cala_{\tau_1}^{\tau_2}}  |\Omega^k T^4 F|^2 d\mi_{g_{\rrr}}  \lesssim E_0^2 \ee^4 (1+\tau_2 )^{\alpha} \tag{\textbf{D1'}} ,
\end{equation}
\begin{equation}\label{D2'}
\sum_{k\mik 5}  \left( \int_{\tau_1}^{\tau_2} \left( \int_{\si_{\tau'} \cap \{ r \meg 2M -\delta \}}  |\Omega^k T^4 F |^2 d\mi_{g_{\rrr}} \right)^{1/2} d\tau' \right)^2  \lesssim E_0^2 \ee^4 (1+\tau_2 )^{\alpha} \tag{\textbf{D2'}} ,
\end{equation}
\begin{equation}\label{E1'}
\sum_{k\mik 5 , l\mik 2}  \int_{\tau_1}^{\tau_2} \int_{\si_{\tau'} \cap \cala_{\tau_1}^{\tau_2} } D  |\Omega^k T^l Y F|^2 d\mi_{g_{\rrr}} \lesssim E_0^2 \ee^4  \tag{\textbf{E1'}} ,
\end{equation}
\begin{equation}\label{E2'}
\sum_{k\mik 5 , l\mik 2}  \int_{\calc_{\tau_1}^{\tau_2}} D |\Omega^k T^{l+1} Y F |^2  d\mi_{g_{\calc}}   \lesssim E_0^2 \ee^4  \tag{\textbf{E2'}} .
\end{equation}

\end{thm}
\begin{proof}
For the purpose of this proof, we will use the following notation: whenever we apply $\Omega^k$ angular derivatives to $T\psi$, $Y\psi$ or $\slashed{\nabla} \psi$ we will denote this term by $T\psi_k$, $Y\psi_k$ or $\slashed{\nabla} \psi_k$ respectively. We also note that we will use $S$ for the hypersurface $S_{\tau}$ for some $\tau$, and $N$ for the hypersurface $N_{\tau}$ for some $\tau$.

\eqref{A1'}: We have that
$$ \Omega^k F = \sum_{k_1 + k_2 = k} D^{3/2} ( Y\psi_{k_1} ) \cdot ( Y\psi_{k_2} ) + 2 \sqrt{D} ( T\psi_{k_1} ) \cdot ( Y\psi_{k_2} ) + \sqrt{D} \langle \slashed{\nabla} \psi_{k_1} , \slashed{\nabla} \psi_{k_2} \rangle , $$
which gives us that
$$ \int_S |\Omega^k F |^2 d\mi_{g_S} \lesssim $$ $$ \lesssim \sum_{k_1 + k_2 = k} \left( \int_S D^3 (Y\psi_{k_1} )^2 (Y\psi_{k_2} )^2  d\mi_{g_S} + \int_S D (T\psi_{k_1} )^2 (Y\psi_{k_2} )^2  d\mi_{g_S}
+ \int_S D | \slashed{\nabla} \psi_{k_1} |^2 | \slashed{\nabla} \psi_{k_2} |^2  d\mi_{g_S} \right) . $$
We examine separately the terms of the last line. We also look at the case where $k=5$ since all other cases can be treated in the same way and are additionally simpler. 

$$ \sum_{k_1 + k_2 = 5} \int_S D^3 (Y\psi_{k_1} )^2 (Y\psi_{k_2} )^2  d\mi_{g_S} = $$ $$ = \sum_{k_1 + k_2 = 5, k_1 , k_2 \mik 3} \int_S D^3 (Y\psi_{k_1} )^2 (Y\psi_{k_2} )^2  d\mi_{g_S} + \sum_{k_1 + k_2 = 5, k_1 = 4,5 \mbox{ or } k_2 = 4,5} \int_S D^3 (Y\psi_{k_1} )^2 (Y\psi_{k_2} )^2  d\mi_{g_S} \doteq $$ $$ \doteq I_{A1}a + II_{A1}a . $$
Now we have that
$$ I_{A1}a \lesssim \sum_{i , m_i \mik 3 } \frac{\ee^2}{1+\tau} \int_S D (Y\psi_{m_i} )^2 d\mi_{g_S} \lesssim \frac{E_0 \ee^4}{(1+ \tau )^3} , $$
where we used \eqref{d2yg} and \eqref{denp1}.

We also have that
$$ II_{A1}a =  \sum_{k_1 + k_2 = 5 , k_1 > k_2} \int_S D^3 (Y\psi_{k_1} )^2 (Y\psi_{k_2} )^2  d\mi_{g_S} \lesssim \sum_{k_1 + k_2 = 5 , k_1 > k_2}  \int_S D (Y\psi_{k_1} )^2 \cdot \left( \int_{\mathbb{S}^2} D^2 (Y\psi_{k_2+2} )^2 d\omega \right)  d\mi_{g_S} \lesssim $$ $$ \lesssim \frac{E_0 \ee^2}{1+\tau} \sum_{i , m_i \mik 3 }  \int_S D (Y\psi_{m_i} )^2 d\mi_{g_S} \lesssim \frac{E_0 \ee^4}{(1+ \tau )^3} , $$
where we used \eqref{d2yg1} and \eqref{denp1}.

Now we look at
$$ \sum_{k_1 + k_2 = 5} \int_S D (T\psi_{k_1} )^2 (Y\psi_{k_2} )^2 d\mi_{g_S} = $$ $$ = \sum_{k_1 + k_2 = 5 , k_1 , k_2 \mik 3} \int_S D (T\psi_{k_1} )^2 (Y\psi_{k_2} )^2 d\mi_{g_S} + \sum_{k_1 + k_2 = 5 , k_1 = 4,5} \int_S D (T\psi_{k_1} )^2 (Y\psi_{k_2} )^2 d\mi_{g_S} + $$ $$ + \sum_{k_1 + k_2 = 5 , k_2 = 4,5} \int_S D (T\psi_{k_1} )^2 (Y\psi_{k_2} )^2 d\mi_{g_S}\doteq I_{A1}b + II_{A1}b + III_{A1}b . $$
We have that
$$ I_{A1}b \lesssim \sum_{i , m_i \mik 3} \frac{\ee^2}{1+\tau} \int_S D (Y\psi_{m_i} )^2 d\mi_{g_S} \lesssim \frac{E_0^2\ee^4}{(1+\tau)^3} , $$
where we used \eqref{dgt} and \eqref{denp1}.

Now we have that
$$ II_{A1}b \lesssim \sum_{k_1 + k_2 = 5, k_1 = 4,5 } \int_S (T\psi_{k_1} )^2 \cdot \left( \int_{\mathbb{S}^2} D (Y\psi_{k_2+2} )^2 d\omega \right) d\mi_{g_S} \lesssim \frac{E_0 \ee^2}{1+\tau} \sum_{i , m_i \mik 4} \int_S D (Y\psi_{m_i} )^2 d\mi_{g_S} \lesssim \frac{E_0^2\ee^4}{(1+\tau )^3} ,$$
where we used \eqref{dgt1} and \eqref{denp1}. 

We also have that
$$ III_{A1}b \lesssim \sum_{k_1 + k_2 = 5, k_2 = 4,5 } \int_S D(Y\psi_{k_1} )^2 \cdot \left( \int_{\mathbb{S}^2}  (T\psi_{k_2+2} )^2 d\omega \right) d\mi_{g_S} \lesssim \frac{E_0 \ee^2}{1+\tau} \sum_{i , m_i \mik 3} \int_S D (Y\psi_{m_i} )^2 d\mi_{g_S} \lesssim \frac{E_0^2\ee^4}{(1+\tau )^3} ,$$
where we used \eqref{dgt}, \eqref{dgt1} and \eqref{denp1}.  

Finally we look at
$$ \sum_{k_1 + k_2 = 5} \int_S D | \slashed{\nabla} \psi_{k_1} |^2 | \slashed{\nabla} \psi_{k_2} |^2  d\mi_{g_S} = $$ $$ =\sum_{k_1 + k_2 = 5 , k_1 , k_2 \mik 3} \int_S D | \slashed{\nabla} \psi_{k_1} |^2 | \slashed{\nabla} \psi_{k_2} |^2  d\mi_{g_S} + \sum_{k_1 + k_2 = 5, k_1 = 4,5 \mbox{ or } k_2 = 4,5} \int_S D | \slashed{\nabla} \psi_{k_1} |^2 | \slashed{\nabla} \psi_{k_2} |^2  d\mi_{g_S} \doteq I_{A1}c + II_{A1}c .$$
We have that
$$ I_{A1}c \lesssim \sum_{i , m_i \mik 3} \frac{E_0 \ee^2}{1+\tau} \int_S D |\slashed{\nabla} \psi_{m_i} |^2 d\mi_{g_S} \lesssim \frac{E_0^2\ee^4}{(1+\tau )^3} , $$
where we used \eqref{dg} and \eqref{denp1}.

We also have that
$$ II_{A1}c \lesssim \sum_{k_1 + k_2 = 5 , k_1 > k_2} \int_S D |\slashed{\nabla} \psi_{k_1} |^2 \cdot \left( \int_{\mathbb{S}^2} |\slashed{\nabla} \psi_{k_2 +2} |^2 d\omega \right) d\mi_{g_S} \lesssim \frac{E_0 \ee^2}{1+\tau} \sum_{i , m_i \meg 4} \int_S D |\slashed{\nabla} \psi_{m_i} |^2 d\mi_{g_S} \lesssim \frac{E_0^2\ee^4}{(1+\tau)^3} , $$
where we used \eqref{dg1} and \eqref{denp1}.

All the previous computations gathered together imply that
\begin{equation}\label{A1FS}
 \int_S |\Omega^k F|^2 d\mi_{g_S} \lesssim \frac{E_0^2\ee^4}{(1+\tau )^3}  \mbox{ for any $k \mik 5$},
 \end{equation}
which in turn implies that
\begin{equation}\label{A1F}
\left( \int_{\tau_1}^{\tau_2} \left( \int_{S_{\tau'}} |\Omega^k F|^2 d\mi_{g_S} \right)^{1/2} d\tau' \right)^2 \lesssim \frac{E_0^2\ee^4}{1+\tau_1} , \mbox{ for any $k\mik 5$ and any $\tau_1$, $\tau_2$},
\end{equation}
and
\begin{equation}\label{A1FI}
\int_{\tau_1}^{\tau_2}  \int_{S_{\tau'}} |\Omega^k F|^2 d\mi_{g_{\rrr}}  \lesssim \frac{E_0^2\ee^4}{( 1+\tau_1 )^2} , \mbox{ for any $k\mik 5$ and any $\tau_1$, $\tau_2$}.
\end{equation}

For $TF$ we have that
$$ \Omega^k TF = \sum_{k_1 + k_2 = k} 2 D^{3/2} (TY\psi_{k_1} ) \cdot (Y\psi_{k_2} ) + \sqrt{D} (TT\psi_{k_1} ) \cdot (Y\psi_{k_2} ) + \sqrt{D} (T\psi_{k_1})\cdot (TY\psi_{k_2} ) + 2 \sqrt{D} \langle T\slashed{\nabla} \psi_{k_1} , \slashed{\nabla} \psi_{k_2} \rangle , $$
which implies that
$$\int_S |\Omega^k TF|^2 d\mi_{g_S} \lesssim \sum_{k_1 + k_2 =k} \left( \int_S D^3 (TY\psi_{k_1} )^2 (Y\psi_{k_2} )^2 d\mi_{g_S} + \right. $$ $$ \left. + \int_S D (T^2 \psi_{k_1} )^2 (Y\psi_{k_2} )^2 d\mi_{g_S} + \int_S D (T\psi_{k_1} )^2 (TY\psi_{k_2} )^2 d\mi_{g_S} + \int_S D |T\slashed{\nabla} \psi_{k_1} |^2 | \slashed{\nabla} \psi_{k_2} |^2 d\mi_{g_S} \right) . $$
Once more, we examine separately the terms of the last line for $k=5$ since these estimates are actually easier for all other $k$.
$$ \sum_{k_1 + k_2 = 5} \int_S D^3 (TY\psi_{k_1} )^2 (Y\psi_{k_2} )^2 d\mi_{g_S} = $$ $$ = \sum_{k_1 + k_2 = 5 , k_1 , k_2 \mik 3} \int_S D^3 (TY\psi_{k_1} )^2 (Y\psi_{k_2} )^2 d\mi_{g_S} + \sum_{k_1 + k_2 = 5 , k_1 = 4,5 \mbox{ or } k_2 = 4,5} \int_S D^3 (TY\psi_{k_1} )^2 (Y\psi_{k_2} )^2 d\mi_{g_S} \doteq $$ $$ \doteq I_{A1}d + II_{A1}d . $$

We have that
$$ I_{A1}d \lesssim \frac{E_0 \ee^2}{1+\tau} \sum_{i , m_i \mik 3} \int_S D (TY\psi_{m_i} )^2 d\mi_{g_S} \lesssim \frac{E_0^2\ee^4}{(1+\tau)^3} , $$
where we used \eqref{d2yg} and \eqref{denp2}.

We also have that
$$ II_{A1}d \lesssim \sum_{k_1 + k_2 = 5 , k_1 > k_2} \int_S D (TY\psi_{k_1} )^2  \cdot \left( \int_{\mathbb{S}^2} D^2 (Y\psi_{k_2 +2} )^2 d\omega \right) d\mi_{g_S} + $$ $$ + \sum_{k_1 + k_2 = 5 , k_2 > k_1} \int_S D (Y\psi_{k_2} )^2 \cdot \left( \int_{\mathbb{S}^2} D^2 (TY\psi_{k_1 +2} )^2 d\omega \right) d\mi_{g_S} \lesssim  \frac{E_0^2\ee^4}{(1+\tau)^3} , $$
where we used \eqref{denp2}, \eqref{d2yg1} for the first term, and \eqref{denp1}, \eqref{d2tyg1} for the second one.

Now we look at the term
$$ \sum_{k_1 + k_2 = 5} \int_S D (T^2 \psi_{k_1} )^2 (Y\psi_{k_2} )^2 d\mi_{g_S} = $$ $$ =  \sum_{k_1 + k_2 = 5 , k_1 , k_2 \mik 3 } \int_S D (T^2 \psi_{k_1} )^2 (Y\psi_{k_2} )^2 d\mi_{g_S} + \sum_{k_1 + k_2 = 5, k_1 = 4,5 \mbox{ or } k_2 = 4,5} \int_S D (T^2 \psi_{k_1} )^2 (Y\psi_{k_2} )^2 d\mi_{g_S} \doteq $$ $$\doteq I_{A1}e + II_{A1}e . $$

We have that
$$ I_{A1}e \lesssim \frac{E_0 \ee^2}{1+\tau} \sum_{i , m_i \mik 3} \int_S D (Y\psi_{m_i} )^2 d\mi_{g_S} \lesssim \frac{E_0^2 \ee^4}{(1+\tau )^3} , $$
where we used \eqref{dgtt} and \eqref{denp1}.

We also have that
$$ II_{A1}e \lesssim \sum_{k_1 + k_2 = 5 , k_1 > k_2} \int_S  (T^2 \psi_{k_1} )^2  \cdot \left( \int_{\mathbb{S}^2} D (Y\psi_{k_2 +2} )^2 d\omega \right) d\mi_{g_S} + $$ $$ + \sum_{k_1 + k_2 = 5 , k_2 > k_1} \int_S D (Y\psi_{k_2} )^2 \cdot \left( \int_{\mathbb{S}^2} (T^2 \psi_{k_1 +2} )^2 d\omega \right) d\mi_{g_S} \lesssim \frac{E_0^2 \ee^4}{(1+\tau)^{3}}  , $$
where we used \eqref{denp1} and \eqref{dgtt1}.

Next we look at
$$ \sum_{k_1 + k_2 = 5} \int_S D (T\psi_{k_1} )^2 (TY\psi_{k_2} )^2 d\mi_{g_S} = $$ $$ =  \sum_{k_1 + k_2 = 5 , k_1 , k_2 \mik 3 } \int_S D (T\psi_{k_1} )^2 (TY\psi_{k_2} )^2 d\mi_{g_S} + \sum_{k_1 + k_2 = 5, k_1 = 4,5 \mbox{ or } k_2 = 4,5} \int_S D (T\psi_{k_1} )^2 (TY\psi_{k_2} )^2 d\mi_{g_S} \doteq $$ $$\doteq I_{A1}f + II_{A1}f . $$

We have that
$$ I_{A1}f \lesssim \frac{E_0 \ee^2}{1+\tau} \sum_{i , m_i \mik 3} \int_S D (TY\psi_{m_i} )^2 d\mi_{g_S} \lesssim \frac{E_0^2\ee^4}{(1+\tau)^3} , $$
where we used \eqref{dgt} and \eqref{denp2}.

We also have that
$$ II_{A1}f \lesssim \sum_{k_1 + k_2 = 5 , k_1 > k_2} \int_S D (TY\psi_{k_1} )^2  \cdot \left( \int_{\mathbb{S}^2} (T\psi_{k_2 +2} )^2 d\omega \right) d\mi_{g_S} + $$ $$ + \sum_{k_1 + k_2 = 5 , k_2 > k_1} \int_S  (T\psi_{k_2} )^2 \cdot \left( \int_{\mathbb{S}^2} D (TY\psi_{k_1 +2} )^2 d\omega \right) d\mi_{g_S} \lesssim \frac{E_0^2\ee^4}{(1+\tau)^{3}} , $$
where we used \eqref{denp2} and \eqref{dgt1}.

Finally we look at
$$ \sum_{k_1 + k_2 = 5} \int_S D | T\slashed{\nabla} \psi_{k_1} |^2 | \slashed{\nabla} \psi_{k_2} |^2  d\mi_{g_S} = $$ $$ =\sum_{k_1 + k_2 = 5 , k_1 , k_2 \mik 3} \int_S D | T\slashed{\nabla} \psi_{k_1} |^2 | \slashed{\nabla} \psi_{k_2} |^2  d\mi_{g_S} + \sum_{k_1 + k_2 = 5, k_1 = 4,5 \mbox{ or } k_2 = 4,5} \int_S D | T\slashed{\nabla} \psi_{k_1} |^2 | \slashed{\nabla} \psi_{k_2} |^2  d\mi_{g_S} \doteq $$ $$ \doteq I_{A1}g + II_{A1}g .$$

We have that:
$$ I_{A1}g \lesssim \sum_{i , m_i \mik 3} \frac{E_0 \ee^2}{1+\tau} \int_S D |T\slashed{\nabla} \psi_{m_i} |^2 d\mi_{g_S} \lesssim \frac{E_0^2\ee^4}{(1+\tau )^3} , $$
where we used \eqref{dg} and \eqref{denp2}.

We also have that:
$$ II_{A1}g \lesssim \sum_{k_1 + k_2 = 5 , k_1 > k_2} \int_S D |\slashed{\nabla} \psi_{k_1} |^2 \cdot \left( \int_{\mathbb{S}^2} |\slashed{\nabla} \psi_{k_2 +2} |^2 d\omega \right) d\mi_{g_S} \lesssim $$ $$ \lesssim \frac{E_0 \ee^2}{1+\tau} \sum_{i , m_i \meg 3} \int_S D |\slashed{\nabla} \psi_{m_i} |^2 d\mi_{g_S} +  \frac{E_0 \ee^2}{1+\tau} \sum_{i , m_i \meg 3} \int_S D |T \slashed{\nabla} \psi_{m_i} |^2 d\mi_{g_S}\lesssim \frac{E_0^2\ee^4}{(1+\tau)^3} , $$
where we used \eqref{dg1}, \eqref{dgt1}, \eqref{denp1} and \eqref{denp2}.

Gathering together all the estimates for $\Omega^k TF$ we have that
\begin{equation}\label{A1TFS}
 \int_S |\Omega^k TF|^2 d\mi_{g_S} \lesssim \frac{E_0^2 \ee^4}{(1+\tau)^{3}}  \mbox{ for any $k\mik 5$} , 
 \end{equation}
which along with \eqref{A1FS} closes the \eqref{A1'} bootstrap, and which in turn implies that
\begin{equation}\label{A1TF}
\left( \int_{\tau_1}^{\tau_2} \left( \int_{S_{\tau'}} |\Omega^k TF|^2 d\mi_{g_S} \right)^{1/2} d\tau' \right)^2 \lesssim \frac{E_0^2\ee^4}{1+\tau_1} , \mbox{ for any $k\mik 5$ and any $\tau_1$, $\tau_2$},
\end{equation}
and
\begin{equation}\label{A1TFI}
 \int_{\tau_1}^{\tau_2} \int_{S_{\tau'}} |\Omega^k TF|^2 d\mi_{g_{\rrr}}  \lesssim \frac{E_0^2\ee^4}{(1+\tau_1)^2} , \mbox{ for any $k\mik 5$ and any $\tau_1$, $\tau_2$}.
\end{equation}

\eqref{A2'}: We have that in the hypersurface $N$ the following holds true for $\ph = r\psi$ in the $(u,v)$ null coordinates
$$ |\Omega^k F|^2 \lesssim \sum_{k_1 + k_2 = k} \frac{1}{r^4} \left( (\partial_v \ph_{k_1} )^2 (\partial_u \ph_{k_2} )^2 + \psi_{k_1}^2 (\partial_u \ph_{k_2} )^2 + \psi_{k_1}^2 (\partial_v \ph_{k_2} )^2 + \psi_{k_1}^2 \psi_{k_2}^2 \right) + |\slashed{\nabla} \psi_{k_1} |^2 |\slashed{\nabla} \psi_{k_2} |^2 . $$
We have for $\Omega^k F$ using the previous pointwise inequality for it that
$$ \sum_{k_1 + k_2 \mik 5} \int_{\tau_1}^{\tau_2} \int_{N_{\tau'}} r^{3-\aaa} |\Omega^k F|^2 d\mi_{g_{\nnn}} \lesssim \sum_{k_1 + k_2 \mik 5} \int_{u_{\tau_1}}^{u_{\tau_2}} \int_{v_{R_0}}^{\infty} \int_{\mathbb{S}^2}  r^{5-\aaa} |\Omega^k F|^2 d\omega dv du \lesssim $$ $$ \lesssim I_{A2}a + II_{A2}a + III_{A2}a + IV_{A2}a + V_{A2}a . $$

For $IV_{A2}a$ we have that
$$ IV_{A2}a = \sum_{k_1 + k_2 \mik 5}\int_{u_{\tau_1}}^{u_{\tau_2}} \int_{v_{R_0}}^{\infty} \int_{\mathbb{S}^2} r^{1-\aaa} \psi_{k_1}^2 \psi_{k_2}^2  d\omega dv du  \lesssim  \frac{E_0 \ee^2}{ ( 1+\tau_1 )^2} \sum_{i , m_i \mik 5}\int_{u_{\tau_1}}^{u_{\tau_2}} \int_{v_{R_0}}^{\infty} \int_{\mathbb{S}^2} r^{-\aaa}  \psi_{m_i}^2  d\omega dv du \lesssim $$ $$ \lesssim \frac{E_0 \ee^2}{ ( 1+\tau_1 )^2} \sum_{i , m_i \mik 5} \int_{\tau_1}^{\tau_2} \left( \int_{\si_{\tau'}} J^T_{\mu} [\psi_{m_i} ] n^{\mu} d\mi_{g_{\si}} \right) d\tau' \lesssim \frac{E_0^2 \ee^4}{ (1+\tau_1 )^3}, $$
where we used Sobolev when needed, \eqref{dg1away}, \eqref{dg2away}, Hardy's inequality \eqref{hardy} and \eqref{denp1}. The term $V_{A2}a$ can be treated in the same way.

For $III_{A2}a$ we have that
$$ III_{A2}a = \sum_{k_1 + k_2 \mik 5}\int_{u_{\tau_1}}^{u_{\tau_2}} \int_{v_{R_0}}^{\infty} \int_{\mathbb{S}^2} r^{1-\aaa} \psi_{k_1}^2 (\partial_v \varphi_{k_2} )^2  d\omega dv du  \lesssim $$ $$ \lesssim  \left( \sum_{i , l_i \mik 3} \| r^{1-\aaa} \psi_{l_i}^2 \|_{L^{\infty} (\nnn(\tau_1 , \tau_2 ) } +  \sum_{i , m_i = 4, 5} \left\| \int_{\mathbb{S}^2} r^{1-\aaa} \psi_{m_i}^2 d\omega \right\|_{L^{\infty} (\nnn(\tau_1 , \tau_2 ) } \right) \times $$ $$ \times  \sum_{i , n_i \mik 5}\int_{\tau_1}^{\tau_2} \int_{N_{\tau'}} \frac{(\partial_v \ph_{n_i} )^2}{r^2}  d\mi_{g_{\nnn}} \lesssim \frac{E_0^2 \ee^4}{(1+\tau_1 )^4} , $$
where we used Sobolev when needed, \eqref{dg1away}, \eqref{dg2away} and \eqref{rp2}.

For $II_{A2}a$ we have that
$$ II_{A3}a = \sum_{k_1 + k_2 \mik 5}\int_{u_{\tau_1}}^{u_{\tau_2}} \int_{v_{R_0}}^{\infty} \int_{\mathbb{S}^2} r^{1-\aaa} \psi_{k_1}^2 (\partial_u \varphi_{k_2} )^2  d\omega dv du  \lesssim $$ $$ \lesssim  \left( \sum_{i , l_i \mik 3} \int_{v_{R_0}}^{\infty} \sup_{u , \omega} r^{1-\aaa} \psi_{l_i}^2 dv +  \sum_{i , m_i = 4, 5} \int_{v_{R_0}}^{\infty} \int_{\mathbb{S}^2}  \sup_u r^{1-\aaa} \psi_{m_i}^2 d\omega dv  \right) \times $$ $$ \times  \sum_{i , n_i \mik 5}\sup_v \left( \int_{u_{\tau_1}}^{u_{\tau_2}} \int_{\mathbb{S}^2} (\partial_u \ph_{n_i} )^2  d\omega du \right) \lesssim \frac{E_0^2 \ee^4}{(1+\tau_1 )^3} , $$
where we used Sobolev when needed, \eqref{dg3away}, \eqref{dg4away} and \eqref{denp1} along with the bootstrap assumptions \eqref{A1} and \eqref{A2}, since we have the estimate
\begin{equation}\label{tfluxu}
 \int_{u_{\tau_1}}^{u_{\tau_2}} \int_{\mathbb{S}^2} (\partial_u \ph_{n_i} )^2  d\omega du \lesssim \int_{\si_{\tau_1}} J^T_{\mu} [\psi_k ] n^{\mu} d\mi_{g_{\si}} + \int_{\tau_1}^{\tau_2} \int_{S_{\tau'}} |\Omega^k F |^2 d\mi_{g_{\rrr}} + \int_{\tau_1}^{\tau_2} \int_{S_{\tau'}} r^{1+\eta} |\Omega^k F |^2 d\mi_{g_{\rrr}} + 
\end{equation} 
 $$ + \int_{\calc_{\tau_1}^{\tau_2}} |\Omega^k T F|^2 d\mi_{g_{\nnn}} + \sup_{\tau' \in [\tau_1 , \tau_2 ]} \int_{\si_{\tau'} \cap \calc_{\tau_1}^{\tau_2}} |\Omega^k F |^2 d\mi_{g_{\si}} ,$$
for any $\eta > 0$.

For $I_{A2}a$ we have that
$$ I_{A2}a = \sum_{k_1 + k_2 \mik 5}\int_{u_{\tau_1}}^{u_{\tau_2}} \int_{v_{R_0}}^{\infty} \int_{\mathbb{S}^2} r^{1-\aaa} (\partial_v \ph_{k_1})^2 (\partial_u \psi_{k_2} )^2  d\omega dv du  \lesssim $$ $$ \lesssim    \sum_{l_1 + l_2 \mik 5 , l_2 \mik 2}\int_{u_{\tau_1}}^{u_{\tau_2}}  \int_{\mathbb{S}^2} (\partial_u \ph_{l_1} )^2  d\omega du dv \cdot \int_{v_{R_0}}^{\infty} \sup_{u , \omega} r^{1-\aaa} (\partial_v \ph_{l_2} )^2 dv + $$ $$ + \sum_{m_1 + m_2 \mik 5 , m_2 \mik 2 }\int_{\tau_1}^{\tau_2} \int_{N_{\tau'}} r (\partial_v \ph_{m_1} )^2 \cdot \left( \int_{\mathbb{S}^2} r^{-\aaa} (\partial_u \ph_{m_2 +2 } )^2 d\omega' \right) d\mi_{g_N} \doteq I_{A2}b + II_{A2}b .  $$

We have that
$$ I_{A2}b = \sum_{l_1 + l_2 \mik 5 , l_2 \mik 2}\int_{u_{\tau_1}}^{u_{\tau_2}} \int_{\mathbb{S}^2} (\partial_u \ph_{l_1} )^2  d\omega du \cdot \int_{v_{R_0}}^{\infty}\sup_{u , \omega} r^{1-\aaa} (\partial_v \ph_{l_2} )^2 dv \lesssim \frac{E_0 \ee^2}{(1+\tau_1 )^2} \sum_{m_i \mik 2}\int_{v_{R_0}}^{\infty} \sup_{u , \omega} r^{1-\aaa} (\partial_v \ph_{m_i})^2 dv , $$
where we used \eqref{tfluxu}. But we also have that
$$ \sum_{m_i \mik 2}\int_{v_{R_0}}^{\infty} \sup_{u , \omega} r^{1-\aaa} (\partial_v \ph_{m_i})^2 dv \lesssim \sum_{i , m_i \mik 2}\int_{N_{\tau_1}} \frac{(\partial_v \ph_{m_i + 2} )^2}{r^{1+\aaa}} d\mi_{g_N} + $$ $$ + \int_{\tau_1}^{\tau_2} \int_{N_{\tau'}} \left( \frac{(\partial_v \ph_{m_i + 2} )^2}{r^{1+\aaa}} + \frac{|\slashed{\nabla} \psi_{m_i +3} |^2}{r^{1+\aaa}} + r^{1-\aaa} |\Omega^{m_i +2 } F |^2 \right)d\mi_{g_{\nnn}} \lesssim E_0 \ee^2 , $$
where we used Sobolev, the fundamental Theorem of Calculus, the equation, Proposition \ref{ent}, assumption \eqref{A2} and \eqref{rp1}. So finally we have that
$$  I_{A2}b \lesssim \frac{E_0^2 \ee^4}{(1+\tau_1 )^2} . $$

For $II_{A2}b$ we have that
$$ II_{A3}b = \sum_{m_1 + m_2 \mik 5 , m_2 \mik 2 }\int_{\tau_1}^{\tau_2} \int_{N_{\tau'}} r (\partial_v \ph_{m_1} )^2 \cdot \left( \int_{\mathbb{S}^2} r^{-\aaa} (\partial_u \ph_{m_2 +2 } )^2 d\omega' \right) d\omega dv d\tau'  \lesssim $$ $$ \lesssim \frac{E_0 \ee^2}{1+\tau_1} \sum_{i , m_i \mik 2}\int_{\tau_1}^{\tau_2} \sup_v r^{-\aaa} \int_{\mathbb{S}^2} (\partial_u \ph_{m_i +2} )^2 d\omega d\tau' ,$$
where we used \eqref{rp1}. We also have that
$$ \sum_{i , m_i \mik 2}\int_{\tau_1}^{\tau_2} \sup_v r^{-\aaa} \int_{\mathbb{S}^2} (\partial_u \ph_{m_i +2} )^2 d\omega \lesssim \sum_{i , m_i \mik 2}\int_{\tau_1}^{\tau_2}   \int_{\mathbb{S}^2} \left. r^{-\aaa} (\partial_u \ph_{m_i +2} )^2 \right|_{v = u_{\tau'}} d\omega d\tau' + $$ $$ + \int_{\tau_1}^{\tau_2} \int_{N_{\tau'}} \left( \frac{(\partial_u \ph_{m_i +2 } )^2}{r^{3+\aaa}} + \frac{|\slashed{\nabla} \psi_{m_i + 3} |^2}{r^{1+\aaa}} + r^{1-\aaa} |\Omega^{m_i +2} F |^2 \right) d\mi_{g_{\nnn}} \lesssim \frac{E_0 \ee^2}{(1+\tau_1 )^2} , $$
where we used the fundamental Theorem of Calculus, the equation, Proposition \ref{ent}, assumption \eqref{A2} and \eqref{tfluxu}. In the end we have that
$$ II_{A2}b \lesssim \frac{E_0^2 \ee^4}{(1+\tau_1 )^3} , $$
which finally gives us that
$$ I_{A2}a \lesssim \frac{E_0^2 \ee^4}{(1+\tau_1 )^2} . $$
The term with $\Omega^k TF$ can be treated in a similar fashion.

\eqref{A3'}: With the pointwise estimates from \eqref{A2'} (where the nonlinearity $F$ was written with respect to $\varphi = r \psi$) we use the following notation
$$ \sum_{k_1 + k_2 \mik 5} \int_{N} r^2 |\Omega^k F|^2 d\mi_{g_N} \lesssim  I_{A3} + II_{A3} + III_{A3} + IV_{A3}+ V_{A3} . $$
For $IV_{A3}$ we have that
$$ IV_{A3} = \sum_{k_1 + k_2 \mik 5} \int_{N} \frac{\psi_{k_1}^2 \psi_{k_2}^2}{r^2} d\mi_{g_N} \lesssim \frac{E_0 \ee^2}{(1+\tau )^2} \sum_{i , m_i \mik 5} \int_{N} J^T_{\mu} [\psi_{m_i} ] n^{\mu} d\mi_{g_N}  \lesssim \frac{E_0^2 \ee^4}{(1+\tau )^4} , $$
where we used Sobolev when needed, \eqref{dg}, \eqref{dg1}, Hardy's inequality \eqref{hardy} and \eqref{rp1}. $V_{A3}$ can be treated similarly.

For $III_{A3}$ we have that
$$ III_{A3} = \sum_{k_1 + k_2 \mik 5} \int_{N} \psi_{k_1}^2 \frac{(\partial_v \ph_{k_2} )^2}{r^2} d\mi_{g_N} \lesssim \frac{E_0 \ee^2}{(1+\tau )^2} \sum_{i , m_i \mik 5} \int_{N} \frac{(\partial_v \ph_{m_i} )^2}{r^2} d\mi_{g_N} \lesssim \frac{E_0^2 \ee^4}{(1+\tau )^4}  , $$
where we used Sobolev when needed, \eqref{dg}, \eqref{dg1} and \eqref{rp2}.

For $II_{A3}$ we have that
$$ II_{A3} = \sum_{k_1 + k_2 \mik 5} \int_{N} \frac{\psi_{k_1}^2}{r^2} (\partial_u \ph_{k_2} )^2 d\mi_{g_N} \lesssim \frac{E_0 \ee^2}{1+\tau } \sum_{i , m_i \mik 5} \int_{N} J^T_{\mu} [\psi_{m_i} ] n^{\mu} d\mi_{g_N} \lesssim \frac{E_0^2 \ee^4}{(1+\tau )^3}  , $$
where we used Sobolev when needed, \eqref{dgt3away}, \eqref{dgt4away} and Hardy's inequality \eqref{hardy}.

For $I_{A3}$ we have that
$$ I_{A4} = \sum_{k_1 + k_2 \mik 5} \int_{N} \frac{( \partial_v \ph_{k_1} )^2}{r^2} (\partial_u \ph_{k_2} )^2 d\mi_{g_N} \lesssim \frac{E_0 \ee^2}{1+\tau } \sum_{i , m_i \mik 5} \int_{N} \frac{(\partial_v \ph_{m_i} )^2}{r^2} d\mi_{g_N}  \lesssim \frac{E_0^2 \ee^4}{(1+\tau )^3}  , $$
where we used Sobolev when needed, \eqref{dgt3away}, \eqref{dgt4away} and \eqref{rp2}.

So in the end we have that
$$ \sum_{k \mik 5} \left( \int_{\tau_1}^{\tau_2} \left( \int_{N_{\tau'}} r^2 |\Omega^k F|^2 d\mi_{g_N} \right)^{1/2} d\tau' \right)^2 \lesssim \frac{E_0^2 \ee^4}{1+\tau_1} . $$

\eqref{A4'}: We start by noticing that the part that involves $\Omega^k TF$ satisfies the desired estimate by \eqref{A2'} (in $\calc$ we have actually something even stronger than what is needed). For $\Omega^k T^2F$ we have that
$$ \Omega^k T^2 F = \sum_{k_1 + k_2 = k} 2 D^{3/2} (T^2 Y\psi_{k_1} ) \cdot (Y\psi_{k_2} ) + 2D^{3/2} (TY\psi_{k_1} ) \cdot (TY\psi_{k_2} ) + \sqrt{D} (T^3 \psi_{k_1} ) \cdot (Y\psi_{k_2} ) + $$ $$ + 2 \sqrt{D} (T^2 \psi_{k_1} ) \cdot (TY\psi_{k_2} ) + \sqrt{D} (T\psi_{k_1} ) \cdot (T^2 Y\psi_{k_2} ) + 2 \sqrt{D} \langle T^2 \slashed{\nabla} \psi_{k_1} , \slashed{\nabla} \psi_{k_2} \rangle + 2 \sqrt{D} \langle T\slashed{\nabla} \psi_{k_1} , T\slashed{\nabla} \psi_{k_2} \rangle , $$
which implies that
$$ |\Omega^k T^2 F|^2 \lesssim \sum_{k_1 + k_2 = k} D^3 (T^2 Y\psi_{k_1} )^2 (Y\psi_{k_2} )^2  +D^{3} (TY\psi_{k_1} )^2  (TY\psi_{k_2} )^2+ D (T^3 \psi_{k_1} )^2 (Y\psi_{k_2} )^2 + $$ $$ + D (T^2 \psi_{k_1} )^2  (TY\psi_{k_2} )^2 + D (T\psi_{k_1} )^2  (T^2 Y\psi_{k_2} )^2 + D | T^2 \slashed{\nabla} \psi_{k_1} |^2  |\slashed{\nabla} \psi_{k_2} |^2 + D | T\slashed{\nabla} \psi_{k_1} |^2 | T\slashed{\nabla} \psi_{k_2} |^2 . $$
We denote the terms above as follows after integrating them over $S \cap \calc$
$$ \int_{S_{\tau} \cap \calc_{\tau_1}^{\tau_2}}  |\Omega^k TF|^2 d\mi_{g_S} \lesssim I_{A4} + II_{A4} + III_{A4} + IV_{A4} + V_{A4} + VI_{A4} + VII_{A4} . $$
Now we have that
$$ I_{A4} = \sum_{k_1 + k_2 \mik 5 } \int_{S_{\tau} \cap \calc_{\tau_1}^{\tau_2}} D^3 (T^2 Y\psi_{k_1} )^2 (Y\psi_{k_2} )^2  d\mi_{g_S} \lesssim $$ $$ \lesssim \frac{E_0 \ee^2}{(1+\tau )^2} \sum_{i , m_i \mik 5} \int_{S_{\tau} \cap \calc_{\tau_1}^{\tau_2}} (T^2 Y\psi_{m_i} )^2 d\mi_{g_S} \lesssim \frac{E_0 \ee^2}{(1+\tau )^2} \sum_{i , m_i \mik 5} \int_{S_{\tau}} J^T_{\mu} [T^2 \psi_{m_i} ] n^{\mu} d\mi_{g_S} \lesssim \frac{E_0^2 \ee^4}{(1+\tau )^{4-\aaa}} , $$
where we used Sobolev when needed, \eqref{dgt1away}, \eqref{dgt2away} and \eqref{denp3}.

For $II_{A4}$ we have that
$$ II_{A4} = \sum_{k_1 + k_2 \mik 5 } \int_{S_{\tau} \cap \calc_{\tau_1}^{\tau_2}} D^{3} (TY\psi_{k_1} )^2  (TY\psi_{k_2} )^2  d\mi_{g_S} \lesssim $$ $$\lesssim \frac{E_0 \ee^2}{(1+\tau )^{2-\aaa}} \sum_{i , m_i \mik 5} \int_{S_{\tau} \cap \calc_{\tau_1}^{\tau_2}} (T Y\psi_{m_i} )^2 d\mi_{g_S} \lesssim \frac{E_0 \ee^2}{(1+\tau )^{2-\aaa}} \sum_{i , m_i \mik 5} \int_{S_{\tau}} J^T_{\mu} [T \psi_{m_i} ] n^{\mu} d\mi_{g_S} \lesssim \frac{E_0^2 \ee^4}{(1+\tau )^{4-\aaa}} , $$
where we used Sobolev when needed, \eqref{dgtt1away}, \eqref{dgtt2away} and \eqref{denp2}.

For $III_{A4}$ we have that
$$ III_{A4} = \sum_{k_1 + k_2 \mik 5 } \int_{S_{\tau} \cap \calc_{\tau_1}^{\tau_2}} D (T^3 \psi_{k_1} )^2 (Y\psi_{k_2} )^2 d\mi_{g_S} \lesssim $$ $$\lesssim \frac{E_0 \ee^2}{(1+\tau )^2} \sum_{i , m_i \mik 5} \int_{S_{\tau} \cap \calc_{\tau_1}^{\tau_2}} (T^3 \psi_{m_i} )^2 d\mi_{g_S} \lesssim \frac{E_0 \ee^2}{(1+\tau )^2} \sum_{i , m_i \mik 5} \int_{S_{\tau}} J^T_{\mu} [T^2 \psi_{m_i} ] n^{\mu} d\mi_{g_S} \lesssim \frac{E_0^2 \ee^4}{(1+\tau )^{4-\aaa}} , $$
where we used Sobolev when needed, \eqref{dgt1away}, \eqref{dgt2away} and \eqref{denp3}.

For $IV_{A4}$ we have that
$$ IV_{A4} = \sum_{k_1 + k_2 \mik 5 } \int_{S_{\tau} \cap \calc_{\tau_1}^{\tau_2}}  D (T^2 \psi_{k_1} )^2  (TY\psi_{k_2} )^2 d\mi_{g_S} \lesssim $$ $$ \lesssim \frac{\ee^2}{(1+\tau )^{2-\aaa}} \sum_{i , m_i \mik 5} \int_{S_{\tau} \cap \calc_{\tau_1}^{\tau_2}} (T Y\psi_{m_i} )^2 d\mi_{g_S} \lesssim \frac{\ee^2}{(1+\tau )^{2-\aaa}} \sum_{i , m_i \mik 5} \int_{S_{\tau}} J^T_{\mu} [T \psi_{m_i} ] n^{\mu} d\mi_{g_S} \lesssim \frac{\ee^4}{(1+\tau )^{4-\aaa}} , $$
where we used Sobolev when needed, \eqref{dgtt1away}, \eqref{dgtt2away} and \eqref{denp2}.

For $V_{A4}$ we have that
$$ V_{A4} = \sum_{k_1 + k_2 \mik 5 } \int_{S_{\tau} \cap \calc_{\tau_1}^{\tau_2}} D (T\psi_{k_1} )^2  (T^2 Y\psi_{k_2} )^2 d\mi_{g_S} \lesssim $$ $$ \lesssim\frac{E_0 \ee^2}{(1+\tau )^2} \sum_{i , m_i \mik 5} \int_{S_{\tau} \cap \calc_{\tau_1}^{\tau_2}} (T^2 Y \psi_{m_i} )^2 d\mi_{g_S} \lesssim \frac{E_0 \ee^2}{(1+\tau )^2} \sum_{i , m_i \mik 5} \int_{S_{\tau}} J^T_{\mu} [T^2 \psi_{m_i} ] n^{\mu} d\mi_{g_S} \lesssim \frac{E_0^2 \ee^4}{(1+\tau )^{4-\aaa}} , $$
where we used Sobolev when needed, \eqref{dgt1away}, \eqref{dgt2away} and \eqref{denp3}.

For $VI_{A4}$ we have that
$$ VI_{A4} = \sum_{k_1 + k_2 \mik 5 } \int_{S_{\tau} \cap \calc_{\tau_1}^{\tau_2}} D | T^2 \slashed{\nabla} \psi_{k_1} |^2  |\slashed{\nabla} \psi_{k_2} |^2 d\mi_{g_S} \lesssim $$ $$ \lesssim \sum_{i , m_i \mik 5} \frac{E_0 \ee^2}{(1+\tau )^2} \int_{S_{\tau} \cap \calc_{\tau_1}^{\tau_2}} D | T^2 \slashed{\nabla} \psi_{m_i} |^2 d\mi_{g_S} + \sum_{i , m_i \mik 5} \frac{E_0 \ee^2}{(1+\tau )^{2-\aaa}} \int_{S_{\tau} \cap \calc_{\tau_1}^{\tau_2}} D |  \slashed{\nabla} \psi_{m_i} |^2 d\mi_{g_S} \lesssim \frac{E_0^2 \ee^4}{(1+\tau )^{4-\aaa}} , $$
where we used Sobolev when needed, \eqref{dg1away}, \eqref{dg2away}, \eqref{dgtt1away}, \eqref{dgtt2away}, \eqref{denp3} and \eqref{denp1}.

Finally for $VII_{A4}$ we also have that
$$ VII_{A4} = \sum_{k_1 + k_2 \mik 5 } \int_{S_{\tau} \cap \calc_{\tau_1}^{\tau_2}} D | T \slashed{\nabla} \psi_{k_1} |^2  |T\slashed{\nabla} \psi_{k_2} |^2 d\mi_{g_S} \lesssim $$ $$ \lesssim  \frac{E_0 \ee^2}{(1+\tau)^{2-\aaa}} \sum_{i , m_i \mik 5}  \int_{S_{\tau} \cap \calc_{\tau_1}^{\tau_2}} | T\slashed{\nabla} \psi_{m_i} |^2 d\mi_{g_S} \lesssim \frac{E_0^2 \ee^4}{(1+\tau )^{4-\aaa}} , 
$$ 
where we used Sobolev when needed, \eqref{dgt1away}, \eqref{dgt2away} and \eqref{denp2}.

Gathering all the above estimate we get that for any $k \mik 5$ we have that
$$ \int_{S_{\tau} \cap \calc_{\tau_1}^{\tau_2}} |\Omega^k T^2 F|^2 d\mi_{g_S} \lesssim \frac{E_0^2 \ee^4}{(1+\tau )^{4-\aaa}} , $$
which implies the better than required estimate
$$ \int_{\calc_{\tau_1}^{\tau_2}} |\Omega^k TF|^2 d\mi_{g_{\calc}} \lesssim \frac{E_0^2 \ee^4}{(1+\tau_1 )^{3-\aaa}} . $$

\eqref{B1'}: We deal again with the seven terms that we dealt with  in \eqref{A4'}, but now after integrating them over $S$. We have that
$$I_{B1} = \sum_{k_1 + k_2 \mik 5 } \int_S D^3 (T^2 Y\psi_{k_1} )^2 (Y\psi_{k_2} )^2 d\mi_{g_S} \lesssim \frac{E_0 \ee^2}{1+\tau} \sum_{i , m_i \mik 5} \int_S D (T^2 Y \psi_{m_i} )^2 d\mi_{g_S} \lesssim \frac{E_0^2 \ee^4}{(1+\tau )^{3-\aaa}} , $$
where we used Sobolev when needed, \eqref{d2yg}, \eqref{d2yg1} and \eqref{denp3}.

For $II_{B1}$ we have that
$$ II_{B1} = \sum_{k_1 + k_2 \mik 5 } \int_S D^3 (T Y\psi_{k_1} )^2 (TY\psi_{k_2} )^2 d\mi_{g_S} \lesssim \frac{E_0 \ee^2}{1+\tau} \sum_{i , m_i \mik 5} \int_S D (T Y \psi_{m_i} )^2 d\mi_{g_S} \lesssim \frac{E_0^2 \ee^4}{(1+\tau )^{3}} , $$
where we used Sobolev when needed, \eqref{d2tyg}, \eqref{d2tyg1} and \eqref{denp2}.

For $III_{B1}$ we have that
$$ III_{B1} = \sum_{k_1 + k_2 \mik 5 } \int_S D (T^3 \psi_{k_1} )^2 (Y\psi_{k_2} )^2 d\mi_{g_S} \lesssim\frac{E_0 \ee^2}{(1+\tau)^{3/4-\aaa}} \sum_{i , m_i \mik 5} \int_S D (Y \psi_{m_i} )^2 d\mi_{g_S} \lesssim \frac{E_0^2 \ee^4}{(1+\tau )^{11/4-\aaa}} , $$
where we used Sobolev when needed, \eqref{dyg}, \eqref{dyg1} and \eqref{denp3}.

For $IV_{B1}$ we have that
$$ IV_{B1} = \sum_{k_1 + k_2 \mik 5 } \int_S D (T^2 \psi_{k_1} )^2 (TY\psi_{k_2} )^2 d\mi_{g_S} \lesssim\frac{E_0 \ee^2}{(1+\tau )^{1-\aaa}} \sum_{i , m_i \mik 5} \int_S D (TY \psi_{m_i} )^2 d\mi_{g_S} \lesssim \frac{E_0^2 \ee^4}{(1+\tau )^{3-\aaa}} , $$
where we used Sobolev when needed, \eqref{dgtt}, \eqref{dgtt1} and \eqref{denp2}. 

For $V_{B1}$ we have that
$$ V_{B1} = \sum_{k_1 + k_2 \mik 5 } \int_S D (T \psi_{k_1} )^2 (T^2 Y\psi_{k_2} )^2 d\mi_{g_S} \lesssim\frac{\ee^2}{1+\tau} \sum_{i , m_i \mik 5} \int_S D (T^2 Y \psi_{m_i} )^2 d\mi_{g_S} \lesssim \frac{\ee^4}{(1+\tau )^{3-\aaa}} , $$
where we used Sobolev when needed, \eqref{dgt}, \eqref{dgt1} and \eqref{denp3}. 

For $VI_{B1}$ we have that
$$ VI_{B1} = \sum_{k_1 + k_2 \mik 5 } \int_S D |T^2 \slashed{\nabla} \psi_{k_1} |^2 | \slashed{\nabla} \psi_{k_2} |^2 d\mi_{g_S} \lesssim $$ $$ \lesssim \frac{E_0 \ee^2}{1+\tau} \sum_{i , m_i \mik 5} \int_S D | T^2 \slashed{\nabla} \psi_{m_i} |^2 d\mi_{g_S} + \frac{E_0 \ee^2}{(1+\tau )^{2-\aaa}} \sum_{i , m_i \mik 5} \int_S | \slashed{\nabla} \psi_{m_i} |^2 d\mi_{g_S} \lesssim \frac{E_0^2 \ee^4}{(1+\tau )^{3}} , $$
where we used Sobolev when needed, \eqref{dg}, \eqref{dg1}, \eqref{d2gtt}, \eqref{d2gtt1}, \eqref{denp1} and \eqref{denp3}. 

Finally for $VII_{B1}$ we have that
$$ VII_{B1} = \sum_{k_1 + k_2 \mik 5 } \int_S D |T \slashed{\nabla} \psi_{k_1} |^2 | T\slashed{\nabla} \psi_{k_2} |^2 d\mi_{g_S} \lesssim $$ $$ \lesssim\frac{E_0 \ee^2}{(1+\tau )^{1-\aaa}} \sum_{i , m_i \mik 5} \int_S | T\slashed{\nabla} \psi_{m_i} |^2 d\mi_{g_S} \lesssim \frac{E_0^2 \ee^4}{(1+\tau )^{3-\aaa}} , $$
where we used Sobolev when needed, \eqref{dgt}, \eqref{dgt1} and \eqref{denp2}. 

\eqref{B2'}: All terms in the proof of \eqref{B1'} decay with rate $-3+\aaa$ (which is enough to prove \eqref{B2'}), apart from the term $\sum_{k_1 + k_2 \mik 5 } \int_S D (T^3 \psi_{k_1} )^2 (Y\psi_{k_2} )^2 d\mi_{g_S}$. We will treat this differently using Morawetz, and get the improved estimates that is required here. We have that
$$ \sum_{k_1 + k_2 \mik 5 } \int_{\tau_1}^{\tau_2} \int_{S_{\tau'}} D (T^3 \psi_{k_1} )^2 (Y\psi_{k_2} )^2 d\mi_{g_{\rrr}} = $$ $$ = \sum_{k_1 + k_2 \mik 5 } \int_{\cala_{\tau_1}^{\tau_2}} D (T^3 \psi_{k_1} )^2 (Y\psi_{k_2} )^2 d\mi_{g_{\cala}} + \sum_{k_1 + k_2 \mik 5 } \int_{\tau_1}^{\tau_2} \int_{S_{\tau'} \cap (\cala_{\tau_1}^{\tau_2} )^c} D (T^3 \psi_{k_1} )^2 (Y\psi_{k_2} )^2 d\mi_{g_{\rrr}} . $$
For the term $\sum_{k_1 + k_2 \mik 5 } \int_{\tau_1}^{\tau_2} \int_{S_{\tau'} \cap (\cala_{\tau_1}^{\tau_2} )^c} D (T^3 \psi_{k_1} )^2 (Y\psi_{k_2} )^2 d\mi_{g_{\rrr}}$ we have that it decays with the rate $-2$ by the proof of \eqref{A4'}. For the term $\sum_{k_1 + k_2 \mik 5 } \int_{\cala_{\tau_1}^{\tau_2}} D (T^3 \psi_{k_1} )^2 (Y\psi_{k_2} )^2 d\mi_{g_{\cala}}$ we have that
$$\sum_{k_1 + k_2 \mik 5 } \int_{\cala_{\tau_1}^{\tau_2}} D (T^3 \psi_{k_1} )^2 (Y\psi_{k_2} )^2 d\mi_{g_{\cala}} \lesssim \frac{E_0 \ee^2}{(1+\tau_1 )^{1/2}} \int_{\tau_1}^{\tau_2} \int_{S_{\tau'}} (T^3 \psi_{k_1} )^2 d\mi_{g_{\rrr}} \lesssim \frac{E_0^2 \ee^4}{(1+\tau_1 )^{5/2-\aaa}} ,$$
where we used Sobolev when needed, \eqref{dyg}, \eqref{dyg1} and the estimate
$$ \int_{\cala_{\tau_1}^{\tau_2}} (T^3 \psi_{k} )^2 d\mi_{g_{\cala}} \lesssim \frac{E_0 \ee^2}{(1+\tau_1 )^{2-\aaa}} \mbox{ for any $k \mik 5$} , $$
which is a consequence of \eqref{denp3}.

\eqref{B3'}: The proof is very similar to that of \eqref{A3'} and hence will not be repeated.

\eqref{B4'}: We have that
$$ \Omega^k T^3 F = \sum_{k_1 + k_2 = k } 2 D^{3/2} (T^3 Y\psi_{k_1} ) \cdot (Y\psi_{k_2} ) + 6 D^{3/2} (T^2 Y\psi_{k_1} ) \cdot (TY\psi_{k_2} ) + \sqrt{D} (T^4 \psi_{k_1} ) \cdot (Y\psi_{k_2} ) + $$ $$ + 3\sqrt{D} (T^3 \psi_{k_1} ) \cdot (TY\psi_{k_2} ) + 3 \sqrt{D} (T^2 \psi_{k_1} ) \cdot (T^2 Y\psi_{k_2} ) +  \sqrt{D} (T\psi_{k_1} ) \cdot (T^3 Y\psi_{k_2} ) + $$ $$ + 2\sqrt{D} \langle T^3 \slashed{\nabla} \psi_{k_1} , \slashed{\nabla} \psi_{k_2} \rangle + 6\sqrt{D} \langle T^2 \slashed{\nabla} \psi_{k_1} , T\slashed{\nabla} \psi_{k_2} \rangle , $$
which implies that
$$ |\Omega^k T^3 F|^2 \lesssim \sum_{k_1 + k_2 = k } D^3 (T^3 Y\psi_{k_1} )^2 (Y\psi_{k_2} )^2 + D^3 (T^2 Y\psi_{k_1} )^2 (TY\psi_{k_2} )^2 + D(T^4 \psi_{k_1} )^2 (Y\psi_{k_2} )^2 + $$ $$ + D (T^3 \psi_{k_1} )^2 (TY\psi_{k_2} )^2 + D (T^2 \psi_{k_1} )^2 (T^2 Y\psi_{k_2} )^2  + D (T\psi_{k_1} )^2 (T^3 Y\psi_{k_2} )^2 + $$ $$ + D | T^3 \slashed{\nabla} \psi_{k_1}|^2  |\slashed{\nabla} \psi_{k_2}|^2 + D | T^2 \slashed{\nabla} \psi_{k_1}|^2 |T\slashed{\nabla} \psi_{k_2} |^2 . $$
We denote as follows the terms of the above pointwise inequality after integrating them over $S_{\tau} \cap \calc_{\tau_1}^{\tau_2}$
$$ \sum_{k \mik 5}\int_{S_{\tau}\cap \calc_{\tau_1}^{\tau_2}}  |\Omega^k T^3 F|^2 d\mi_{g_S} \lesssim I_{B4} + II_{B4} + III_{B4} + IV_{B4} + V_{B4} + VI_{B4} + VII_{B4} + VIII_{B4} . $$

We have that
$$ I_{B4} = \sum_{k_1 + k_2 \mik 5} \int_{S_{\tau}\cap \calc_{\tau_1}^{\tau_2}}  D^3 (T^3 Y \psi_{k_1} )^2 (Y\psi_{k_2} )^2 d\mi_{g_S} \lesssim \frac{\ee^2}{(1+\tau )^2} \int_{S_{\tau}\cap \calc_{\tau_1}^{\tau_2}}  D^3 (T^3 Y \psi_{k_1} )^2 d\mi_{g_S} \lesssim $$ $$ \lesssim \frac{E_0 \ee^2}{(1+\tau )^2} \sum_{i , m_i \mik 5} \int_{S_{\tau}\cap \calc_{\tau_1}^{\tau_2}}  J^T_{\mu} [T^3 \psi_{m_i} ] n^{\mu} d\mi_{g_S} \lesssim \frac{E_0^2 \ee^4}{(1+\tau )^{3-\aaa}} , $$
where we used Sobolev when needed, \eqref{dgt1away}, \eqref{dgt2away} and \eqref{denp4}.

For $II_{B4}$ we have that
$$ II_{B4} = \sum_{k_1 + k_2 \mik 5} \int_{S_{\tau}\cap \calc_{\tau_1}^{\tau_2}}  D^3 (T^2 Y \psi_{k_1} )^2 (TY\psi_{k_2} )^2 d\mi_{g_S} \lesssim \frac{E_0 \ee^2}{(1+\tau )^{2-\aaa}} \sum_{i , m_i \mik 5} \int_{S_{\tau}\cap \calc_{\tau_1}^{\tau_2}}  D^3 (T^2 Y \psi_{m_i} )^2 d\mi_{g_S} \lesssim $$ $$ \lesssim \frac{E_0 \ee^2}{(1+\tau )^{2-\aaa}} \sum_{i , m_i \mik 5} \int_{S_{\tau}\cap \calc_{\tau_1}^{\tau_2}}  J^T_{\mu} [T^2 \psi_{m_i} ] n^{\mu} d\mi_{g_S} \lesssim \frac{E_0^2 \ee^4}{(1+\tau )^{4-2\aaa}} , $$
where we used Sobolev when needed, \eqref{dgtt1away}, \eqref{dgtt2away} and \eqref{denp3}.

For $III_{B4}$ we have that
$$ III_{B4} = \sum_{k_1 + k_2 \mik 5} \int_{S_{\tau}\cap \calc_{\tau_1}^{\tau_2}}  D (T^4 \psi_{k_1} )^2 (Y\psi_{k_2} )^2 d\mi_{g_S} \lesssim \frac{E_0 \ee^2}{(1+\tau )^{2}} \sum_{i , m_i \mik 5} \int_{S_{\tau}\cap \calc_{\tau_1}^{\tau_2}}  D (T^4  \psi_{m_i} )^2 d\mi_{g_S} \lesssim $$ $$ \lesssim \frac{E_0 \ee^2}{(1+\tau )^{2}} \sum_{i , m_i \mik 5} \int_{S_{\tau}\cap \calc_{\tau_1}^{\tau_2}}  J^T_{\mu} [T^3 \psi_{m_i} ] n^{\mu} d\mi_{g_S} \lesssim \frac{E_0^2 \ee^4}{(1+\tau )^{3-\aaa}} , $$
where we used Sobolev when needed, \eqref{dgt1away}, \eqref{dgt2away} and \eqref{denp4}.

For $IV_{B4}$ we have that
$$ IV_{B4} = \sum_{k_1 + k_2 \mik 5} \int_{S_{\tau}\cap \calc_{\tau_1}^{\tau_2}}  D (T^3 \psi_{k_1} )^2 (TY\psi_{k_2} )^2 d\mi_{g_S} \lesssim \frac{E_0 \ee^2}{(1+\tau )^{2-\aaa}} \sum_{i , m_i \mik 5} \int_{S_{\tau}\cap \calc_{\tau_1}^{\tau_2}}  D (T^3 \psi_{m_i} )^2 d\mi_{g_S} \lesssim $$ $$ \lesssim \frac{E_0 \ee^2}{(1+\tau )^{2-\aaa}} \sum_{i , m_i \mik 5} \int_{S_{\tau}\cap \calc_{\tau_1}^{\tau_2}}  J^T_{\mu} [T^2 \psi_{m_i} ] n^{\mu} d\mi_{g_S} \lesssim \frac{E_0^2 \ee^4}{(1+\tau )^{4-2\aaa}} , $$
where we used Sobolev when needed, \eqref{dgtt1away}, \eqref{dgtt2away} and \eqref{denp3}.

For $V_{B4}$ we have that
$$ V_{B4} = \sum_{k_1 + k_2 \mik 5} \int_{S_{\tau}\cap \calc_{\tau_1}^{\tau_2}}  D (T^2 \psi_{k_1} )^2 (T^2 Y\psi_{k_2} )^2 d\mi_{g_S} \lesssim \frac{E_0 \ee^2}{(1+\tau )^{2-\aaa}} \sum_{i , m_i \mik 5} \int_{S_{\tau}\cap \calc_{\tau_1}^{\tau_2}}  D (T^2 Y \psi_{m_i} )^2 d\mi_{g_S} \lesssim $$ $$ \lesssim \frac{E_0 \ee^2}{(1+\tau )^{2-\aaa}} \sum_{i , m_i \mik 5} \int_{S_{\tau}\cap \calc_{\tau_1}^{\tau_2}}  J^T_{\mu} [T^2 \psi_{m_i} ] n^{\mu} d\mi_{g_S} \lesssim \frac{E_0^2 \ee^4}{(1+\tau )^{4-2\aaa}} , $$
where we used Sobolev when needed, \eqref{dgtt1away}, \eqref{dgtt2away} and \eqref{denp3}.

For $VI_{B4}$ we have that
$$ VI_{B4} = \sum_{k_1 + k_2 \mik 5} \int_{S_{\tau}\cap \calc_{\tau_1}^{\tau_2}}  D (T \psi_{k_1} )^2 (T^3 Y\psi_{k_2} )^2 d\mi_{g_S} \lesssim \frac{E_0 \ee^2}{(1+\tau )^{2}} \sum_{i , m_i \mik 5} \int_{S_{\tau}\cap \calc_{\tau_1}^{\tau_2}}  D (T^3 Y \psi_{m_i} )^2 d\mi_{g_S} \lesssim $$ $$ \lesssim \frac{E_0 \ee^2}{(1+\tau )^{2}} \sum_{i , m_i \mik 5} \int_{S_{\tau}\cap \calc_{\tau_1}^{\tau_2}}  J^T_{\mu} [T^3 \psi_{m_i} ] n^{\mu} d\mi_{g_S} \lesssim \frac{E_0^2 \ee^4}{(1+\tau )^{3-\aaa}} , $$
where we used Sobolev when needed, \eqref{dgt1away}, \eqref{dgt2away} and \eqref{denp4}.

For $VII_{B4}$ we have that
$$ VII_{B4} = \sum_{k_1 + k_2 \mik 5} \int_{S_{\tau}\cap \calc_{\tau_1}^{\tau_2}}  D |T^3 \slashed{\nabla} \psi_{k_1} |^2 |\slashed{\nabla}\psi_{k_2} |^2 d\mi_{g_S} \lesssim $$ $$ \lesssim \frac{E_0 \ee^2}{(1+\tau )^{2}} \sum_{i , m_i \mik 5} \int_{S_{\tau}\cap \calc_{\tau_1}^{\tau_2}}  D |T^3 \slashed{\nabla} \psi_{m_i} |^2 d\mi_{g_S} + \frac{E_0 \ee^2}{(1+\tau )^{1-\aaa}} \sum_{i , m_i \mik 5}\int_{S_{\tau}\cap \calc_{\tau_1}^{\tau_2}}  D | \slashed{\nabla} \psi_{m_i} |^2 d\mi_{g_S} \lesssim $$ $$ \lesssim \frac{E_0 \ee^2}{(1+\tau )^{2}} \sum_{i , m_i \mik 5} \int_{S_{\tau}\cap \calc_{\tau_1}^{\tau_2}}  J^T_{\mu} [T^3 \psi_{m_i} ] n^{\mu} d\mi_{g_S} + \frac{E_0 \ee^2}{(1+\tau )^{1-\alpha}} \sum_{i , m_i \mik 5} \int_{S_{\tau}\cap \calc_{\tau_1}^{\tau_2}}  J^T_{\mu} [ \psi_{m_i} ] n^{\mu} d\mi_{g_S}  \lesssim \frac{E_0^2 \ee^4}{(1+\tau )^{3-\aaa}} , $$
where we used Sobolev when needed, \eqref{dg1away}, \eqref{dg2away}, \eqref{dgttt1away}, \eqref{dgttt2away}, \eqref{denp1} and \eqref{denp4}.

For $VIII_{B4}$ we have that
$$ VIII_{B4} = \sum_{k_1 + k_2 \mik 5} \int_{S_{\tau}\cap \calc_{\tau_1}^{\tau_2}}  D |T^2 \slashed{\nabla} \psi_{k_1} |^2 |T\slashed{\nabla}\psi_{k_2} |^2 d\mi_{g_S} \lesssim $$ $$ \lesssim \frac{E_0 \ee^2}{(1+\tau )^{2}} \sum_{i , m_i \mik 5} \int_{S_{\tau}\cap \calc_{\tau_1}^{\tau_2}}  D |T^2 \slashed{\nabla} \psi_{m_i} |^2 d\mi_{g_S} + \frac{E_0 \ee^2}{(1+\tau )^{1-\aaa}} \sum_{i , m_i \mik 5}\int_{S_{\tau}\cap \calc_{\tau_1}^{\tau_2}}  D | T\slashed{\nabla} \psi_{m_i} |^2 d\mi_{g_S} \lesssim $$ $$ \lesssim \frac{E_0 \ee^2}{(1+\tau )^{2}} \sum_{i , m_i \mik 5} \int_{S_{\tau}\cap \calc_{\tau_1}^{\tau_2}}  J^T_{\mu} [T^2 \psi_{m_i} ] n^{\mu} d\mi_{g_S} + \frac{E_0 \ee^2}{(1+\tau )^{1-\aaa}} \sum_{i , m_i \mik 5} \int_{S_{\tau}\cap \calc_{\tau_1}^{\tau_2}}  J^T_{\mu} [ T\psi_{m_i} ] n^{\mu} d\mi_{g_S}  \lesssim \frac{E_0^2 \ee^4}{(1+\tau )^{3-\aaa}} , $$
where we used Sobolev when needed, \eqref{dgt1away}, \eqref{dgt2away}, \eqref{dgtt1away}, \eqref{dgtt2away}, \eqref{denp2} and \eqref{denp3}.

Gathering together all the previous estimates, we get that in the end
$$ \int_{S_{\tau}\cap \calc_{\tau_1}^{\tau_2}} |\Omega^k T^3 F|^2 d\mi_{g_S} \lesssim \frac{E_0^2 \ee^4}{(1+\tau )^{3-\aaa}} \mbox{ for any $k \mik 5$} , $$
which implies the desired estimate \eqref{B4'}.

\eqref{C1'}: We deal again the same terms from \eqref{B4'}, but this time we integrate them on $S$. We have that
$$ I_{C1} = \sum_{k_1 + k_2 \mik 5} \int_{S}  D^3 (T^3 Y \psi_{k_1} )^2 (Y\psi_{k_2} )^2 d\mi_{g_S} \lesssim \frac{E_0 \ee^2}{1+\tau } \int_S D (T^3 Y \psi_{k_1} )^2 d\mi_{g_S} \lesssim $$ $$ \lesssim \frac{E_0 \ee^2}{1+\tau } \sum_{i , m_i \mik 5} \int_S  J^T_{\mu} [T^3 \psi_{m_i} ] n^{\mu} d\mi_{g_S} \lesssim \frac{E_0^2 \ee^4}{(1+\tau )^{2-\aaa}} , $$
where we used Sobolev when needed, \eqref{d2yg}, \eqref{d2yg1} and \eqref{denp4}.

For $II_{C1}$ we have that
$$ II_{C1} = \sum_{k_1 + k_2 \mik 5} \int_S  D^3 (T^2 Y \psi_{k_1} )^2 (TY\psi_{k_2} )^2 d\mi_{g_S} \lesssim \frac{E_0 \ee^2}{1+\tau } \sum_{i , m_i \mik 5} \int_S  D (T^2 Y \psi_{m_i} )^2 d\mi_{g_S} \lesssim $$ $$ \lesssim \frac{E_0 \ee^2}{1+\tau } \sum_{i , m_i \mik 5} \int_S J^T_{\mu} [T^2 \psi_{m_i} ] n^{\mu} d\mi_{g_S} \lesssim \frac{E_0^2 \ee^4}{(1+\tau )^{3-\aaa}} , $$
where we used Sobolev when needed, \eqref{d2tyg}, \eqref{d2tyg1} and \eqref{denp3}.

For $III_{C1}$ we have that
$$ III_{C1} = \sum_{k_1 + k_2 \mik 5} \int_S  D (T^4 \psi_{k_1} )^2 (Y\psi_{k_2} )^2 d\mi_{g_S} \lesssim \frac{E_0 \ee^2}{ (1+\tau )^{1/2}} \sum_{i , m_i \mik 5} \int_S (T^4  \psi_{m_i} )^2 d\mi_{g_S} \lesssim $$ $$ \lesssim \frac{E_0 \ee^2}{ (1+\tau )^{1/2}} \sum_{i , m_i \mik 5} \int_S  J^T_{\mu} [T^3 \psi_{m_i} ] n^{\mu} d\mi_{g_S} \lesssim \frac{E_0^2 \ee^4}{(1+\tau )^{3/2-\aaa}} , $$
where we used Sobolev when needed, \eqref{dyg}, \eqref{dyg1} and \eqref{denp4}.

For $IV_{C1}$ we have that
$$ IV_{C1} = \sum_{k_1 + k_2 \mik 5} \int_S  D (T^3 \psi_{k_1} )^2 (TY\psi_{k_2} )^2 d\mi_{g_S} \lesssim \frac{E_0 \ee^2}{(1+\tau )^{1/2}} \sum_{i , m_i \mik 5} \int_S   (T^3 \psi_{m_i} )^2 d\mi_{g_S} \lesssim $$ $$ \lesssim \frac{E_0 \ee^2}{(1+\tau )^{1/2}} \sum_{i , m_i \mik 5} \int_S J^T_{\mu} [T^2 \psi_{m_i} ] n^{\mu} d\mi_{g_S} \lesssim \frac{E_0^2 \ee^4}{(1+\tau )^{5/2-\aaa}} , $$
where we used Sobolev when needed, \eqref{dtyg}, \eqref{dtyg1} and \eqref{denp3}.

For $V_{C1}$ we have that
$$ V_{C1} = \sum_{k_1 + k_2 \mik 5} \int_S  D (T^2 \psi_{k_1} )^2 (T^2 Y\psi_{k_2} )^2 d\mi_{g_S} \lesssim \frac{E_0 \ee^2}{(1+\tau )^{1-\aaa} } \sum_{i , m_i \mik 5} \int_S  D (T^2 Y \psi_{m_i} )^2 d\mi_{g_S} \lesssim $$ $$ \lesssim \frac{E_0 \ee^2}{1+\tau} \sum_{i , m_i \mik 5} \int_S  J^T_{\mu} [T^2 \psi_{m_i} ] n^{\mu} d\mi_{g_S} \lesssim \frac{E_0^2 \ee^4}{(1+\tau )^{3-2\aaa}} , $$
where we used Sobolev when needed, \eqref{dgtt}, \eqref{dgtt1} and \eqref{denp3}.

For $VI_{C1}$ we have that
$$ VI_{C1} = \sum_{k_1 + k_2 \mik 5} \int_S  D (T \psi_{k_1} )^2 (T^3 Y\psi_{k_2} )^2 d\mi_{g_S} \lesssim \frac{E_0 \ee^2}{1+\tau } \sum_{i , m_i \mik 5} \int_S D (T^3 Y \psi_{m_i} )^2 d\mi_{g_S} \lesssim $$ $$ \lesssim \frac{E_0 \ee^2}{1+\tau } \sum_{i , m_i \mik 5} \int_S  J^T_{\mu} [T^3 \psi_{m_i} ] n^{\mu} d\mi_{g_S} \lesssim \frac{E_0^2 \ee^4}{(1+\tau )^{2-\aaa}} , $$
where we used Sobolev when needed, \eqref{dgt1away}, \eqref{dgt2away} and \eqref{denp4}.

For $VII_{C1}$ we have that
$$ VII_{C1} = \sum_{k_1 + k_2 \mik 5} \int_S D |T^3 \slashed{\nabla} \psi_{k_1} |^2 |\slashed{\nabla}\psi_{k_2} |^2 d\mi_{g_S} \lesssim $$ $$ \lesssim \frac{E_0 \ee^2}{1+\tau } \sum_{i , m_i \mik 5} \int_S  D |T^3 \slashed{\nabla} \psi_{m_i} |^2 d\mi_{g_S} + \frac{E_0 \ee^2}{(1+\tau )^{1-\aaa}} \sum_{i , m_i \mik 5}\int_S   | \slashed{\nabla} \psi_{m_i} |^2 d\mi_{g_S} \lesssim $$ $$ \lesssim \frac{E_0 \ee^2}{1+\tau } \sum_{i , m_i \mik 5} \int_S  J^T_{\mu} [T^3 \psi_{m_i} ] n^{\mu} d\mi_{g_S} + \frac{E_0 \ee^2}{(1+\tau )^{1-\aaa}} \sum_{i , m_i \mik 5} \int_S  J^T_{\mu} [ \psi_{m_i} ] n^{\mu} d\mi_{g_S}  \lesssim \frac{E_0^2 \ee^4}{(1+\tau )^{2-\aaa}} , $$
where we used Sobolev when needed, \eqref{dg}, \eqref{dg1}, \eqref{dgttt}, \eqref{dgttt1}, \eqref{denp1} and \eqref{denp4}.

For $VIII_{C1}$ we have that
$$ VIII_{C1} = \sum_{k_1 + k_2 \mik 5} \int_S  D |T^2 \slashed{\nabla} \psi_{k_1} |^2 |T\slashed{\nabla}\psi_{k_2} |^2 d\mi_{g_S} \lesssim $$ $$ \lesssim \frac{E_0 \ee^2}{1+\tau } \sum_{i , m_i \mik 5} \int_S  D |T^2 \slashed{\nabla} \psi_{m_i} |^2 d\mi_{g_S} + \frac{E_0 \ee^2}{ (1+\tau )^{2-\aaa} } \sum_{i , m_i \mik 5}\int_S   | T\slashed{\nabla} \psi_{m_i} |^2 d\mi_{g_S} \lesssim $$ $$ \lesssim \frac{E_0 \ee^2}{1+\tau } \sum_{i , m_i \mik 5} \int_S  J^T_{\mu} [T^2 \psi_{m_i} ] n^{\mu} d\mi_{g_S} + \frac{E_0 \ee^2}{(1+\tau )^{2-\aaa} } \sum_{i , m_i \mik 5} \int_S  J^T_{\mu} [ T\psi_{m_i} ] n^{\mu} d\mi_{g_S}  \lesssim \frac{E_0^2 \ee^4}{(1+\tau )^{3-\aaa}} , $$
where we used Sobolev when needed, \eqref{dgt}, \eqref{dgt1}, \eqref{d2gtt}, \eqref{d2gtt1}, \eqref{denp2} and \eqref{denp3}.

All the above estimates imply that
$$ \int_S |\Omega^k T^3 F |^2 d\mi_{g_S} \lesssim \frac{E_0^2 \ee^4}{(1+\tau )^{2-\aaa}} \mbox{ for any $k \mik 5$}. $$

\eqref{C2'}: Apart from the term $\sum_{k_1 + k_2 \mik 5} \int_S  D (T^4 \psi_{k_1} )^2 (Y\psi_{k_2} )^2 d\mi_{g_S}$ all the other terms in the proof of \eqref{C1'} have decay of rate $-2+\aaa$, which enough in order to prove \eqref{C2'}. But for the aforementioned problematic term we have that
$$ \sum_{k_1 + k_2 \mik 5} \int_{\tau_1}^{\tau_2} \int_{S_{\tau'}}  D (T^4 \psi_{k_1} )^2 (Y\psi_{k_2} )^2 d\mi_{g_{\rrr}} = $$ $$ = \sum_{k_1 + k_2 \mik 5} \int_{\cala_{\tau_1}^{\tau_2}} D (T^4 \psi_{k_1} )^2 (Y\psi_{k_2} )^2 d\mi_{g_{\cala}} + \sum_{k_1 + k_2 \mik 5} \int_{\tau_1}^{\tau_2} \int_{S_{\tau'} \cap (\cala_{\tau_1}^{\tau_2} )^c}  D (T^4 \psi_{k_1} )^2 (Y\psi_{k_2} )^2 d\mi_{g_{\rrr}} . $$
The second term $ \sum_{k_1 + k_2 \mik 5} \int_{\tau_1}^{\tau_2} \int_{S_{\tau'} \cap (\cala_{\tau_1}^{\tau_2} )^c}  D (T^4 \psi_{k_1} )^2 (Y\psi_{k_2} )^2 d\mi_{g_{\rrr}}$ has decay of rate $-2+\aaa$ by the proof of \eqref{B4'}. For $\sum_{k_1 + k_2 \mik 5} \int_{\cala_{\tau_1}^{\tau_2}} D (T^4 \psi_{k_1} )^2 (Y\psi_{k_2} )^2 d\mi_{g_{\cala}}$ we have that
$$ \sum_{k_1 + k_2 \mik 5} \int_{\cala_{\tau_1}^{\tau_2}} D (T^4 \psi_{k_1} )^2 (Y\psi_{k_2} )^2 d\mi_{g_{\cala}}  \lesssim \frac{E_0 \ee^2}{(1+\tau_1 )^{1/2}} \sum_{i , m_i \mik 5} \int_{\cala_{\tau_1}^{\tau_2}} (T^4 \psi_{m_i} )^2  d\mi_{g_{\cala}} \lesssim \frac{E_0^2 \ee^4}{(1+\tau_1 )^{3/2-\aaa}} , $$
where we used Sobolev when needed, \eqref{dyg}, \eqref{dyg1} and the estimate 
$$ \int_{\cala_{\tau_1}^{\tau_2}} (T^4 \psi_{k} )^2  d\mi_{g_{\cala}} \lesssim \frac{E_0 \ee^2}{(1+\tau_1 )^{1-\aaa}} \mbox{ for any $k\mik 5$}, $$
which is a consequence of \eqref{denp4}.

\eqref{C3'}: Again, the proof is very similar to the proofs of \textbf{A2'} and \textbf{B3'} and will not be repeated.

\eqref{C4'}: We have that
$$ \Omega^k T^4 F = \sum_{k_1 + k_2 = k } 2 D^{3/2} (T^4 Y\psi_{k_1} ) \cdot (Y\psi_{k_2} ) + 8D^{3/2} (T^3 Y\psi_{k_1} ) \cdot (TY\psi_{k_2} ) + 6D^{3/2} (T^2 Y\psi_{k_1} ) \cdot (T^2 Y\psi_{k_2} ) + $$ $$ + \sqrt{D} (T^5 \psi_{k_1} ) \cdot (Y\psi_{k_2} ) + 4\sqrt{D} (T^4 \psi_{k_1} ) \cdot (TY\psi_{k_2} ) + 6 \sqrt{D} (T^3 \psi_{k_1} ) \cdot (T^2 Y\psi_{k_2} ) + 4\sqrt{D} (T^2 \psi_{k_1} ) \cdot (T^3 Y\psi_{k_2} ) + $$ $$ + \sqrt{D} (T\psi_{k_1} ) \cdot (T^4 Y\psi_{k_2} ) + 2\sqrt{D} \langle T^4 \slashed{\nabla} \psi_{k_1} , \slashed{\nabla} \psi_{k_2} \rangle + 8\sqrt{D} \langle T^3 \slashed{\nabla} \psi_{k_1} , T\slashed{\nabla} \psi_{k_2} \rangle  + 6\sqrt{D} \langle T^2 \slashed{\nabla} \psi_{k_1} , T^2 \slashed{\nabla} \psi_{k_2} \rangle , $$
which implies that
$$ |\Omega^k T^4 F|^2 \lesssim \sum_{k_1 + k_2 = k }  D^3 (T^4 Y\psi_{k_1} )^2 (Y\psi_{k_2} )^2 + D^3 (T^3 Y\psi_{k_1} )^2 (TY\psi_{k_2} )^2 + D^3 (T^2 Y\psi_{k_1} )^2 (T^2 Y\psi_{k_2} )^2 + $$ $$ + D (T^5 \psi_{k_1} )^2 (Y\psi_{k_2} )^2 +D (T^4 \psi_{k_1} )^2 (TY\psi_{k_2} )^2 + D (T^3 \psi_{k_1} )^2 (T^2 Y\psi_{k_2} )^2 + D (T^2 \psi_{k_1} )^2 (T^3 Y\psi_{k_2} )^2 + $$ $$ + D(T\psi_{k_1} )^2 (T^4 Y\psi_{k_2} )^2 + D | T^4 \slashed{\nabla} \psi_{k_1}|^2 | \slashed{\nabla} \psi_{k_2} |^2 + D |T^3 \slashed{\nabla} \psi_{k_1}|^2 |T\slashed{\nabla} \psi_{k_2}|^2  + D| T^2 \slashed{\nabla} \psi_{k_1}|^2 |T^2 \slashed{\nabla} \psi_{k_2} |^2 . $$
After integrating over $S \cap \calc$ we denote the above terms as follows
$$ \sum_{k \mik 5} \int_{S_{\tau} \cap \calc_{\tau_1}^{\tau_2}} |\Omega^k T^4 F|^2 d\mi_{g_S} \lesssim I_{C4} + II_{C4} + III_{C4} + IV_{C4} + V_{C4} + VI_{C4} + VII_{C4} + VIII_{C4} + IX_{C4} + X_{C4} + XI_{C4} . $$
We have that
$$ I_{C4} = \sum_{k_1 + k_2 \mik 5} \int_{S_{\tau} \cap \calc_{\tau_1}^{\tau_2}} D^3 (T^4 Y \psi_{k_1} )^2 (Y\psi_{k_2} )^2  d\mi_{g_S}  \lesssim \frac{E_0 \ee^2}{(1+\tau )^2} \sum_{i , m_i \mik 5} \int_{S_{\tau} \cap \calc_{\tau_1}^{\tau_2}} D^3 (T^4 Y \psi_{m_i} )^2  d\mi_{g_S} \lesssim $$ $$ \lesssim \frac{E_0 \ee^2}{(1+\tau )^2} \int_{S_{\tau} \cap \calc_{\tau_1}^{\tau_2}} J^T_{\mu} [T^4 \psi_{m_i} ] n^{\mu}  d\mi_{g_S} \lesssim \frac{E_0^2 \ee^4}{(1+\tau )^{2-\aaa}} , $$
where we used Sobolev when needed, \eqref{dgt1away}, \eqref{dgt2away} and \eqref{denp5}.

For $II_{C4}$ we have that
$$ II_{C4} = \sum_{k_1 + k_2 \mik 5} \int_{S_{\tau} \cap \calc_{\tau_1}^{\tau_2}} D^3 (T^3 Y \psi_{k_1} )^2 (TY\psi_{k_2} )^2  d\mi_{g_S}  \lesssim \frac{E_0 \ee^2}{(1+\tau )^{2-\aaa}} \sum_{i , m_i \mik 5} \int_{S_{\tau} \cap \calc_{\tau_1}^{\tau_2}} D^3 (T^3 Y \psi_{m_i} )^2  d\mi_{g_S} \lesssim $$ $$ \lesssim \frac{E_0 \ee^2}{(1+\tau )^{2-\aaa}} \int_{S_{\tau} \cap \calc_{\tau_1}^{\tau_2}} J^T_{\mu} [T^3 \psi_{m_i} ] n^{\mu}  d\mi_{g_S} \lesssim \frac{E_0^2 \ee^4}{(1+\tau )^{3-2\aaa}} , $$
where we used Sobolev when needed, \eqref{dgtt1away}, \eqref{dgtt2away} and \eqref{denp4}.

For $III_{C4}$ we have that
$$ III_{C4} = \sum_{k_1 + k_2 \mik 5} \int_{S_{\tau} \cap \calc_{\tau_1}^{\tau_2}} D^3 (T^2 Y \psi_{k_1} )^2 (T^2 Y\psi_{k_2} )^2  d\mi_{g_S}  \lesssim \frac{\ee^2}{(1+\tau )^{1-\aaa}} \sum_{i , m_i \mik 5} \int_{S_{\tau} \cap \calc_{\tau_1}^{\tau_2}} D^3 (T^2 Y \psi_{m_i} )^2  d\mi_{g_S} \lesssim $$ $$ \lesssim \frac{\ee^2}{(1+\tau )^{1-\aaa}} \int_{S_{\tau} \cap \calc_{\tau_1}^{\tau_2}} J^T_{\mu} [T^2 \psi_{m_i} ] n^{\mu}  d\mi_{g_S} \lesssim \frac{\ee^4}{(1+\tau )^{3-2\aaa}} , $$
where we used Sobolev when needed, \eqref{dgttt1away}, \eqref{dgttt2away} and \eqref{denp3}.

For $IV_{C4}$ we have that
$$ IV_{C4} = \sum_{k_1 + k_2 \mik 5} \int_{S_{\tau} \cap \calc_{\tau_1}^{\tau_2}} D (T^5 \psi_{k_1} )^2 ( Y\psi_{k_2} )^2  d\mi_{g_S}  \lesssim \frac{E_0 \ee^2}{(1+\tau )^{2}} \sum_{i , m_i \mik 5} \int_{S_{\tau} \cap \calc_{\tau_1}^{\tau_2}} D (T^5  \psi_{m_i} )^2  d\mi_{g_S} \lesssim $$ $$ \lesssim \frac{E_0 \ee^2}{(1+\tau )^{2}} \int_{S_{\tau} \cap \calc_{\tau_1}^{\tau_2}} J^T_{\mu} [T^4 \psi_{m_i} ] n^{\mu}  d\mi_{g_S} \lesssim \frac{E_0^2 \ee^4}{(1+\tau )^{2-\aaa}} , $$
where we used Sobolev when needed, \eqref{dgt1away}, \eqref{dgt2away} and \eqref{denp5}.

For $V_{C4}$ we have that
$$ V_{C4} = \sum_{k_1 + k_2 \mik 5} \int_{S_{\tau} \cap \calc_{\tau_1}^{\tau_2}} D (T^4 \psi_{k_1} )^2 ( TY\psi_{k_2} )^2  d\mi_{g_S}  \lesssim \frac{E_0 \ee^2}{(1+\tau )^{2-\aaa}} \sum_{i , m_i \mik 5} \int_{S_{\tau} \cap \calc_{\tau_1}^{\tau_2}} D (T^4  \psi_{m_i} )^2  d\mi_{g_S} \lesssim $$ $$ \lesssim \frac{E_0 \ee^2}{(1+\tau )^{2-\aaa}} \int_{S_{\tau} \cap \calc_{\tau_1}^{\tau_2}} J^T_{\mu} [T^3 \psi_{m_i} ] n^{\mu}  d\mi_{g_S} \lesssim \frac{E_0^2 \ee^4}{(1+\tau )^{3-2\aaa}} , $$
where we used Sobolev when needed, \eqref{dgtt1away}, \eqref{dgtt2away} and \eqref{denp4}.

For $VI_{C4}$ we have that
$$ VI_{C4} = \sum_{k_1 + k_2 \mik 5} \int_{S_{\tau} \cap \calc_{\tau_1}^{\tau_2}} D (T^3 \psi_{k_1} )^2 ( T^2 Y\psi_{k_2} )^2  d\mi_{g_S}  \lesssim \frac{E_0 \ee^2}{(1+\tau )^{1-\aaa}} \sum_{i , m_i \mik 5} \int_{S_{\tau} \cap \calc_{\tau_1}^{\tau_2}} D (T^2 Y \psi_{m_i} )^2  d\mi_{g_S} \lesssim $$ $$ \lesssim \frac{E_0 \ee^2}{(1+\tau )^{1-\aaa}} \int_{S_{\tau} \cap \calc_{\tau_1}^{\tau_2}} J^T_{\mu} [T^2 \psi_{m_i} ] n^{\mu}  d\mi_{g_S} \lesssim \frac{E_0^2 \ee^4}{(1+\tau )^{3-2\aaa}} , $$
where we used Sobolev when needed, \eqref{dgttt1away}, \eqref{dgttt2away} and \eqref{denp3}.

For $VII_{C4}$ we have that
$$ VII_{C4} = \sum_{k_1 + k_2 \mik 5} \int_{S_{\tau} \cap \calc_{\tau_1}^{\tau_2}} D (T^2 \psi_{k_1} )^2 ( T^3 Y\psi_{k_2} )^2  d\mi_{g_S}  \lesssim \frac{E_0 \ee^2}{(1+\tau )^{2-\aaa}} \sum_{i , m_i \mik 5} \int_{S_{\tau} \cap \calc_{\tau_1}^{\tau_2}} D (T^3 Y \psi_{m_i} )^2  d\mi_{g_S} \lesssim $$ $$ \lesssim \frac{E_0 \ee^2}{(1+\tau )^{2-\aaa}} \int_{S_{\tau} \cap \calc_{\tau_1}^{\tau_2}} J^T_{\mu} [T^3 \psi_{m_i} ] n^{\mu}  d\mi_{g_S} \lesssim \frac{E_0^2 \ee^4}{(1+\tau )^{3-2\aaa}} , $$
where we used Sobolev when needed, \eqref{dgtt1away}, \eqref{dgtt2away} and \eqref{denp4}.

For $VIII_{C4}$ we have that
$$ VIII_{C4} = \sum_{k_1 + k_2 \mik 5} \int_{S_{\tau} \cap \calc_{\tau_1}^{\tau_2}} D (T \psi_{k_1} )^2 ( T^4 Y\psi_{k_2} )^2  d\mi_{g_S}  \lesssim \frac{E_0 \ee^2}{(1+\tau )^{2}} \sum_{i , m_i \mik 5} \int_{S_{\tau} \cap \calc_{\tau_1}^{\tau_2}} D (T^4 Y \psi_{m_i} )^2  d\mi_{g_S} \lesssim $$ $$ \lesssim \frac{E_0 \ee^2}{(1+\tau )^{2}} \int_{S_{\tau} \cap \calc_{\tau_1}^{\tau_2}} J^T_{\mu} [T^4 \psi_{m_i} ] n^{\mu}  d\mi_{g_S} \lesssim \frac{E_0^2 \ee^4}{(1+\tau )^{2-\aaa}} , $$
where we used Sobolev when needed, \eqref{dgt1away}, \eqref{dgt2away} and \eqref{denp5}.

For $IX_{C4}$ we have that
$$ IX_{C4} = \sum_{k_1 + k_2 \mik 5} \int_{S_{\tau} \cap \calc_{\tau_1}^{\tau_2}} D |T^4 \slashed{\nabla} \psi_{k_1} |^2 | \slashed{\nabla} \psi_{k_2} |^2  d\mi_{g_S}  \lesssim $$ $$ \lesssim \frac{E_0 \ee^2}{(1+\tau )^{2}} \sum_{i , m_i \mik 5} \int_{S_{\tau} \cap \calc_{\tau_1}^{\tau_2}} D |T^4 \slashed{\nabla} \psi_{m_i} |^2  d\mi_{g_S} + E_0 \ee^2 (1+\tau )^{\aaa} \sum_{i , m_i \mik 5} \int_{S_{\tau} \cap \calc_{\tau_1}^{\tau_2}} D |\slashed{\nabla} \psi_{m_i} |^2  d\mi_{g_S} \lesssim $$ $$ \lesssim \frac{E_0 \ee^2}{(1+\tau )^{2}} \sum_{i , m_i \mik 5} \int_{S_{\tau} \cap \calc_{\tau_1}^{\tau_2}} J^T_{\mu} [T^4 \psi_{m_i} ] n^{\mu}  d\mi_{g_S} + E_0 \ee^2 (1+\tau )^{\aaa} \sum_{i , m_i \mik 5} \int_{S_{\tau} \cap \calc_{\tau_1}^{\tau_2}} J^T_{\mu} [\psi_{m_i} ] n^{\mu}  d\mi_{g_S} \lesssim \frac{E_0^2 \ee^4}{(1+\tau )^{2-\aaa}} , $$
where we used Sobolev when needed, \eqref{dg1away}, \eqref{dg2away}, \eqref{dgtttt}, \eqref{dgtttt1}, \eqref{denp1} and \eqref{denp5}.

For $X_{C4}$ we have that
$$ X_{C4} = \sum_{k_1 + k_2 \mik 5} \int_{S_{\tau} \cap \calc_{\tau_1}^{\tau_2}} D |T^3 \slashed{\nabla} \psi_{k_1} |^2 | T\slashed{\nabla} \psi_{k_2} |^2  d\mi_{g_S}  \lesssim $$ $$ \lesssim \frac{E_0 \ee^2}{(1+\tau )^{2}} \sum_{i , m_i \mik 5} \int_{S_{\tau} \cap \calc_{\tau_1}^{\tau_2}} D |T^3 \slashed{\nabla} \psi_{m_i} |^2  d\mi_{g_S} + \frac{E_0 \ee^2}{ (1+\tau )^{1-\aaa}} \sum_{i , m_i \mik 5} \int_{S_{\tau} \cap \calc_{\tau_1}^{\tau_2}} D |T \slashed{\nabla} \psi_{m_i} |^2  d\mi_{g_S} \lesssim $$ $$ \lesssim \frac{E_0 \ee^2}{(1+\tau )^{2}} \sum_{i , m_i \mik 5} \int_{S_{\tau} \cap \calc_{\tau_1}^{\tau_2}} J^T_{\mu} [T^3 \psi_{m_i} ] n^{\mu}  d\mi_{g_S} + \frac{E_0 \ee^2}{ (1+\tau )^{1-\aaa}} \sum_{i , m_i \mik 5} \int_{S_{\tau} \cap \calc_{\tau_1}^{\tau_2}} J^T_{\mu} [T\psi_{m_i} ] n^{\mu}  d\mi_{g_S} \lesssim \frac{E_0^2 \ee^4}{(1+\tau )^{3-\aaa}} , $$
where we used Sobolev when needed, \eqref{dgt1away}, \eqref{dgt2away}, \eqref{dgttt1away}, \eqref{dgttt2away}, \eqref{denp2} and \eqref{denp4}.

For $XI_{C4}$ we have that
$$ XI_{C4} = \sum_{k_1 + k_2 \mik 5} \int_{S_{\tau} \cap \calc_{\tau_1}^{\tau_2}} D |T^2 \slashed{\nabla} \psi_{k_1} |^2 | T^2 \slashed{\nabla} \psi_{k_2} |^2  d\mi_{g_S}  \lesssim $$ $$ \lesssim \frac{E_0 \ee^2}{( 1+\tau )^{2-\aaa} } \sum_{i , m_i \mik 5} \int_{S_{\tau} \cap \calc_{\tau_1}^{\tau_2}} D |T^2 \slashed{\nabla} \psi_{m_i} |^2  d\mi_{g_S} \lesssim \frac{E_0 \ee^2}{(1+\tau )^{2-\aaa}} \sum_{i , m_i \mik 5} \int_{S_{\tau} \cap \calc_{\tau_1}^{\tau_2}} J^T_{\mu} [T^2 \psi_{m_i} ] n^{\mu}  d\mi_{g_S} \lesssim \frac{E_0^2 \ee^4}{(1+\tau )^{4-2\aaa}} , $$
where we used Sobolev when needed, \eqref{dgt1away}, \eqref{dgt2away}, \eqref{dgtt1away}, \eqref{dgtt2away}, \eqref{denp2} and \eqref{denp4}.

Gathering all the above estimates we get that
$$ \int_{S_{\tau} \cap \calc_{\tau_1}^{\tau_2}}  | \Omega^k T^4 F |^2  d\mi_{g_S} \lesssim \frac{E_0^2 \ee^4}{(1+\tau )^{2-\aaa}} \mbox{ for any $k\mik 5$}, $$
which implies \eqref{C4'}.

\eqref{D1'}: We use the same pointwise estimate from the proof of \eqref{C4'}, but now we integrate just over $S$. We have that
$$ I_{D1} = \sum_{k_1 + k_2 \mik 5} \int_S D^3 (T^4 Y \psi_{k_1} )^2 (Y\psi_{k_2} )^2  d\mi_{g_S}  \lesssim \frac{E_0 \ee^2}{1+\tau } \sum_{i , m_i \mik 5} \int_S D (T^4 Y \psi_{m_i} )^2  d\mi_{g_S} \lesssim $$ $$ \lesssim \frac{E_0 \ee^2}{1+\tau } \int_S J^T_{\mu} [T^4 \psi_{m_i} ] n^{\mu}  d\mi_{g_S} \lesssim \frac{E_0^2 \ee^4}{(1+\tau )^{1-\aaa}} , $$
where we used Sobolev when needed, \eqref{d2yg}, \eqref{d2yg1} and \eqref{denp5}.

For $II_{D1}$ we have that
$$ II_{D1} = \sum_{k_1 + k_2 \mik 5} \int_S D^3 (T^3 Y \psi_{k_1} )^2 (TY\psi_{k_2} )^2  d\mi_{g_S}  \lesssim \frac{E_0 \ee^2}{1+\tau } \sum_{i , m_i \mik 5} \int_S D (T^3 Y \psi_{m_i} )^2  d\mi_{g_S} \lesssim $$ $$ \lesssim \frac{E_0 \ee^2}{1+\tau } \int_S J^T_{\mu} [T^3 \psi_{m_i} ] n^{\mu}  d\mi_{g_S} \lesssim \frac{E_0^2 \ee^4}{(1+\tau )^{2-\aaa}} , $$
where we used Sobolev when needed, \eqref{d2tyg}, \eqref{d2tyg1} and \eqref{denp4}.

For $III_{D1}$ we have that
$$ III_{D1} = \sum_{k_1 + k_2 \mik 5} \int_S D^3 (T^2 Y \psi_{k_1} )^2 (T^2 Y\psi_{k_2} )^2  d\mi_{g_S}  \lesssim \frac{E_0 \ee^2}{(1+\tau )^{1-\aaa/2}} \sum_{i , m_i \mik 5} \int_S D^3 (T^2 Y \psi_{m_i} )^2  d\mi_{g_S} \lesssim $$ $$ \lesssim \frac{E_0 \ee^2}{(1+\tau )^{1-\aaa/2}} \int_S J^T_{\mu} [T^2 \psi_{m_i} ] n^{\mu}  d\mi_{g_S} \lesssim \frac{E_0^2 \ee^4}{(1+\tau )^{3-3\aaa/2}} , $$
where we used Sobolev when needed, \eqref{d2ttyg}, \eqref{d2ttyg1} and \eqref{denp3}.

For $IV_{D1}$ we do not get enough decay just by integrating over $S$, we take the spacetime integral instead and we have that
$$ \int_{\tau_1}^{\tau_2} IV_{D1} d\tau' = \sum_{k_1 + k_2 \mik 5} \int_{\tau_1}^{\tau_2} \int_{S_{\tau'}} D (T^5 \psi_{k_1} )^2 ( Y\psi_{k_2} )^2  d\mi_{g_{\rrr}}  =$$ $$ = \int_{\cala_{\tau_1}^{\tau_2}}D (T^5 \psi_{k_1} )^2 ( Y\psi_{k_2} )^2  d\mi_{g_{\cala}} + \int_{\tau_1}^{\tau_2} \int_{S_{\tau'} \cap (\cala_{\tau_1}^{\tau_2} )^c} D (T^5 \psi_{k_1} )^2 ( Y\psi_{k_2} )^2  d\mi_{g_{\rrr}} .$$ 
For the second integral of the last line we notice that
$$ \sum_{k_1 + k_2 \mik 5} \int_{S_{\tau} \cap (\cala_{\tau_1}^{\tau_2} )^c} D (T^5 \psi_{k_1} )^2 ( Y\psi_{k_2} )^2  d\mi_{g_S} \lesssim \frac{E_0 \ee^2}{(1+\tau )^2} \sum_{i , m_i \mik 5} \int_{S_{\tau} \cap (\cala_{\tau_1}^{\tau_2} )^c} D (T^5 \psi_{k_1} )^2 d\mi_{g_S} \lesssim \frac{E_0^2 \ee^4}{(1+\tau )^{2-\aaa}} \Rightarrow $$ $$ \Rightarrow \int_{\tau_1}^{\tau_2} \int_{S_{\tau'} \cap (\cala_{\tau_1}^{\tau_2} )^c} D (T^5 \psi_{k_1} )^2 ( Y\psi_{k_2} )^2  d\mi_{g_{\rrr}} \lesssim \frac{E_0^2 \ee^4}{(1+\tau_1 )^{1-\aaa}} , $$
where we used Sobolev when needed, \eqref{dgt1away}, \eqref{dgt2away} and \eqref{denp5}.
For $\int_{\cala_{\tau_1}^{\tau_2}}D (T^5 \psi_{k_1} )^2 ( Y\psi_{k_2} )^2  d\mi_{g_{\cala}}$ we have that
$$ \int_{\cala_{\tau_1}^{\tau_2}} D (T^5 \psi_{k_1} )^2 ( Y\psi_{k_2} )^2  d\mi_{g_{\cala}}\lesssim \frac{E_0 \ee^2}{(1+\tau_1 )^{1/2}} \sum_{i , m_i \mik 5} \int_{\cala_{\tau_1}^{\tau_2}} (T^5  \psi_{m_i} )^2  d\mi_{g_S} \lesssim \frac{\ee^4 (1+\tau_2 )^{\aaa}}{(1+\tau_1 )^{1/2}} , $$
where we used Sobolev when needed, \eqref{dyg}, \eqref{dyg1} and 
$$ \sum_{i , m_i \mik 5} \int_{\cala_{\tau_1}^{\tau_2}} (T^5 \psi_{m_i} )^2 d\mi_{g_{\cala}} \lesssim E_0 \ee^2 (1+\tau_2 )^{\aaa} , $$
which is a consequence of \eqref{denp5}. So in the end we get that
\begin{equation}\label{IVD1}
 \int_{\tau_1}^{\tau_2} IV_{D1} d\tau' \lesssim \frac{E_0^2 \ee^4 (1+\tau_2 )^{\aaa}}{(1+\tau_1 )^{1/2}} \lesssim E_0^2 \ee^4 (1+\tau_2 )^{\aaa}. 
\end{equation}

For $V_{D1}$ we have that
$$ V_{D1} = \sum_{k_1 + k_2 \mik 5} \int_S D (T^4 \psi_{k_1} )^2 ( TY\psi_{k_2} )^2  d\mi_{g_S}  \lesssim \frac{E_0 \ee^2}{(1+\tau )^{1/2}} \sum_{i , m_i \mik 5} \int_S (T^4  \psi_{m_i} )^2  d\mi_{g_S} \lesssim $$ $$ \lesssim \frac{E_0 \ee^2}{(1+\tau )^{1/2}} \int_{S_{\tau} \cap \calc_{\tau_1}^{\tau_2}} J^T_{\mu} [T^3 \psi_{m_i} ] n^{\mu}  d\mi_{g_S} \lesssim \frac{E_0^2 \ee^4}{(1+\tau )^{3/2-\aaa}} , $$
where we used Sobolev when needed, \eqref{dtyg}, \eqref{dtyg1} and \eqref{denp4}.

For $VI_{D1}$ we have that
$$ VI_{D1} = \sum_{k_1 + k_2 \mik 5} \int_S D (T^3 \psi_{k_1} )^2 ( T^2 Y\psi_{k_2} )^2  d\mi_{g_S}  \lesssim \frac{E_0 \ee^2}{(1+\tau )^{1/2-\aaa/2}} \sum_{i , m_i \mik 5} \int_S  (T^3 \psi_{m_i} )^2  d\mi_{g_S} \lesssim $$ $$ \lesssim \frac{E_0 \ee^2}{(1+\tau )^{1/2-\aaa/2}} \int_S J^T_{\mu} [T^2 \psi_{m_i} ] n^{\mu}  d\mi_{g_S} \lesssim \frac{E_0^2 \ee^4}{(1+\tau )^{5/2-3\aaa/2}} , $$
where we used Sobolev when needed, \eqref{dttyg}, \eqref{dttyg1} and \eqref{denp3}.

For $VII_{D1}$ we have that
$$ VII_{D1} = \sum_{k_1 + k_2 \mik 5} \int_S D (T^2 \psi_{k_1} )^2 ( T^3 Y\psi_{k_2} )^2  d\mi_{g_S}  \lesssim \frac{E_0 \ee^2}{(1+\tau )^{1-\aaa}} \sum_{i , m_i \mik 5} \int_S D (T^3 Y \psi_{m_i} )^2  d\mi_{g_S} \lesssim $$ $$ \lesssim \frac{E_0 \ee^2}{(1+\tau )^{1-\aaa}} \int_S J^T_{\mu} [T^3 \psi_{m_i} ] n^{\mu}  d\mi_{g_S} \lesssim \frac{E_0^2 \ee^4}{(1+\tau )^{2-2\aaa}} , $$
where we used Sobolev when needed, \eqref{dgtt}, \eqref{dgtt1} and \eqref{denp4}.

For $VIII_{D1}$ we have that
$$ VIII_{D1} = \sum_{k_1 + k_2 \mik 5} \int_S D (T \psi_{k_1} )^2 ( T^4 Y\psi_{k_2} )^2  d\mi_{g_S}  \lesssim \frac{\ee^2}{1+\tau } \sum_{i , m_i \mik 5} \int_S D (T^4 Y \psi_{m_i} )^2  d\mi_{g_S} \lesssim $$ $$ \lesssim \frac{\ee^2}{1+\tau } \int_S J^T_{\mu} [T^4 \psi_{m_i} ] n^{\mu}  d\mi_{g_S} \lesssim \frac{\ee^4}{(1+\tau )^{1-\aaa}} , $$
where we used Sobolev when needed, \eqref{dgt}, \eqref{dgt1} and \eqref{denp5}.

For $IX_{D1}$ we have that
$$ IX_{D1} = \sum_{k_1 + k_2 \mik 5} \int_{S_{\tau} } D |T^4 \slashed{\nabla} \psi_{k_1} |^2 | \slashed{\nabla} \psi_{k_2} |^2  d\mi_{g_S}  \lesssim $$ $$ \lesssim \frac{E_0 \ee^2}{1+\tau } \sum_{i , m_i \mik 5} \int_{S_{\tau} } D |T^4 \slashed{\nabla} \psi_{m_i} |^2  d\mi_{g_S} + E_0 \ee^2 (1+\tau )^{\aaa} \sum_{i , m_i \mik 5} \int_{S_{\tau} }  |\slashed{\nabla} \psi_{m_i} |^2  d\mi_{g_S} \lesssim $$ $$ \lesssim \frac{E_0 \ee^2}{1+\tau } \sum_{i , m_i \mik 5} \int_{S_{\tau} } J^T_{\mu} [T^4 \psi_{m_i} ] n^{\mu}  d\mi_{g_S} + E_0 \ee^2 (1+\tau )^{\aaa} \sum_{i , m_i \mik 5} \int_{S_{\tau} \cap \calc_{\tau_1}^{\tau_2}} J^T_{\mu} [\psi_{m_i} ] n^{\mu}  d\mi_{g_S} \lesssim \frac{E_0^2 \ee^4}{(1+\tau )^{1-\aaa}} , $$
where we used Sobolev when needed, \eqref{dg1away}, \eqref{dg2away}, \eqref{dgtttt}, \eqref{dgtttt1}, \eqref{denp1} and \eqref{denp5}.

For $X_{D1}$ we have that
$$ X_{D1} = \sum_{k_1 + k_2 \mik 5} \int_{S_{\tau} } D |T^3 \slashed{\nabla} \psi_{k_1} |^2 | T\slashed{\nabla} \psi_{k_2} |^2  d\mi_{g_S}  \lesssim $$ $$ \lesssim \frac{E_0 \ee^2}{1+\tau } \sum_{i , m_i \mik 5} \int_{S_{\tau} } D |T^3 \slashed{\nabla} \psi_{m_i} |^2  d\mi_{g_S} + \frac{E_0 \ee^2}{ (1+\tau )^{1-\aaa}} \sum_{i , m_i \mik 5} \int_{S_{\tau} }  |T \slashed{\nabla} \psi_{m_i} |^2  d\mi_{g_S} \lesssim $$ $$ \lesssim \frac{E_0 \ee^2}{1+\tau } \sum_{i , m_i \mik 5} \int_{S_{\tau} } J^T_{\mu} [T^3 \psi_{m_i} ] n^{\mu}  d\mi_{g_S} + \frac{E_0 \ee^2}{ (1+\tau )^{1-\aaa}} \sum_{i , m_i \mik 5} \int_{S_{\tau}} J^T_{\mu} [T\psi_{m_i} ] n^{\mu}  d\mi_{g_S} \lesssim \frac{E_0^2 \ee^4}{(1+\tau )^{2-\aaa}} , $$
where we used Sobolev when needed, \eqref{dgt1away}, \eqref{dgt2away}, \eqref{dgttt1away}, \eqref{dgttt2away}, \eqref{denp2} and \eqref{denp4}.

For $XI_{D1}$ we have that
$$ XI_{D1} = \sum_{k_1 + k_2 \mik 5} \int_{S_{\tau} } D |T^2 \slashed{\nabla} \psi_{k_1} |^2 | T^2 \slashed{\nabla} \psi_{k_2} |^2  d\mi_{g_S}  \lesssim $$ $$ \lesssim \frac{E_0 \ee^2}{(1+\tau )^{1-\aaa} } \sum_{i , m_i \mik 5} \int_{S_{\tau}} D |T^2 \slashed{\nabla} \psi_{m_i} |^2  d\mi_{g_S} \lesssim \frac{E_0 \ee^2}{(1+\tau )^{1-\aaa}} \sum_{i , m_i \mik 5} \int_{S_{\tau} } J^T_{\mu} [T^2 \psi_{m_i} ] n^{\mu}  d\mi_{g_S} \lesssim \frac{E_0^2 \ee^4}{(1+\tau )^{2-2\aaa}} , $$
where we used Sobolev when needed, \eqref{dgt1away}, \eqref{dgt2away}, \eqref{dgttt1away}, \eqref{dgttt2away}, \eqref{denp2} and \eqref{denp4}.

Gathering all the above estimates we note that $|\Omega^k T^4 F |^2$ integrated on $S$ decays with rate $-1+\aaa$ (which is enough for the proof of \eqref{D1'}), apart from the term given in \eqref{IVD1} which was shown though to satisfy the necessary estimate. 

\eqref{D2'}: The same estimates that were obtained in the proof of \eqref{C4'} imply also that
$$ \int_{\si_{\tau} \cap \{ r \meg 2M - \delta \}} |\Omega^k T^4 F|^2 d\mi_{g_S} \lesssim \frac{\ee^4}{(1+\tau )^{2-\aaa}} \mbox{ for any $k\mik 5$}, $$
which in turn implies \eqref{D2'}.

\eqref{E1'}: We have that
$$ \Omega^k YF = \sum_{k_1 + k_2 = k} D^{1/2} D' (Y\psi_{k_1} ) \cdot (Y\psi_{k_2} ) + 2 D^{3/2} (Y\psi_{k_1} ) \cdot (Y^2 \psi_{k_2} ) + 2 \sqrt{D} (TY\psi_{k_1} ) \cdot (Y\psi_{k_2} ) + $$ $$ + 2\sqrt{D} (T\psi_{k_1} ) \cdot (Y^2 \psi_{k_2} ) + 2\sqrt{D} \langle Y\slashed{\nabla} \psi_{k_1} , \slashed{\nabla} \psi_{k_2} \rangle + \frac{M}{r^2} \cdot g^{\beta \gamma} \cdot \partial_{\beta} \psi_{k_1} \cdot \partial_{\gamma} \psi_{k_2} , $$
and after commuting $YF$ once with $T$ we get that
$$ \Omega^k TYF = \sum_{k_1 + k_2 = k} 2 D^{1/2} D' (TY\psi_{k_1} ) \cdot (Y\psi_{k_2} ) + 2 D^{3/2} (TY\psi_{k_1} ) \cdot (Y^2 \psi_{k_2} ) + 2 D^{3/2} (Y\psi_{k_1} ) \cdot (TY^2 \psi_{k_2} )+$$ $$ + 2 \sqrt{D} (T^2 Y\psi_{k_1} ) \cdot (Y\psi_{k_2} ) +2 \sqrt{D} (T Y\psi_{k_1} ) \cdot (TY\psi_{k_2} )  + 2\sqrt{D} (T^2 \psi_{k_1} ) \cdot (Y^2 \psi_{k_2} ) + $$ $$ + 2\sqrt{D} (T \psi_{k_1} ) \cdot (TY^2 \psi_{k_2} ) + 2\sqrt{D} \langle TY\slashed{\nabla} \psi_{k_1} , \slashed{\nabla} \psi_{k_2} \rangle +2\sqrt{D} \langle Y\slashed{\nabla} \psi_{k_1} , T\slashed{\nabla} \psi_{k_2} \rangle +\frac{M}{r^2} \cdot g^{\beta \gamma} \cdot \partial_{\beta} T\psi_{k_1} \cdot \partial_{\gamma} \psi_{k_2} , $$
Finally we commute once more with $T$
$$ \Omega^k T^2 YF =  \sum_{k_1 + k_2 = k} 2 D^{1/2} D' (T^2 Y\psi_{k_1} ) \cdot (Y\psi_{k_2} ) + 2 D^{1/2} D' (TY\psi_{k_1} ) \cdot (TY\psi_{k_2} ) + 2 D^{3/2} (T^2 Y\psi_{k_1} ) \cdot (Y^2 \psi_{k_2} ) + $$ $$ + 4 D^{3/2} (TY\psi_{k_1} ) \cdot (T Y^2 \psi_{k_2} ) +  2 D^{3/2} (Y\psi_{k_1} ) \cdot (T^2 Y^2 \psi_{k_2} ) + 2 \sqrt{D} (T^3 Y\psi_{k_1} ) \cdot (Y\psi_{k_2} )  + 6 \sqrt{D} (T^2 Y\psi_{k_1} ) \cdot (TY\psi_{k_2} )   + $$ $$ + 2\sqrt{D} (T^3 \psi_{k_1} ) \cdot (Y^2 \psi_{k_2} )  + 4\sqrt{D} (T^2 \psi_{k_1} ) \cdot (TY^2 \psi_{k_2} ) + 2\sqrt{D} (T \psi_{k_1} ) \cdot (T^2Y^2 \psi_{k_2} ) + 2\sqrt{D} \langle T^2 Y\slashed{\nabla} \psi_{k_1} , \slashed{\nabla} \psi_{k_2} \rangle + $$ $$  + 4\sqrt{D} \langle TY\slashed{\nabla} \psi_{k_1} , T\slashed{\nabla} \psi_{k_2} \rangle + 2\sqrt{D} \langle Y\slashed{\nabla} \psi_{k_1} , T^2 \slashed{\nabla} \psi_{k_2} \rangle  + \frac{M}{r^2} \cdot g^{\beta \gamma} \cdot \partial_{\beta} T^2 \psi_{k_1} \cdot \partial_{\gamma} \psi_{k_2} + \frac{M}{r^2} \cdot g^{\beta \gamma} \cdot \partial_{\beta} T \psi_{k_1} \cdot \partial_{\gamma} T\psi_{k_2} . $$
We note that the following inequalities hold true
$$ \sum_{k_1 + k_2 \mik 5} \int_{\cala_{\tau_1}^{\tau_2}}  \cdot \frac{M^2}{r^4} \left( g^{\beta \gamma} \cdot \partial_{\beta} \psi_{k_1} \cdot \partial_{\gamma} \psi_{k_2} \right)^2 d\mi_{g_{\cala}} \lesssim \sum_{k \mik 5} \int_{\cala_{\tau_1}^{\tau_2}} | \Omega^k F |^2 d\mi_{g_{\cala}} \lesssim \frac{E_0^2 \ee^4}{(1+\tau_1)^2} \mbox{ by \eqref{A1'}}, $$
 $$ \sum_{k_1 + k_2 \mik 5} \int_{\cala_{\tau_1}^{\tau_2}}  \cdot \frac{M^2}{r^4} \left( g^{\beta \gamma} \cdot \partial_{\beta} T\psi_{k_1} \cdot \partial_{\gamma} \psi_{k_2} \right)^2 d\mi_{g_{\cala}} \lesssim \sum_{k \mik 5} \int_{\cala_{\tau_1}^{\tau_2}} | \Omega^k TF |^2 d\mi_{g_{\cala}} \lesssim \frac{E_0^2 \ee^4}{(1+\tau_1)^2} \mbox{ by \eqref{A1'}}, $$
 and
 $$ \sum_{k_1 + k_2 \mik 5} \int_{\cala_{\tau_1}^{\tau_2}}  \cdot \frac{M^2}{r^4} \left( \left( g^{\beta \gamma} \cdot \partial_{\beta} T^2 \psi_{k_1} \cdot \partial_{\gamma} \psi_{k_2} \right) + \left( g^{\beta \gamma} \cdot \partial_{\beta} T \psi_{k_1} \cdot \partial_{\gamma} T\psi_{k_2} \right) \right)^2 d\mi_{g_{\cala}} \lesssim $$ $$ \lesssim \sum_{k \mik 5} \int_{\cala_{\tau_1}^{\tau_2}} | \Omega^k TF |^2 d\mi_{g_{\cala}} \lesssim \frac{E_0^2 \ee^4}{(1+\tau_1)^{2-\aaa}} \mbox{ by \eqref{B2'}}. $$
 So for $\Omega^k YF$, $\Omega^k TYF$ and $\Omega^k T^2 YF$ we focus on all the other terms.
 
For $\Omega^k YF$ we have that
$$ \sum_{k \mik 5} \int_{\cala_{\tau_1}^{\tau_2}}  |\Omega^k YF |^2 d\mi_{g_{\cala}} \lesssim I_{E1}a + II_{E1}a + III_{E1}a + IV_{E1}a + V_{E1}a + VI_{E_1}a+ \frac{E_0^2 \ee^4}{(1+\tau_1)^2} . $$

For $I_{E1}a$ we have that
$$ I_{E1}a = \sum_{k_1 + k_2 \mik 5} \int_{\cala_{\tau_1}^{\tau_2}} D^2 (Y\psi_{k_1} )^2 (Y\psi_{k_2} )^2 d\mi_{g_{\cala}} \Rightarrow $$ $$ \Rightarrow \sum_{k_1 + k_2 \mik 5} \int_{S_{\tau} \cap\cala_{\tau_1}^{\tau_2}} D^3 (Y\psi_{k_1} )^2 (Y\psi_{k_2} )^2 d\mi_{g_S} \lesssim \frac{E_0 \ee^2}{(1+\tau )^{1/2}} \sum_{i , m_i \mik 5} \int_{S_{\tau} \cap\cala_{\tau_1}^{\tau_2}} D (Y\psi_{m_i} )^2 d\mi_{g_S} \lesssim \frac{E_0^2 \ee^4}{(1+\tau )^{5/2}} ,$$
where we used Sobolev when needed, \eqref{dyg}, \eqref{dyg1} and \eqref{denp1}. As a consequence we get that
$$ I_{E1}a \lesssim \frac{E_0^2 \ee^4}{(1+\tau_1 )^{3/2}} . $$

For $II_{E1}a$ we have that
$$ II_{E1}a = \sum_{k_1 + k_2 \mik 5} \int_{\cala_{\tau_1}^{\tau_2}} D^3 (Y\psi_{k_1} )^2 (Y^2 \psi_{k_2} )^2 d\mi_{g_{\cala}} \lesssim \frac{E_0 \ee^2}{( 1+\tau_1 )^{1/2} } \sum_{i , m_i \mik 5}  \int_{\cala_{\tau_1}^{\tau_2}} D^2 (Y^2 \psi_{m_i} )^2 d\mi_{g_{\cala}} \Rightarrow $$ $$ II_{E_1}a \lesssim \frac{E_0^2 \ee^4}{( 1+\tau_1 )^{1/2}} , $$
where we used Sobolev when needed, \eqref{dyg}, \eqref{dyg1} and the estimate for $l=0$
\begin{equation}\label{ydenp2}
\int_{\cala_{\tau_1}^{\tau_2}} \left( D^{3/2} (\Omega^k T^l Y^2 \psi )^2 + |\slashed{\nabla} \Omega^k T^l Y \psi |^2 \right) \, d\mi_{g_{\cala}} \lesssim E_0 \ee^2 \mbox{ for any $k \mik 5$, any $l\mik 2$ and any $\tau_1$, $\tau_2$ with $\tau_1 < \tau_2$},
\end{equation}
which is a consequence of \eqref{ydenp1}.

For $III_{E1}a$ we have that
$$ \sum_{k_1 + k_2 \mik 5} \int_{\cala_{\tau_1}^{\tau_2}} D (Y\psi_{k_1} )^2 (TY\psi_{k_2} )^2 d\mi_{g_{\cala}} \lesssim \frac{E_0 \ee^2}{(1+\tau_1 )^{1/2}} \int_{\cala_{\tau_1}^{\tau_2}} (TY\psi_{k_2} )^2 d\mi_{g_{\cala}} \lesssim \frac{E_0^2 \ee^4}{(1+\tau_1 )^{1/2}},$$
where we used Sobolev when needed, \eqref{dyg}, \eqref{dyg1} and \eqref{d32tye}. 

For $IV_{E1}a$ we have that
$$ IV_{E1}a = \sum_{k_1 + k_2 \mik 5} \int_{\cala_{\tau_1}^{\tau_2}} D (T\psi_{k_1} )^2 (Y^2 \psi_{k_2} )^2 d\mi_{g_{\cala}} \Rightarrow $$ $$ \sum_{k_1 + k_2 \mik 5} \int_{S_{\tau} \cap \cala_{\tau_1}^{\tau_2}}D (T\psi_{k_1} )^2 (Y^2 \psi_{k_2} )^2 d\mi_{g_S} \lesssim \frac{E_0 \ee^2}{(1+\tau )^{1+\beta} } \sum_{i , m_i \mik 5}  \int_{S_{\tau} \cap \cala_{\tau_1}^{\tau_2}} D (Y^2 \psi_{m_i} )^2 d\mi_{g_S} \lesssim \frac{E_0^2 \ee^4}{( 1+\tau )^{1+\beta}} , $$
where we used Sobolev when needed, \eqref{dyg}, \eqref{dyg1} and the estimate\eqref{d32tye} for $l=0$. The above then implies that
$$  IV_{E1}a \lesssim E_0^2 \ee^4 . $$

For $V_{E1}a$ we have that
$$ V_{E1}a = \sum_{k_1 + k_2 \mik 5} \int_{\cala_{\tau_1}^{\tau_2}} (T \psi_{k_1} )^2 (Y \psi_{k_2} )^2 d \mi_{g_{\cala}} \Rightarrow $$ $$ \sum_{k_1 + k_2 \mik 5} \int_{S_{\tau} \cap \cala_{\tau_1}^{\tau_2}} (T \psi_{k_1} )^2 (Y \psi_{k_2} )^2 d \mi_{g_S} \lesssim \frac{E_0 \ee^2}{(1+\tau )^{1+\beta}} \sum_{i , m_i \mik 5}  \int_{S_{\tau} \cap \cala_{\tau_1}^{\tau_2}} (Y \psi_{m_i} )^2 d\mi_{g_S} \lesssim \frac{E_0^2 \ee^4}{(1+\tau )^{1+\beta}} \Rightarrow $$ $$ V_{E1}a \lesssim E_0^2 \ee^4 . $$

For $VI_{E1}a$ we have that
$$ VI_{E1}a = \sum_{k_1 + k_2 \mik 5} \int_{\cala_{\tau_1}^{\tau_2}} D |Y \slashed{\nabla}\psi_{k_1} |^2 | \slashed{\nabla} \psi_{k_2} |^2 d\mi_{g_{\cala}} \lesssim $$  $$ \lesssim \frac{E_0 \ee^2}{( 1+\tau_1 )^{1/2}} \sum_{i , m_i \mik 5} \int_{\cala_{\tau_1}^{\tau_2}} | \slashed{\nabla} \psi_{m_i} |^2 d\mi_{g_{\cala}} + \frac{E_0 \ee^2}{1+\tau_1 } \sum_{i , m_i \mik 5}  \int_{\cala_{\tau_1}^{\tau_2}} D | \slashed{\nabla} Y \psi_{m_i} |^2 d\mi_{g_{\cala}} \lesssim \frac{\ee^4}{1+\tau_1} \lesssim $$
$$ \lesssim \frac{E_0 \ee^2}{(1+\tau_1 )^{1/2} } \sum_{i , m_i \mik 5} \int_{\tau_1}^{\tau_2} \int_{S_{\tau} \cap \cala_{\tau_1}^{\tau_2}} J^T_{\mu} [\psi_{m_i} ] n^{\mu} d\mi_{g_{S}} d\tau + \frac{E_0 \ee^2}{1+\tau_1 } \sum_{i , m_i \mik 5}  \int_{\cala_{\tau_1}^{\tau_2}} D | \slashed{\nabla} Y \psi_{m_i} |^2 d\mi_{g_{\cala}} \lesssim \frac{E_0^2 \ee^4}{( 1+\tau_1 )^{1/2}} , $$
where we used Sobolev when needed, \eqref{dg}, \eqref{dg1}, \eqref{dyg}, \eqref{dyg1}, \eqref{denp1} and estimate\eqref{ydenp2} for $l=0$.

Now we look at $\Omega^k TYF$ and we have that
$$ \sum_{k \mik 5} \int_{\cala_{\tau_1}^{\tau_2}} |\Omega^k TYF |^2 d\mi_{g_{\cala}} \lesssim I_{E1}b + II_{E1}b + III_{E1}b + IV_{E1}b + V_{E1}b + $$ $$+ VI_{E1}b + VII_{E1}b + VIII_{E1}b +IX_{E1}b + X_{E1}b + \frac{E_0^2 \ee^4}{(1+\tau_1)^2} . $$

For $I_{E1}b$ we have that
$$ I_{E1}b = \sum_{k_1 + k_2 \mik 5}\int_{\cala_{\tau_1}^{\tau_2}} D^2 (TY\psi_{k_1} )^2 (Y\psi_{k_2} )^2 d\mi_{g_{\cala}} \Rightarrow  $$ $$ \Rightarrow \sum_{k_1 + k_2 \mik 5} \int_S D^3 (TY\psi_{k_1} )^2 (Y\psi_{k_2} )^2 d\mi_{g_S} \lesssim \frac{E_0 \ee^2}{(1+\tau )^{1/2}} \sum_{i , m_i \mik 5} \int_S D (TY\psi_{m_i} )^2 d\mi_{g_S} \lesssim \frac{E_0^2 \ee^4}{(1+\tau )^{5/2}} \Rightarrow $$ $$ \Rightarrow I_{E_1}b \lesssim \frac{E_0^2 \ee^4}{(1+\tau_1 )^{3/2}} , 
$$
where we used Sobolev when needed, \eqref{dyg}, \eqref{dyg1} and \eqref{denp2}.

For $II_{E1}b$ we have that
$$ II_{E1}b = \sum_{k_1 + k_2 \mik 5}\int_{\cala_{\tau_1}^{\tau_2}} D^3 (TY\psi_{k_1} )^2 (Y^2 \psi_{k_2} )^2 d\mi_{g_{\cala}} \lesssim \frac{E_0 \ee^2}{( 1+\tau_1 )^{1/2}} \sum_{i , m_i \mik 5} \int_{\cala_{\tau_1}^{\tau_2}} D^2 (Y^2 \psi_{m_i} )^2 d\mi_{g_{\cala}} \Rightarrow $$ $$ II_{E1}b  \lesssim \frac{E_0^2 \ee^4}{( 1+\tau_1 )^{1/2}} , $$
where we used Sobolev when needed, \eqref{dtyg}, \eqref{dtyg1} and estimate \eqref{d32tye} for $l=0$.

For $III_{E1}b$ we have that
$$ III_{E1}b = \sum_{k_1 + k_2 \mik 5}\int_{\cala_{\tau_1}^{\tau_2}} D^3 (Y\psi_{k_1} )^2 (TY^2 \psi_{k_2} )^2 d\mi_{g_{\cala}} \lesssim \frac{E_0 \ee^2}{( 1+\tau_1 )^{1/2}} \sum_{i , m_i \mik 5} \int_{\cala_{\tau_1}^{\tau_2}} D^2 (TY^2 \psi_{m_i} )^2 d\mi_{g_{\cala}} \Rightarrow $$ $$III_{E1}b  \lesssim \frac{E_0^2 \ee^4}{(1+\tau_1 )^{1/2}} , $$
where we used Sobolev when needed, \eqref{dyg}, \eqref{dyg1} and estimate \eqref{d32tye} for $l=1$.

For $IV_{E1}b$ we have that
$$ IV_{E1}b = \sum_{k_1 + k_2 \mik 5}\int_{\cala_{\tau_1}^{\tau_2}}  D (T^2 Y \psi_{k_1} )^2 (Y\psi_{k_2} )^2 d\mi_{g_{\cala}} \lesssim  $$ $$ \lesssim  \frac{E_0 \ee^2}{( 1+\tau )^{1/2}} \sum_{i , m_i \mik 5} \int_{\cala_{\tau_1}^{\tau_2}} ( T^2 Y\psi_{m_i} )^2 d\mi_{g_{\cala}} \lesssim \frac{E_0^2 \ee^4}{(1+\tau_1  )^{1/2}}  , 
$$
where we used Sobolev when needed, \eqref{dgtt}, \eqref{dgtt1} and \eqref{d32tye} for $l=1$.

For $V_{E1}b$ we have that
$$ V_{E1}b = \sum_{k_1 + k_2 \mik 5}\int_{\cala_{\tau_1}^{\tau_2}} D (T Y\psi_{k_1} )^2 (T Y\psi_{k_2} )^2 d\mi_{g_{\cala}} \lesssim  $$ $$ \lesssim \sum_{k_1 + k_2 \mik 5} \int_{\cala_{\tau_1}^{\tau_2}} D (T Y\psi_{k_1} )^2 (T Y\psi_{k_2} )^2 d\mi_{g_{\cala}} \lesssim \frac{E_0 \ee^2}{( 1+\tau_1 )^{1/2}} \sum_{i , m_i \mik 5} \int_{\cala_{\tau_1}^{\tau_2}} (T Y\psi_{m_i} )^2 d\mi_{g_S} \lesssim \frac{E_0^2 \ee^4}{(1+\tau_1 )^{1/2}} \Rightarrow $$ $$ \Rightarrow I_{E_1}b \lesssim \frac{E_0^2 \ee^4}{(1+\tau_1 )^{1/2}} , 
$$
where we used Sobolev when needed, \eqref{dtyg}, \eqref{dtyg1} and \eqref{d32tye} for $l=0$.

For $VI_{E1}b$ we have that
$$ VI_{E1}b = \sum_{k_1 + k_2 \mik 5}\int_{\cala_{\tau_1}^{\tau_2}} D (T^2\psi_{k_1} )^2 (Y^2 \psi_{k_2} )^2 d\mi_{g_{\cala}} \Rightarrow $$ $$\sum_{k_1 + k_2 \mik 5} \int_{S_{\tau} \cap \cala_{\tau_1}^{\tau_2}} D (T^2\psi_{k_1} )^2 (Y^2 \psi_{k_2} )^2 d\mi_{g_S} \lesssim \frac{E_0 \ee^2}{(1+\tau)^{1+\beta}} \sum_{i , m_i \mik 5} \int_{\cala_{\tau_1}^{\tau_2}} D (Y^2 \psi_{m_i} )^2 d\mi_{g_{\cala}} \lesssim \frac{E_0^2 \ee^4}{( 1+\tau )^{1+\beta}} , $$
where we used Sobolev when needed, \eqref{dgtt}, \eqref{dgtt1} and estimate \eqref{ydenp2} for $l=0$. This implies that
$$ VI_{E1}b \lesssim E_0^2 \ee^4 . $$

For $VII_{E1}b$ we have that
$$ VII_{E1}b = \sum_{k_1 + k_2 \mik 5}\int_{\cala_{\tau_1}^{\tau_2}} D (T\psi_{k_1} )^2 (TY^2 \psi_{k_2} )^2 d\mi_{g_{\cala}} \Rightarrow $$ $$ \sum_{k_1 + k_2 \mik 5}\int_{S_{\tau} \cap \cala_{\tau_1}^{\tau_2}} D (T\psi_{k_1} )^2 (TY^2 \psi_{k_2} )^2 d\mi_{g_S} \frac{E_0 \ee^2}{( 1+\tau )^{1+\beta}} \sum_{i , m_i \mik 5} \int_{\cala_{\tau_1}^{\tau_2}} D (TY^2 \psi_{m_i} )^2 d\mi_{g_{\cala}} \lesssim \frac{E_0^2 \ee^4}{( 1+\tau )^{1+\beta}} , $$
where we used Sobolev when needed, \eqref{dgt}, \eqref{dgt1} and estimate \eqref{d32tye} for $l=1$. This implies that
$$ VIII_{E1}b \lesssim E_0^2 \ee^4 . $$

For $VIII_{E1}b$ we have that
$$ VIII_{E1}b = \sum_{k_1 + k_2 \mik 5}\int_{\cala_{\tau_1}^{\tau_2}} D |TY \slashed{\nabla} \psi_{k_1} |^2 |\slashed{\nabla} \psi_{k_2} |^2 d\mi_{g_{\cala}} \lesssim $$ $$ \lesssim  \sum_{k_1 + k_2 \mik 5}\int_{\cala_{\tau_1}^{\tau_2}} D |\slashed{\nabla} T\psi_{k_1} |^2 |\slashed{\nabla} \psi_{k_2} |^2 d\mi_{g_{\cala}}  +  \sum_{k_1 + k_2 \mik 5}\int_{\cala_{\tau_1}^{\tau_2}} D | \slashed{\nabla} TY\psi_{k_1} |^2 |\slashed{\nabla} \psi_{k_2} |^2 d\mi_{g_{\cala}} \lesssim $$ $$ \lesssim \frac{E_0 \ee^2}{1+\tau_1} \sum_{i , m_i \mik 5} \int_{\cala_{\tau_1}^{\tau_2}} |\slashed{\nabla} T\psi_{m_i} |^2 d\mi_{g_{\cala}}+\frac{E_0 \ee^2}{1+\tau_1 } \sum_{i , m_i \mik 5} \int_{\cala_{\tau_1}^{\tau_2}} |\slashed{\nabla} \psi_{m_i} |^2 d\mi_{g_{\cala}} + $$ $$+ \frac{E_0 \ee^2}{1 + \tau_1} \sum_{i , m_i \mik 5} \int_{\cala_{\tau_1}^{\tau_2}} D |\slashed{\nabla} TY \psi_{m_i} |^2 d\mi_{g_{\cala}} \lesssim \frac{E_0^2 \ee^4}{1+\tau_1} , $$
where we used Sobolev when needed, \eqref{dg}, \eqref{dg1}, \eqref{dgt}, \eqref{dgt1}, \eqref{d2tyg}, \eqref{d2tyg1}, \eqref{denp1}, \eqref{denp2} and estimate \eqref{ydenp2} for $l=1$.

For $IX_{E1}b$ we have that
$$ IX_{E1}b = \sum_{k_1 + k_2 \mik 5}\int_{\cala_{\tau_1}^{\tau_2}} D |Y \slashed{\nabla} \psi_{k_1} |^2 |T\slashed{\nabla} \psi_{k_2} |^2 d\mi_{g_{\cala}} \lesssim $$ $$ \lesssim  \sum_{k_1 + k_2 \mik 5}\int_{\cala_{\tau_1}^{\tau_2}} D |\slashed{\nabla} \psi_{k_1} |^2 |\slashed{\nabla} T\psi_{k_2} |^2 d\mi_{g_{\cala}}  +  \sum_{k_1 + k_2 \mik 5}\int_{\cala_{\tau_1}^{\tau_2}} D | \slashed{\nabla} Y\psi_{k_1} |^2 |\slashed{\nabla} T\psi_{k_2} |^2 d\mi_{g_{\cala}} \lesssim $$ $$ \lesssim \frac{E_0 \ee^2}{1+\tau_1} \sum_{i , m_i \mik 5} \int_{\cala_{\tau_1}^{\tau_2}} |\slashed{\nabla} \psi_{m_i} |^2 d\mi_{g_{\cala}} + \frac{E_0 \ee^2}{(1+\tau_1 )^{1/2}} \sum_{i , m_i \mik 5} \int_{\cala_{\tau_1}^{\tau_2}} |\slashed{\nabla} T \psi_{m_i} |^2 d\mi_{g_{\cala}} + $$ $$ +\frac{E_0 \ee^2}{1+\tau_1} \sum_{i , m_i \mik 5} \int_{\cala_{\tau_1}^{\tau_2}} D |\slashed{\nabla} Y \psi_{m_i} |^2 d\mi_{g_{\cala}} \lesssim \frac{E_0^2 \ee^4}{(1+\tau_1 )^{1/2}} , $$
where we used Sobolev when needed, \eqref{dg}, \eqref{dg1}, \eqref{dgt}, \eqref{dgt1}, \eqref{d2yg}, \eqref{d2yg1}, \eqref{denp1}, \eqref{denp2} and estimate \eqref{ydenp2} for $l=0$.

Finally we look at $\Omega^k T^2 YF$ and we have that
$$  \sum_{k \mik 5} \int_{\cala_{\tau_1}^{\tau_2}} |\Omega^k T^2 YF |^2 d\mi_{g_{\cala}} \lesssim I_{E1}c + II_{E1}c + III_{E1}c + IV_{E1}c + V_{E1}c + $$ $$ + VI_{E1}c + VII_{E1}c + VIII_{E1}c +IX_{E1}c + X_{E1}c + XI_{E1}c + $$ $$ + XII_{E1}c +XIII_{E1}c  + \frac{E_0^2 \ee^4}{(1+\tau_1)^{2-\aaa}} . $$

For $I_{E1}c$ we note that it can be treated in the same way as $I_{E1}b$, where we use \eqref{denp3} instead of \eqref{denp2}.

For $II_{E1}c$ we note that $II_{E1}c \lesssim II_{E1}b$. For $III_{E1}c$ we have that
\begin{equation*}
\begin{split}
 III_{E1}c =  \sum_{k \mik 5} \int_{\cala_{\tau_1}^{\tau_2}} D^3 & ( T^2 Y\psi_{k_1} )^2 (Y^2 \psi_{k_2} )^2 d\mi_{g_{\cala}} \\ \lesssim & \frac{E_0 \ee^2}{(1+\tau_1 )^{1/2-\aaa/2}} \sum_{i , m_i \mik 5} \int_{\cala_{\tau_1}^{\tau_2}} D^2 (Y^2 \psi_{m_i} )^2 d\mi_{g_{\cala}} \lesssim \frac{E_0^2 \ee^4}{(1+\tau_1 )^{1/2-\aaa/2}} ,
\end{split}
\end{equation*}
where we used Sobolev when needed, \eqref{dttyg}, \eqref{dttyg1} and estimate \eqref{d32tye} for $l=0$.

For $IV_{E1}c$ we have that
$$ IV_{E1}c =  \sum_{k \mik 5} \int_{\cala_{\tau_1}^{\tau_2}} D^3 ( T Y\psi_{k_1} )^2 (TY^2 \psi_{k_2} )^2 d\mi_{g_{\cala}} \lesssim \frac{E_0 \ee^2}{( 1+\tau_1 )^{1/2}} \sum_{i , m_i \mik 5} \int_{\cala_{\tau_1}^{\tau_2}} D^2 (TY^2 \psi_{m_i} )^2 d\mi_{g_{\cala}} \lesssim \frac{E_0^2 \ee^4}{( 1+\tau_1 )^{1/2} } ,$$
where we used Sobolev when needed, \eqref{d2tyg}, \eqref{d2tyg1} and estimate \eqref{d32tye} for $l=1$.

For $V_{E1}c$ we have that
$$ V_{E1}c =  \sum_{k \mik 5} \int_{\cala_{\tau_1}^{\tau_2}} D^3 ( Y\psi_{k_1} )^2 (T^2 Y^2 \psi_{k_2} )^2 d\mi_{g_{\cala}} \lesssim \frac{E_0 \ee^2}{( 1+\tau_1 )^{1/2}} \sum_{i , m_i \mik 5} \int_{\cala_{\tau_1}^{\tau_2}} D^2 (T^2 Y^2 \psi_{m_i} )^2 d\mi_{g_{\cala}} \lesssim \frac{E_0^2 \ee^4}{( 1+\tau_1 )^{1/2} } ,$$
where we used Sobolev when needed, \eqref{d2yg}, \eqref{d2yg1} and estimate \eqref{d32tye} for $l=2$.

For $VI_{E1}c$ we have that
$$ VI_{E1}c =  \sum_{k \mik 5} \int_{\cala_{\tau_1}^{\tau_2}} D ( T^3 Y\psi_{k_1} )^2 (Y \psi_{k_2} )^2 d\mi_{g_{\cala}} \lesssim $$ $$ \lesssim \frac{E_0 \ee^2}{( 1+\tau_1 )^{1/2}} \sum_{i , m_i \mik 5} \int_{\cala_{\tau_1}^{\tau_2}}  (T^3 Y \psi_{m_i} )^2 d\mi_{g_S} \lesssim \frac{E_0^2 \ee^4}{( 1+\tau )^{1/2} } , $$ 
where we used Sobolev when needed, \eqref{dyg}, \eqref{dyg1} and \eqref{d32tye} for $l=2$.

For $VII_{E1}c$ we have that
$$ VII_{E1}c =  \sum_{k_1 + k_2  \mik 5} \int_{\cala_{\tau_1}^{\tau_2}} D ( T^2 Y\psi_{k_1} )^2 (TY \psi_{k_2} )^2 d\mi_{g_{\cala}} \lesssim $$ $$ \lesssim \frac{E_0 \ee^2}{( 1+\tau_1 )^{1/2}} \sum_{i , m_i \mik 5} \int_{\cala_{\tau_1}^{\tau_2}} (T^2 Y \psi_{m_i} )^2 d\mi_{g_S} \lesssim \frac{E_0^2 \ee^4}{( 1+\tau_1 )^{1/2} } \Rightarrow $$ $$ \Rightarrow VII_{E1}c \lesssim \frac{E_0^2 \ee^4}{(1+\tau_1 )^{1/2}}, $$
where we used Sobolev when needed, \eqref{dtyg}, \eqref{dtyg1} and \eqref{d32tye} for $l=1$.

For $VIII_{E1}c$ we have that
$$ VIII_{E1}c =  \sum_{k_1 + k_2 \mik 5} \int_{\cala_{\tau_1}^{\tau_2}} D ( T^3 \psi_{k_1} )^2 (Y^2 \psi_{k_2} )^2 d\mi_{g_{\cala}} \Rightarrow $$ $$ \sum_{k_1 + k_2 \mik 5} \int_{S_{\tau} \cap \cala_{\tau_1}^{\tau_2}} D ( T^3 \psi_{k_1} )^2 (Y^2 \psi_{k_2} )^2 d\mi_{g_S}  \lesssim $$ $$ \lesssim \frac{E_0 \ee^2}{( 1+\tau_1 )^{1 +\beta}} \sum_{i , m_i \mik 5} \int_{\cala_{\tau_1}^{\tau_2}} D ( Y^2 \psi_{m_i} )^2 d\mi_{g_{\cala}} \lesssim \frac{E_0^2 \ee^4}{( 1+\tau )^{1 + \beta} } ,$$
where we used Sobolev when needed, \eqref{dgttt2}, \eqref{dgttt3} and estimate \eqref{ydenp2} for $l=0$. This implies that
$$ VIII_{E1}c \lesssim E_0^2 \ee^4 . $$

For $IX_{E1}c$ we have that
$$ IX_{E1}c =  \sum_{k_1 + k_2 \mik 5} \int_{\cala_{\tau_1}^{\tau_2}} D ( T^2 \psi_{k_1} )^2 (TY^2 \psi_{k_2} )^2 d\mi_{g_{\cala}} \Rightarrow $$ $$  \sum_{k_1 + k_2 \mik 5} \int_{S_{\tau} \cap \cala_{\tau_1}^{\tau_2}} D ( T^2 \psi_{k_1} )^2 (TY^2 \psi_{k_2} )^2 d\mi_{g_S}  \lesssim \frac{E_0 \ee^2}{ (1+\tau )^{1+\beta} } \sum_{i , m_i \mik 5} \int_{\cala_{\tau_1}^{\tau_2}} D ( TY^2 \psi_{m_i} )^2 d\mi_{g_{\cala}} $$ $$ \lesssim \frac{E_0^2 \ee^4}{ (1+\tau_1 )^{1+\beta} } ,$$
where we used Sobolev when needed, \eqref{dgtt}, \eqref{dgtt1} and estimate \eqref{d32tye} for $l=1$. This implies that
$$ IX_{E1}c \lesssim E_0^2 \ee^4 . $$

For $X_{E1}c$ we have that
$$ X_{E1}c =  \sum_{k_1 + k_2 \mik 5} \int_{\cala_{\tau_1}^{\tau_2}} D ( T \psi_{k_1} )^2 (T^2 Y^2 \psi_{k_2} )^2 d\mi_{g_{\cala}} \Rightarrow $$ $$ \sum_{k_1 + k_2 \mik 5} \int_{S_{\tau} \cap \cala_{\tau_1}^{\tau_2}} D ( T \psi_{k_1} )^2 (T^2 Y^2 \psi_{k_2} )^2 d\mi_{g_S}  \lesssim \frac{E_0 \ee^2}{ (1+\tau )^{1+\beta} } \sum_{i , m_i \mik 5} \int_{\cala_{\tau_1}^{\tau_2}} D ( T^2 Y^2 \psi_{m_i} )^2 d\mi_{g_{\cala}} \lesssim \frac{E_0^2 \ee^4}{ ( 1+\tau )^{1+\beta} } ,$$
where we used Sobolev when needed, \eqref{dgt}, \eqref{dgt1} and estimate \eqref{d32tye} for $l=2$. This implies that
$$ X_{E1}c \lesssim E_0^2 \ee^4 . $$

For $XI_{E1}c$ we have that
$$ XI_{E1}c = \sum_{k_1 + k_2  \mik 5} \int_{\cala_{\tau_1}^{\tau_2}} D | T^2 Y \slashed{\nabla} \psi_{k_1} |^2 | \slashed{\nabla}  \psi_{k_2} |^2 d\mi_{g_{\cala}} \lesssim \sum_{k_1 + k_2 \mik 5} \int_{\cala_{\tau_1}^{\tau_2}} D |  \slashed{\nabla} T^2 \psi_{k_1} |^2 | \slashed{\nabla}  \psi_{k_2} |^2 d\mi_{g_{\cala}} + $$ $$ + \sum_{k_1 + k_2 \mik 5} \int_{\cala_{\tau_1}^{\tau_2}} D |  \slashed{\nabla} T^2 Y \psi_{k_1} |^2 | \slashed{\nabla}  \psi_{k_2} |^2 d\mi_{g_{\cala}}  \lesssim \frac{E_0 \ee^2}{( 1+\tau_1 )^{1/2-\aaa/2} } \sum_{i , m_i \mik 5}  \int_{\cala_{\tau_1}^{\tau_2}}  | \slashed{\nabla}  \psi_{m_i} |^2 d\mi_{g_{\cala}} + $$ $$ + \frac{E_0 \ee^2}{1+\tau_1 } \sum_{i , m_i \mik 5}  \int_{\cala_{\tau_1}^{\tau_2}}    | \slashed{\nabla}  T^2 \psi_{m_i} |^2 d\mi_{g_{\cala}}+ \frac{E_0 \ee^2}{1+\tau_1} \sum_{i , m_i \mik 5}  \int_{\cala_{\tau_1}^{\tau_2}}   | \slashed{\nabla}  T^2 Y\psi_{m_i} |^2 d\mi_{g_{\cala}} \lesssim \frac{E_0^2 \ee^4}{(1+\tau_1 )^{1/2-\aaa/2}},$$
where we used Sobolev when needed, \eqref{dg}, \eqref{dg1}, \eqref{dgtt}, \eqref{dgtt1}, \eqref{d2ttyg}, \eqref{d2ttyg1}, \eqref{denp1}, \eqref{denp3}  and estimate \eqref{d32tye} for $l=2$.

For $XII_{E1}c$ we have that
$$ XII_{E1}c = \sum_{k_1 + k_2  \mik 5} \int_{\cala_{\tau_1}^{\tau_2}} D | T Y \slashed{\nabla} \psi_{k_1} |^2 | T\slashed{\nabla}  \psi_{k_2} |^2 d\mi_{g_{\cala}} \lesssim \sum_{k_1 + k_2 \mik 5} \int_{\cala_{\tau_1}^{\tau_2}} D |  \slashed{\nabla} T \psi_{k_1} |^2 | \slashed{\nabla}  T\psi_{k_2} |^2 d\mi_{g_{\cala}} + $$ $$ + \sum_{k_1 + k_2 \mik 5} \int_{\cala_{\tau_1}^{\tau_2}} D |  \slashed{\nabla} T Y \psi_{k_1} |^2 | \slashed{\nabla}  T\psi_{k_2} |^2 d\mi_{g_{\cala}}  \lesssim \frac{E_0 \ee^2}{( 1+\tau_1 )^{1/2} } \sum_{i , m_i \mik 5}  \int_{\cala_{\tau_1}^{\tau_2}}  | \slashed{\nabla}  T\psi_{m_i} |^2 d\mi_{g_{\cala}} + $$ $$ + \frac{E_0 \ee^2}{1+\tau_1} \sum_{i , m_i \mik 5}  \int_{\cala_{\tau_1}^{\tau_2}}   | \slashed{\nabla}  T^2 Y\psi_{m_i} |^2 d\mi_{g_{\cala}} \lesssim \frac{E_0^2 \ee^4}{( 1+\tau_1 )^{1/2}} , $$
where we used Sobolev when needed, \eqref{dgt}, \eqref{dgt1},  \eqref{dtyg}, \eqref{dtyg1}, \eqref{denp2} and estimate \eqref{d32tye} for $l=1$.

Finally, for $XIII_{E1}c$ we have that
$$ XIII_{E1}c = \sum_{k_1 + k_2  \mik 5} \int_{\cala_{\tau_1}^{\tau_2}} D |  Y \slashed{\nabla} \psi_{k_1} |^2 | \slashed{\nabla}  T^2 \psi_{k_2} |^2 d\mi_{g_{\cala}} \lesssim \sum_{k_1 + k_2 \mik 5} \int_{\cala_{\tau_1}^{\tau_2}} D |  \slashed{\nabla}  \psi_{k_1} |^2 | \slashed{\nabla}  T^2 \psi_{k_2} |^2 d\mi_{g_{\cala}} + $$ $$ + \sum_{k_1 + k_2 \mik 5} \int_{\cala_{\tau_1}^{\tau_2}} D |  \slashed{\nabla}  Y \psi_{k_1} |^2 | \slashed{\nabla} T^2 \psi_{k_2} |^2 d\mi_{g_{\cala}}  \lesssim \frac{E_0 \ee^2}{( 1+\tau_1 )^{1/2}} \sum_{i , m_i \mik 5}  \int_{\cala_{\tau_1}^{\tau_2}}  | \slashed{\nabla}  \psi_{m_i} |^2 d\mi_{g_{\cala}} + $$ $$ + \frac{E_0 \ee^2}{1+\tau_1 } \sum_{i , m_i \mik 5}  \int_{\cala_{\tau_1}^{\tau_2}}    | \slashed{\nabla}  T^2 \psi_{m_i} |^2 d\mi_{g_{\cala}}+ \frac{E_0 \ee^2}{1+\tau_1} \sum_{i , m_i \mik 5}  \int_{\cala_{\tau_1}^{\tau_2}} D^2  | \slashed{\nabla}  T Y\psi_{m_i} |^2 d\mi_{g_{\cala}} \lesssim \frac{E_0^2 \ee^4}{( 1+\tau_1 )^{1/2} },$$
where we used Sobolev when needed, \eqref{dg}, \eqref{dg1}, \eqref{dgtt}, \eqref{dgtt1}, \eqref{dyg}, \eqref{dyg1}, \eqref{denp1}, \eqref{denp3}  and estimate \eqref{d32tye} for $l=0$.

\eqref{E2'}: This term can be treated similarly with the term \eqref{C4'} by using the elliptic estimates of Appendix \ref{appendix}.

\end{proof}
\subsection{Proof of global well-posedness}
Using now Theorem \ref{bootthm}, we can choose $\epsilon$ such that
$$ \bar{C} E_0^2 \epsilon^4 \mik C E_0 \epsilon ^2 , $$
for $\bar{C}$ being the largest constant that shows up in the estimates of Theorem \ref{bootthm}, and for $C$ given by the assumptions in Section \ref{boot}, in order to close all the bootstrap estimates. Global well-posedness for smooth and compactly supported data of size $\epsilon$ can be proved in a quite standard way (see the relevant sections in \cite{rnnonlin1} and \cite{shiwu}).

\section{Conservation laws and asymptotic instabilities}
We consider the spherically mean of our nonlinear wave $\psi$ which we recall that is denoted by $\psi_0 = \int_{\mathbb{S}^2} \psi \, d\omega$. We have the following two basic facts.
\begin{thm}[\textbf{Conservation law on the event horizon}]
Let $\psi$ be a global solution of \eqref{nw}. Then for its spherical mean $\psi_0$ we have that the following quantity is conserved along the event horizon
\begin{equation}\label{h0_nonlin}
H_0 \doteq \left. \left( \partial_r \psi_0 + \frac{1}{M} \psi_0 \right)  \right|_{\mathcal{H}^{+}} (v) = \left. \left( \partial_r \psi_0 + \frac{1}{M} \psi_0  \right)  \right|_{\mathcal{H}^{+}} (0) \mbox{         $\forall v \meg 0$.} 
\end{equation}
\end{thm}
The proof follows directly from the equation after we evaluate it on $r=M$. We note that the conserved quantity for a nonlinear wave equation of the form \eqref{nw} for the spherically symmetric part of the wave is identical to that we have for a linear wave. The reason for this is that an equation of the form \eqref{nw} is identical to the linear wave equation on the horizon due to the weight $\sqrt{D}$ in front of the nonlinearity, which does not play a role in this situation, as it does for nonlinear wave equation that satisfies the classical null condition (compare with the situation in \cite{}).

On the other hand, higher derivatives in $r$ blow-up asymptotically on $\mathcal{H}^{+}$.

\begin{thm}[\textbf{Asymptotic blow-up on the event horizon}]\label{asym_bu}
Let $\psi$ be a solution of \eqref{nw} that emanates for sufficiently small initial data $(\epsilon f, \epsilon g)$ of size $\epsilon > 0$ and assume that additionally we have that:
$$ \int_{\mathbb{S}^2} f \, d\omega > 0 \mbox{  and  } Y\psi_0 (0,M) > 0 , $$
then for any $k > 2$ we have that
\begin{equation}\label{asyminst}
| Y^k \psi_0 | (v,M) \rightarrow \infty \mbox{  as $v \rightarrow \infty$.}
\end{equation}
\end{thm}
\begin{proof}
We consider first the equation that the spherical mean $\psi_0$ satisfies which is given by
$$ \Box_g \psi_0 = \left( \sqrt{D} \cdot g^{\mu \nu} \partial_{\mu} \psi \partial_{\nu} \psi \right)_0 , $$
where $ \left( \sqrt{D} \cdot g^{\mu \nu} \partial_{\mu} \psi \partial_{\nu} \psi \right)_0$ means that we are considering the spherical mean of the above expression. We differentiate the equation above with respect to $Y$ and after evaluating it on the horizon we have that
\begin{equation}\label{y2}
 TY^2 \psi_0 + \frac{1}{M} TY\psi_0 - \frac{1}{M^2} T\psi_0 + \frac{1}{M^2} Y\psi_0 = \frac{1}{2M} \sum_{(1) , (2)} T\psi^{(1)} \cdot Y\psi^{(2)} + \frac{1}{2M} \sum_{(1) , (2)}\langle \slashed{\nabla} \psi^{(1)} , \slashed{\nabla} \psi^{(2)} \rangle , 
 \end{equation}
where by $\sum_{(1) , (2)}$ and the related superscripts on $\psi$ we mean that we add over the appropriate angular frequency localizations that were introduced by taking the spherical mean of the nonlinearity, where the localizations are at the 0-th angular frequency and at the rest (so the sum is finite). Integrating the last equation \eqref{y2} along $\mathcal{H}^{+}$ we have that
$$ Y^2 \psi_0 (v,M) \mik Y^2 \psi_0 (0,M) + \frac{1}{M} Y\psi_0 (0,M) - \frac{1}{M} H_0 - \frac{1}{M^2} \psi_0 (0,M) + \frac{1}{M^2} \psi_0 (v,M) - \frac{1}{2M^2} \int_0^v Y\psi_0 + $$ $$ +\int_0^v \left( \frac{1}{2M} \sum_{(1) , (2)} T\psi^{(1)} \cdot Y\psi^{(2)} + \frac{1}{2M} \sum_{(1) , (2)}\langle \slashed{\nabla} \psi^{(1)} , \slashed{\nabla} \psi^{(2)} \rangle \right). $$
For the last term we have the estimate
$$ \left| \int_0^v \left( \frac{1}{2M} \sum_{(1) , (2)} T\psi^{(1)} \cdot Y\psi^{(2)} + \frac{1}{2M} \sum_{(1) , (2)}\langle \slashed{\nabla} \psi^{(1)} , \slashed{\nabla} \psi^{(2)} \rangle \right) \right| \mik C E_0 \epsilon^2 \log (1+v) . $$
Using the decay of $\psi$, it is easy to show that
$$ \frac{1}{2M^2} \int_0^v Y\psi_0 \meg C H_0 v , $$
where we recall that $H_0 > 0$ by our assumption. Using now the last two estimates, we have that
$$  Y^2 \psi_0 (v,M) \mik Y^2 \psi_0 (0,M) + \frac{1}{M} Y\psi_0 (0,M) - \frac{1}{M} H_0 - C H_0 v + C \frac{\sqrt{E_0} \epsilon}{v^{1/2}} +  C E_0 \epsilon^2 \log (1+v) \mik - c H_0 v , $$
if $v$ is large enough, for some constant $c$, which gives us the desired result. For $k > 3$ we can argue in a similar way.

\end{proof}
\begin{rem}
The positivity condition on the data in the last Theorem \ref{asym_bu} was important for obtaining asymptotic blow-up along $\mathcal{H}^{+}$. If $H_0 = 0$, then it is expected that $Y^2 \psi_0$ remains bounded, while we have asymptotic blow-up along $\mathcal{H}^{+}$ for all quantities $Y^k \psi_0$, $k > 3$. This was shown for the linear case in \cite{aretakis2013}, and it should hold also in for the nonlinear waves of this paper.
\end{rem}

\section{Acknowledgments}

We would like to thank Georgios Moschidis for several useful discussion concerning this work. The second author (S.A.) acknowledges support through NSF grant DMS-1600643 and a Sloan Research Fellowship. The third author (D.G.) acknowledges support by the European Research Council grant no. ERC-2011-StG 279363-HiDGR.

\appendix
\section{Elliptic estimates}\label{appendix}
The following elliptic-type estimates are considered to be quite standard. We have that for $\psi$ a solution of $\Box_g \psi = F$ the following holds true
$$ \int_{\si_{\tau} \cap \{ r\meg r_0 > M \}} (\partial_a \partial_b \psi )^2 d\mi_{g_{\si}} \lesssim  \int_{\si_{\tau} \cap \{ r\meg r_0 \}} J^T_{\mu} [\psi ] n^{\mu} d\mi_{g_{\si}} + \int_{\si_{\tau} \cap \{ r\meg r_0 \}} J^T_{\mu} [T\psi ] n^{\mu} d\mi_{g_{\si}} + \int_{\si_{\tau} \cap \{ r\meg r_0 \}} |F|^2 d\mi_{g_{\si}} , $$
for any fixed $r_0 > M$, and any $\partial_a , \partial_b \in \{ \partial_v , \partial_r , \partial_{\theta}, \partial_{\sigma} \}.$

Department of Mathematics, University of California, Los Angeles, CA, USA, 90095, \\
\textit{email}: yannis@math.ucla.edu

\bigskip

Department of Mathematics, Princeton University, Princeton, NJ, USA, 08544, \\
\textit{email}: aretakis@math.princeton.edu

\bigskip

Department of Mathematics, University of Toronto Scarborough 1265 Military Trail, Toronto, ON, M1C 1A4, Canada, \\ \textit{email}: aretakis@math.toronto.edu

\bigskip

Department of Mathematics, University of Toronto, 40 St George Street, Toronto, ON, Canada, \\
\textit{email}: aretakis@math.toronto.edu

\bigskip

Department of Mathematics, Imperial College London, London, United Kingdom, SW7 2AZ, \\
\textit{email}: dejan.gajic@imperial.ac.uk

\bigskip

Department of Applied Mathematics and Theoretical Physics, University of Cambridge, Wilberforce Road, Cambridge CB3 0WA, United Kingdom, \\
\textit{email}: dg405@cam.ac.uk

\begin{thebibliography}{100}

\bibitem{rnnonlinthesis}
{\sc Angelopoulos, Y.}
\newblock Nonlinear waves on extremal black hole spacetimes.
\newblock {\em Ph{D} {T}hesis, {U}niversity of {T}oronto\/} (2015).

\bibitem{rnnonlin1}
{\sc Angelopoulos, Y.}
\newblock Global spherically symmetric solutions of non-linear wave equations
  with null condition on extremal {R}eissner-{N}ordstr\"{o}m spacetimes.
\newblock {\em {I}nt. {M}at. {R}es. {N}otices (11)\/} (2016), 3279--3355.

\bibitem{improvedrn}
{\sc Angelopoulos, Y., Aretakis, S., and Gajic, D.}
\newblock Late-time asymptotics for solutions to the wave equation on extremal
  {R}eissner--{N}ordstr\"{o}m.
\newblock {\em {I}n preparation\/}.

\bibitem{rnnonlin3}
{\sc Angelopoulos, Y., Aretakis, S., and Gajic, D.}
\newblock Nonlinear wave equations with null condition on extremal
  {R}eissner-{N}ordstr\"{o}m spacetimes {II}: {T}he general case.
\newblock {\em {I}n preparation\/}.

\bibitem{trapping}
{\sc Angelopoulos, Y., Aretakis, S., and Gajic, D.}
\newblock The trapping effect on degenerate horizons.
\newblock {\em To appear in Ann. Henri Poincar\'{e}, arXiv:1512.09094\/}
  (2015).

\bibitem{A1}
{\sc Aretakis, S.}
\newblock Stability and instability of extreme {R}eissner--{N}ordstr\"om black
  hole spacetimes for linear scalar perturbations {I}.
\newblock {\em Commun. Math. Phys. 307\/} (2011), 17--63.

\bibitem{A2}
{\sc Aretakis, S.}
\newblock Stability and instability of extreme {R}eissner--{N}ordstr\"om black
  hole spacetimes for linear scalar perturbations {II}.
\newblock {\em Ann. Henri Poincar\'{e} 12\/} (2011), 1491--1538.

\bibitem{aretakis3}
{\sc Aretakis, S.}
\newblock Decay of axisymmetric solutions of the wave equation on extreme
  {K}err backgrounds.
\newblock {\em J. Funct. Analysis 263\/} (2012), 2770--2831.

\bibitem{aretakisglue}
{\sc Aretakis, S.}
\newblock The characteristic gluing problem and conservation laws for the wave
  equation on null hypersurfaces.
\newblock {\em arXiv:1310.1365\/} (2013).

\bibitem{aretakis2012}
{\sc Aretakis, S.}
\newblock A note on instabilities of extremal black holes from afar.
\newblock {\em Classical and Quantum Gravity 30}, 9 (2013), 095010.

\bibitem{aretakis2013}
{\sc Aretakis, S.}
\newblock On a non-linear instability of extremal black holes.
\newblock {\em Phys. Rev. D 87\/} (2013), 084052.

\bibitem{aretakis4}
{\sc Aretakis, S.}
\newblock Horizon instability of extremal black holes.
\newblock {\em {A}dv. {T}heor. {M}ath. {P}hys. 19\/} (2015), 507--530.

\bibitem{bizon2012}
{\sc Bizo\'{n}, P., and Friedrich, H.}
\newblock A remark about the wave equations on the extreme
  {R}eissner--{N}ordstr\"om black hole exterior.
\newblock {\em Class. Quantum Grav. 30\/} (2013), 065001.

\bibitem{bizon2016}
{\sc Bizo\'{n}, P., and Kahl, M.}
\newblock A {Y}ang-{M}ills field on the extremal {R}eissner-{N}ordstrˆm black
  hole.
\newblock {\em ar{X}iv:1603.04795\/} (2016).

\bibitem{couch}
{\sc Couch, W., and Torrence, R.}
\newblock Conformal invariance under spatial inversion of extreme
  {R}eissner-{N}ordstr\"{o}m black holes.
\newblock {\em Gen. Rel. Grav. 16\/} (1984), 789--792.

\bibitem{newmethod}
{\sc Dafermos, M., and Rodnianski, I.}
\newblock A new physical-space approach to decay for the wave equation with
  applications to black hole spacetimes.
\newblock {\em XVIth International Congress on Mathematical Physics\/} (2010),
  421--432.

\bibitem{gajic2}
{\sc Gajic, D.}
\newblock Linear waves in the interior of extremal black holes {II}.
\newblock {\em arXiv:1512.08953\/} (2015).

\bibitem{gajic}
{\sc Gajic, D.}
\newblock {Linear waves in the interior of extremal black holes {I}}.
\newblock {\em Communications in Mathematical Physics {Online first}\/} (2016).
\newblock \url{http://link.springer.com/article/10.1007/s00220-016-2800-y}.

\bibitem{sergiunull}
{\sc Klainerman, S.}
\newblock The null condition and global existence to nonlinear wave equations.
\newblock {\em Lectures in App. Math.\/} (1986), 293--326.

\bibitem{lmrtprice}
{\sc Lucietti, J., Murata, K., Reall, H.~S., and Tanahashi, N.}
\newblock On the horizon instability of an extreme {R}eissner--{N}ordstr\"om
  black hole.
\newblock {\em JHEP 1303\/} (2013), 035, arXiv:1212.2557.

\bibitem{hj2012}
{\sc Lucietti, J., and Reall, H.}
\newblock Gravitational instability of an extreme {K}err black hole.
\newblock {\em Phys. Rev. D86:104030\/} (2012).

\bibitem{luknullcondition}
{\sc Luk, J.}
\newblock The null condition and global existence for nonlinear wave equations
  on slowly rotating kerr spacetimes.
\newblock {\em Journal Eur. Math. Soc. 15(5)\/} (2013), 1629--1700.

\bibitem{murata2012}
{\sc Murata, K.}
\newblock Instability of higher dimensional extreme black holes.
\newblock {\em Class. Quantum Grav. 30\/} (2013), 075002.

\bibitem{harvey2013}
{\sc Murata, K., Reall, H.~S., and Tanahashi, N.}
\newblock What happens at the horizon(s) of an extreme black hole?
\newblock {\em Classical and Quantum Gravity 30}, 23, 235007.

\bibitem{ori2013}
{\sc Ori, A.}
\newblock Late-time tails in extremal {R}eissner-{N}ordstr\"{o}m spacetime.
\newblock {\em arXiv:1305.1564\/} (2013).

\bibitem{janpaper}
{\sc Sbierski, J.}
\newblock Characterisation of the energy of {G}aussian beams on {L}orentzian
  manifolds with applications to black hole spacetimes.
\newblock {\em Analysis and PDE 8, No. 6\/} (2015), 1379--1420.

\bibitem{sela}
{\sc Sela, O.}
\newblock Late-time decay of perturbations outside extremal charged black hole.
\newblock {\em Phys. Rev. D 93\/} (2016), 024054.

\bibitem{shiwu}
{\sc Yang, S.}
\newblock Global solutions of nonlinear wave equations in time dependent
  inhomogeneous media.
\newblock {\em Arch. Ration. Mech. Anal. 209(2)\/} (2013), 683--728.

\end{thebibliography}
\end{document}